\documentclass[10pt]{amsart}

\usepackage{amsmath,amssymb,latexsym,cite,enumerate}
\usepackage{graphicx,overpic,tikz,color}
\usepackage[colorlinks=true,urlcolor=blue,
citecolor=red,linkcolor=blue,linktocpage,pdfpagelabels,bookmarksnumbered,
bookmarksopen]{hyperref}
\usepackage[english]{babel}

\usepackage[left=2.80cm,right=2.80cm,top=3cm,bottom=3cm]{geometry}

\numberwithin{equation}{section}
\newtheorem{theorem}{Theorem}[section]
\newtheorem{proposition}[theorem]{Proposition}

\newtheorem{algoInt}[theorem]{Algorithm}
\theoremstyle{definition}

\theoremstyle{remark}
\newtheorem{remark}[theorem]{Remark}

\newcommand{\R}{{\mathbb R}}
\renewcommand{\H}{{\mathcal{H}}}

\newcommand{\E}{{\mathcal{E}}}

\newcommand{\C}{\mathcal{C}}
\newcommand{\abs}[1]{\mathopen\vert#1\mathclose\vert}
\newcommand{\norm}[1]{\mathopen\Vert#1\mathclose\Vert}
\newcommand{\intd}{\,{\mathrm d}}

\newcommand{\eps}{\varepsilon}

\newcommand{\iu}{{\rm i}}
\newcommand{\adh}[1]{\operatorname{adh}(#1)}
\newcommand{\wto}{\rightharpoonup}

\DeclareMathOperator{\Ran}{Im}
\DeclareMathOperator{\dist}{dist}
\DeclareMathOperator{\e}{e}

\title[Mountain-pass algorithm for the quasi-linear Schr\"odinger equation]{On the
Mountain-pass algorithm for\\ the quasi-linear Schr\"odinger equation}

\author[C.\ Grumiau]{Christopher Grumiau}
\author[M.\ Squassina]{Marco Squassina}
\author[C.\ Troestler]{Christophe Troestler}

\address{Universite de Mons \newline\indent
Institut de Mathematique \newline\indent
Service d'Analyse Num\'erique \newline\indent
Place du Parc, 20 \newline\indent
B-7000 Mons, Belgique}
\email{Christopher.Grumiau@umons.ac.be}
\email{Christophe.Troestler@umons.ac.be}

\address{Dipartimento di Informatica
\newline\indent
Universit\`a degli Studi di Verona
\newline\indent
C\'a Vignal 2, Strada Le Grazie 15, I-37134 Verona, Italy}
\email{marco.squassina@univr.it}

\thanks{The second author was supported by 2009 national MIUR project:
   ``Variational and Topological
Methods in the Study of Nonlinear Phenomena''. The first and third author are partially supported by the program
``Qualitative study of
 solutions of variational elliptic partial differerential equations. Symmetries,
bifurcations, singularities, multiplicity and numerics'' of the FNRS, project
2.4.550.10.F of the Fonds de la Recherche Fondamentale Collective, Belgium.}

\date{\today}

\begin{document}

\begin{abstract}
  We discuss the application of the Mountain Pass Algorithm to the
  so-called quasi-linear Schr\"odinger equation, 
  which is naturally associated with a class of nonsmooth functionals so
  that the classical algorithm cannot directly be used.
  A change of variable allows us to deal with the lack of regularity.
  We establish the convergence of
  a mountain pass algorithm in this setting.
  Some numerical experiments are also performed and lead to
  some conjectures.
\end{abstract}

\maketitle

\section{Introduction and results}

The aim of this paper is to find numerical solutions of mountain pass
type for the problem 
\begin{equation}
  \label{gs1}
  -\Delta u-u\Delta u^2+ V u =|u|^{p-1}u,\qquad  \text{in $\R^N$},
\end{equation}
for $p>1$ and $V > 0$. This is the equation of standing waves of the
quasi-linear Schr\"odinger equation 
\begin{equation}
   \label{eq.schr1}
 \begin{cases}
 \iu\phi_t+\Delta\phi+\phi\Delta |\phi|^2+|\phi|^{p-1}\phi=0 & \text{in $(0,\infty)\times \R^N$},\\
 \phi(0,x)=\phi_0(x) & \text{in $\R^N$},
 \end{cases}
\end{equation}
where $\iu$ stands for the imaginary unit and the unknown $\phi:(0,\infty)\times \R^N\to {\mathbb C}$ is a complex valued
function. Various physically relevant situations are described by this quasi-linear equation. For example, it is used in plasma physics and
fluid mechanics, in the theory of Heisenberg ferromagnets and magnons and in condensed
matter theory, see e.g.\ the bibliography of~\cite{cojeansqu} and the references therein. The mountain pass solutions
of the semi-linear equation $-\Delta u+V u =|u|^{p-1}u$, corresponding to ground states for the Schr\"odinger equation
\begin{equation}
   \label{eq.schrr}
 \begin{cases}
 \iu\phi_t+\Delta\phi+|\phi|^{p-1}\phi=0 & \text{in $(0,\infty)\times \R^3$},\\
 \phi(0,x)=\phi_0(x) & \text{in $\R^3$},
 \end{cases}
\end{equation}
have been extensively studied in the last decades, both analytically \cite{willem} and numerically. On the numerical side, the typical
tool is the so-called mountain pass algorithm, originally implemented by Choi and McKenna \cite{mpa1} (see also \cite{calsqu9,calsqu11} for a different approach).
This works nicely under the assumption that the functional associated with the problem
is of class $C^1$ on a Hilbert space and it satisfies suitable geometrical assumptions. Now,
we observe that the functional ${\mathcal E}:X\to\bar\R$ naturally but only formally associated with
\eqref{gs1} is
\begin{equation}
  \label{fction}
  u\mapsto \mathcal{E}(u) :=
  \frac{1}{2}\int_{\R^N} (1+2 u^2)|\nabla u|^2 \intd x
  + \frac{1}{2}\int_{\R^N} V u^2 \intd x
  - \frac{1}{p+1}\int_{\R^N} |u|^{p+1} \intd x.
\end{equation}
In fact, taking $X$ as $H^1(\R^N)$ with $N\geq 3$, then the functional \eqref{fction} is not even well defined, as it might be the case that it
assumes the value $+\infty-(+\infty)$ in the range
$(N+2)/(N-2)<p<(3N+2)/(N-2)$. 
In two dimensions, with $X = H^1(\R^2)$, it is well defined from $X$ to $\R\cup\{+\infty\}$ but it is merely
lower semi-continuous. If, instead, $X$ stands for the set of $u\in H^1(\R^N)$ such that $u^2\in H^1(\R^N)$,
then it follows that \eqref{fction} is well defined from $X$ to $\R$, for every choice $1<p<(3N+2)/(N-2)$. From the physical viewpoint
defining $X$ in this way makes it a natural choice for the energy space. On the other hand, for initial data in this space $X$, currently there is
no well-posedness result for problem~\eqref{eq.schr1} (see the discussion in \cite{cojeansqu}). Furthermore,
$X$ is not even a vector space, although $(X,d)$ can be regarded as a complete metric space endowed with
distance the $d(u, v)=\|u-v\|_{H^1} + \|\nabla u^2- \nabla v^2\|_{L^2}$ and it turns out that \eqref{fction} is continuous
over $(X,d)$ suggesting that a possible approach to the study of
problem \eqref{gs1} could be to exploit the (metric) critical point theory
developed in~\cite{corvdegmar}. Nevertheless, this continuity property
is not enough to  establish the convergence of a traditional
Mountain Pass Algorithm.  Actually, it is not even clear
what a satisfying gradient for $\mathcal{E}$ is.
Let us emphasize that we are interested in the convergence of the
algorithm in infinite dimensional spaces which
ensures the convergence of the discretized problem at a rate which
does not blow-up 
when the mesh used for approximation becomes finer. Hence, in    
conclusion, neither the mountain pass algorithm can be directly 
applied for the numerical computation of some solution of \eqref{gs1}
nor, to our knowledge, the current literature (see e.g.~\cite{Lewis-Pang11}
and the references therein)
contains suitable generalization of the mountain pass algorithm to the case of non-smooth functionals, except the case of
locally Lipschitz functional, which unfortunately are incompatible with the regularity available for our framework. On the other hand
in \cite{cojean,cojeansqu}, in order to find the existence and qualitative properties of the solutions to \eqref{gs1}, a
change of variable procedure was performed to relate the solutions $u\in X$ to \eqref{gs1} with the solutions $v\in H^1(\R^N)$ to
an associated semi-linear problem
\begin{eqnarray}
  \label{exis1}
  -\Delta v= f(x, v),\qquad
  f(x, v) := \frac{|r(v)|^{p-1}r(v)- V(x)\, r(v)}{\sqrt{1+2r^2(v)}},
\end{eqnarray}
where the function $r:\R\to\R$ is the unique solution to the Cauchy problem
\begin{equation}
  \label{changevariable-intro}
  r'(s)=\frac{1}{\sqrt{1+2r^2(s)}},\qquad r(0)=0.
\end{equation}
More precisely, $u$ is a smooth solution (say $C^2$) to \eqref{gs1} if
and only if $v=r^{-1}(u)$ is a smooth solution to 
\eqref{exis1}, that is a critical point of the 
$C^1$-functional ${\mathcal T}: H^1(\R^N) \to \R$ defined by
\begin{equation}
  \label{eq:defT}
  v\mapsto \mathcal{T}(v) := \frac{1}{2} \int_{\R^N} |\nabla v|^2\intd x
  - \frac{1}{p+1}\int_{\R^N} |r(v)|^{p+1}\intd x
  + \frac{1}{2}\int_{\R^N} V(x)\, |r(v)|^2\intd x.
\end{equation}
In addition, as we shall see, the mountain pass values and the least
energy values of these functionals 
correspond through the function $r$. Now, by applying the mountain
pass algorithm to ${\mathcal T}$, we can 
find a mountain pass solution $v$ of \eqref{exis1}. Then $u=r(v)$ will
be a mountain pass solution 
of the original problem with a reasonable control on numerical
errors.

We mention that, at least in the case of large constant potentials $V$
and for a general class of quasi-linear problems,
there are some  uniqueness results for the solutions, see
e.g.~\cite{gladiali}.

The article is organized as follows.  In section~\ref{sec2}, we will
establish the equivalence between ground state solutions for
$\mathcal{E}$ and those for $\mathcal{T}$.  Section~\ref{sec:conv-MP}
will recall the mountain pass algorithm and discuss assumptions under
which its convergence is guaranteed for our problem. In section~\ref{num},
we will present our numerical investigations. We finish by a
conclusion giving some conjectures and outlining future work.

\medskip\noindent %
\textit{Notation \& terminology.} We will denote $\Ran f$ the image of
a function $f : A \to B$ i.e., $\Ran f = \{ f(x) : x \in A\}$.
The notation $\adh A$ stands for the topological closure of the
set~$A$.  The norm in the Lebesgue spaces $L^p(\Omega)$ will be written
$\abs{\cdot}_{L^p(\Omega)}$ ($\Omega$ will be dropped if clear from
the context).  We will call a (nonlinear) \emph{projector} any
function that is idempotent i.e., any function $f$ such that
$f(x) = x$ for all $x \in \Ran f$.

\section{Equivalence between Mountain Pass values}
\label{sec2}

We noticed in the introduction that, thanks to the change of unknown
$u = r(v)$, solutions $v \in H^1(\R^N) \cap C^2(\R^N)$
to~\eqref{exis1} correspond to solutions $u \in X\cap
C^2(\R^N)$ to~\eqref{gs1}
where $X=\{u\in H^1(\R^N):\, u^2\in H^1(\R^N)\}$ (see e.g.\
\cite{cojean,cojeansqu}).
In this section,
we want to show that, in addition, if $v$ is at the Mountain Pass
level, then $u$ is at the Mountain Pass level too
in a suitable functional setting, see formula \eqref{leveldes}. Hence, numerically
computing a Mountain Pass solution $v$ up to a certain error, yields a
Mountain Pass solution $u=r\circ v$ to the original
problem~\eqref{gs1} up to a certain error (involving also the error
due to the numerical calculation of the solution $r$ to the Cauchy
problem in \eqref{changevariable-intro}).  In order to prove this,
notice that ${\mathcal T}: H^1(\R^N) \to \R$ defined
by~\eqref{eq:defT} reads
\begin{align*}
  {\mathcal T}(v)
  &=\frac{1}{2} \int_{\R^N} \abs{\nabla v}^2 \intd x
  - \int_{\R^N} F(x, v) \intd x
\end{align*}
where we have set $F(x, t) := \int_0^t f(x, s) \intd s$ with $f$
defined by~\eqref{exis1}.  Notice first that, if
$u=r(v)$ where $v\in H^1(\R^N)\cap C^2(\R^N)$, then $u \in X\cap C^2(\R^N)$.
Furthermore, it follows that
\begin{equation}
  \label{energybal}
  {\mathcal E}(u)= {\mathcal T}(v),
\end{equation}
where $\mathcal{E}$ is the action defined by~\eqref{fction}.
Indeed, we have
\begin{align*}
  {\mathcal E}\bigl(r(v)\bigr)
  &= \frac{1}{2}\int_{\R^N}
  \bigl(1 + 2r^2(v)\bigr) (r')^{2}(v) \abs{\nabla v}^2 \intd x
  -\frac{1}{p+1}\int_{\R^N} \abs{r(v)}^{p+1} \intd x
  +\frac{1}{2}\int_{\R^N} V(x) \abs{r(v)}^2 \intd x\\
  &= \frac{1}{2}\int_{\R^N} \abs{\nabla v}^2 \intd x
  - \frac{1}{p+1}\int_{\R^N} \abs{r(v)}^{p+1} \intd x
  + \frac{1}{2}\int_{\R^N} V(x)\, \abs{r(v)}^2 \intd x\\
  &= {\mathcal T}(v),
\end{align*}
thanks to the Cauchy problem~\eqref{changevariable-intro}.

\begin{proposition}
  \label{correspondence MP}
  Let $\hat v\in H^1(\R^N)\cap C^2(\R^N)$ be a Mountain Pass
  solution to problem~\eqref{exis1}, that is
  \begin{equation}
    \label{eq:MP-T}
    {\mathcal T}(\hat v)=
    \mathop{\inf\vphantom{\sup}}_{\gamma\in\Gamma_{\mathcal T}\,}
    \sup_{t\in [0,1]}{\mathcal T}(\gamma(t)),\qquad
    \Gamma_{\mathcal T}=\big\{\gamma\in C([0,1],H^1(\R^N)):
    \gamma(0)=0,\ {\mathcal T}(\gamma(1))<0\big\}.
  \end{equation}
  Then $\hat u=r (\hat v)\in X\cap C^2(\R^N)$ is a Mountain Pass solution
  to problem~\eqref{gs1}, that is
  \begin{equation}
    \label{leveldes}
    {\mathcal E}(\hat u)=
    \mathop{\inf\vphantom{\sup}}_{\eta\in\Gamma_{\mathcal E}\,}
    \sup_{t\in [0,1]}{\mathcal E}(\eta(t)),\qquad
    \Gamma_{\mathcal E}=\big\{\eta\in C([0,1],X):
    \gamma(0)=0,\ {\mathcal E}(\gamma(1))<0\big\}.
  \end{equation}
  Furthermore, $\hat u$ is also a least energy solution  to~\eqref{gs1}
  for the energy ${\mathcal E}$
  (i.e., achieving the infimum of $\mathcal{E}$ on non-trivial
  solutions to~\eqref{gs1}).
\end{proposition}
\begin{proof}
It is readily seen that if $\hat v\in H^1(\R^N)\cap C^2(\R^N)$ is a solution to \eqref{exis1},
then $\hat u=r (\hat v)\in X\cap C^2(\R^N)$ is a solution to \eqref{gs1}.
Setting now $\tilde \Gamma_{\mathcal T}:=\{r\circ\gamma :
\gamma\in \Gamma_{\mathcal T}\}$,
on account of \eqref{energybal}, we have
\begin{equation*}
  {\mathcal E}(\hat u)={\mathcal T}(\hat v)
  = \mathop{\inf\vphantom{\sup}}_{\gamma\in\Gamma_{\mathcal T}\,}
  \sup_{t\in [0,1]}{\mathcal T}(\gamma(t))
  = \mathop{\inf\vphantom{\sup}}_{\tilde\gamma\in\tilde\Gamma_{\mathcal T}\,}
  \sup_{t\in [0,1]}{\mathcal E}(\tilde\gamma(t)).
\end{equation*}
On the other hand, we will show that $\tilde\Gamma_{\mathcal
  T}=\Gamma_{\mathcal E}$, yielding assertion~\eqref{leveldes}.  In
fact, let $\eta\in \Gamma_{\mathcal E}$ and consider
$\gamma:=r^{-1}\circ\eta$. Then, $\eta=r\circ\gamma$ with $\gamma\in
C\bigl([0,1],H^1(\R^N)\bigr)$, $\gamma(0)=0$ and ${\mathcal
  T}(\gamma(1))={\mathcal E}(r\circ\gamma(1))={\mathcal
  E}(\eta(1))<0$. Hence $\eta\in \tilde\Gamma_{\mathcal T}$. Vice
versa, if $\eta\in \tilde\Gamma_{\mathcal T}$, there exists $\gamma\in
\Gamma_{\mathcal T}$ such that $\eta=r\circ\gamma$. Then, it is
readily seen that $\eta\in C([0,1],X)$, $\eta(0)=r(\gamma(0))=r(0)=0$
and ${\mathcal E}(\eta(1))={\mathcal T}(\gamma(1))<0$.
By Theorem~0.2 of~\cite{Jeanjean-Tanaka}, $\hat v$ is a least energy
solution.  In turn, by Step~II in the proof of Theorem~1.3 of~\cite{cojean}
the last assertion follows.
\end{proof}

\begin{remark}
  As we will see in the next section, in our case $f(v)/v$ (where $f$
  is defined by~\eqref{exis1}) is increasing on $(0, +\infty)$ and
  therefore mountain pass solutions $\hat v$ can be characterized by
  \begin{equation*}
    \mathcal{T}(\hat v) =
    \inf_{v \ne 0} \sup_{t \ge 0} \mathcal{T}(tv)
  \end{equation*}
  instead of~\eqref{eq:MP-T}.  The numerical algorithm described below
  finds a local minimum $\tilde v$ of $v \mapsto \sup_{t \ge 0}
  \mathcal{T}(tv)$ but there is no absolute guarantee that $\tilde v$
  is a mountain pass solution.  Nonetheless, $\tilde v$ is a ``saddle
  point of mountain pass type''~\cite[Definition~1.2]{Lewis-Pang11}
  i.e., there exists an open neighborhood $V$ of $\tilde v$ such that
  $\tilde v$ lies in the closure of two path-connected components of
  $\{ v \in V : \mathcal{T}(v) < \mathcal{T}(\tilde v) \}$.  Since the
  map $v \mapsto r(v) : H^1(\R^N) \to X$ is an homeomorphism and
  $\mathcal{E}(u) = \mathcal{T}\bigl(r^{-1}(u)\bigr)$, the same
  characterization holds for $\tilde u := r(\tilde v)$.  In
  conclusion, the above discussion can be thought as an extension of
  the correspondence of proposition~\ref{correspondence MP} to saddle
  point solutions.
\end{remark}

\section{Mountain pass algorithms
  --- convergence up to a subsequence}
\label{sec:conv-MP}

Let $\H$ be a Hilbert space with
norm
$\norm{\cdot}$, $E$ a closed subspace of $\H$,
$\mathcal{T}:\H\to\R$
a $\C^1$-functional and $P$ a peak selection for $\mathcal{T}$
relative to $E$, i.e., a function
\begin{equation*}
  P: \H\setminus E\to \H: u\mapsto P(u)
\end{equation*}
such that, for any $u\in \H\setminus E$,   $P(u)$ is a local
maximum point of
$\mathcal{T}$ on $\R^+ u \oplus E = \{ tu + e : t \ge 0$ and $e \in E\}$.
The Mountain Pass Algorithm (MPA) uses $P$ to perform a constrained
steepest descent search in order to find critical points of $\mathcal{T}$.
The version we use in this paper is a slightly modified version of the
MPA introduced by
Y.~Li and
J.~Zhou~\cite{zhou1} which in turn is based on the pioneer work of
Y.~S.~Choi and P.~J.~McKenna~\cite{mpa1}.  Let us first recall the
version of Y.~Li and J.~Zhou.

\begin{algoInt}[MPA~\cite{zhou1}]
  \label{mpa}
  \begin{enumerate}
  \item Choose an initial guess $u_0\in\Ran (P)$, a tolerance
    $\varepsilon >0$ and let $n\gets 0$;
  \item if $\norm{\nabla \mathcal{T}(u_n)}\leq \varepsilon $ then  stop; \\
    otherwise,  compute
    \begin{equation}
      \label{eq:step-Zhou}
      u_{n+1}=P\left(\frac{u_n - s_n \nabla \mathcal{T}(u_n)}{
          \norm{u_n - s_n \nabla \mathcal{T}(u_n)}}\right),
    \end{equation}
    for some $s_n\in (0, +\infty)$ satisfying the Armijo's type
    stepsize condition
    \begin{equation*}
      \mathcal{T}(u_{n+1})-\mathcal{T}(u_n)
      < -\frac{1}{2}\norm{\nabla \mathcal{T}(u_n)}\norm{u_{n+1}-u_n};
    \end{equation*}
  \item let $n\gets n+1$ and go to step $2$.
  \end{enumerate}
\end{algoInt}

Conditions under which the sequence $(u_n)$ generated by the MPA
converges, up to a subsequence, to a critical point of $\mathcal{T}$
on $\Ran P$ have been studied~\cite{zhou2} in the
case $\dim E < \infty$.

In~\cite{tacheny}, N.~Tacheny and C.~Troestler proved that,
if an additional metric projection on the
cone $K$ of non-negative functions is performed at each step of the algorithm,
the MPA still converges (at least up 
to a subsequence) and that the limit of the generated sequence is
guaranteed to be a
non-negative solution.  In that paper, the authors studied
a variant of the above algorithm were the update at each step is
given by
\begin{equation}
  \label{eq:stepsize}
  u_{n+1} = P \bigl(P_K(u_n[s_n])\bigr)
  \quad\text{where }
  u_n[s] := u_n - s \frac{\nabla \mathcal{T}(u_n)}{
    \norm{\nabla \mathcal{T}(u_n)}},
\end{equation}
for some $s_n \in S(u_n) \subseteq (0,+\infty)$
with $P_K$ being the metric projector on the cone $K$.
Notice that, with this formulation, the projector $P$ only needs to be
defined on the cone $K$.
The set $S(u_n)$ of
admissible stepsizes is defined as follows: first let
\begin{equation*}
  S^{*}(u_n) :=
  \bigl\{ s>0 :\hspace{0.5em}
  u_n[s]\ne 0 \hspace{0.5em}\text{and}\hspace{0.5em}
  \mathcal{T}\bigl(PP_K(u_n[s])\bigr) - \mathcal{T}(u_n)
  < -\tfrac{1}{2}s \norm{\nabla \mathcal{T}(u_n)} \bigr\}
\end{equation*}
and then define $S(u_n):= S^{*}(u_n)\cap \left(\frac{1}{2}\sup
  S^{*}(u_n), +\infty \right)$.  Note that the right hand side of the
inequality to satisfy does not depend on~$P$.  Under the following 
assumptions on the  action functional 
$\mathcal{T}$ and the peak selection $P$:
\begin{enumerate}
\item[(P1)] $P : K\setminus\{0\} \to K\setminus\{0\}$ is well defined
  and continuous;
\item[(P2)] $\inf_{u \in \Ran(P) \cap K} \mathcal{T}(u) >-\infty$;
\item[(P3)] $0 \notin  \adh{\Ran P \cap K}$;
\item[(P4)] $\mathcal{T}(P_K(u)) \le \mathcal{T}(u) + o(\dist(u,K))$
  as $\dist(u,K) \to 0$;
\end{enumerate}
they prove that the algorithm generates a Palais-Smale sequence in~$K$. 
 Roughly, using
a numerical deformation lemma, they show that a step 
$s_n \in S(u_n)$ exists and that it can be
chosen in a ``locally uniform'' way.  The trick is to avoid $s_n$
being arbitrarily close to $0$ without being ``mandated'' by the
functional.  Let us remark that the definition of $S(u_n)$
is natural (but not the only possible one)
to force this stepsize not to be ``too small''.  Then, under
some additional    
compactness condition (e.g.\ the Palais-Smale condition or a
concentration compactness result), they establish the convergence up to a
subsequence.  If furthermore the solution is isolated in some sense,
the convergence of the whole sequence is proved.

Let us also mention that, recently, inspired from the
theoretical existence result of A.~Szulkin and T.~Weth~\cite{szulkin},
the convergence of the sequence generated by the mountain pass
algorithm has been  proved 
(in some cases, up to a subsequence) even when $\dim E = \infty$ and
$\Omega =\R^N$ (see~\cite{grumtroe}). 

Let us come back to our problem and verify that
assumptions~(P1)--(P4) hold for the functional $\mathcal{T}$
given by~\eqref{eq:defT} and the projection $P$ 
defined as follows: $P(u)$ is the unique maximum point of
$\mathcal{T}$ on the half line $\{ tu : t>0 \}$.

The autonomous case $-\Delta u - u \Delta u^2 = g(u)$
was studied by M.~Colin and
L.~Jeanjean~\cite{cojean} who proved the existence of a positive
solution under the assumptions
\begin{enumerate} 
\item[(g0)] $g$ is locally H\"older continuous on $[0, +\infty)$;
\item[(g1)] $\lim_{u\to 0} \frac{g(u)}{u} <0$;
\item[(g2)] $\lim_{u\to +\infty} \abs{g(u)}/u^{(3N+2)/(N-2)} =0$ when
  $N\ge 3$ or,\\
  for any $\alpha >0$, there exists $C_\alpha >0$ such that, for all
  $u \ge 0$,
  $\abs{g(u)}\le C_\alpha \e^{\alpha u^2}$ when $N=2$;
\item[(g3)] $\exists \psi >0$ such that $G(\psi)>0$ where $G(u) :=
  \int_{0}^u g(t)\intd t$ when $N\ge 2$ (or $G(\psi)=0$,
  $g(\psi)>0$ and $G(u)<0$ and $u\in (0,\psi)$ when
  $N=1$).
\end{enumerate}
In our case, this situation corresponds to $V\neq 0$ constant and
$g(u) = |u|^{p-1}u - Vu$ where assumptions (g0)--(g3) are
verified when $p \in
\bigl(3, \frac{3N+2}{N-2}\bigr)$.

Let us now work with the equivalent problem $-\Delta v = r'(v)
g(r(v))$ where $r$ is defined as in the introduction and  $g$ verifies
(g0)--(g3). Let us call
$f(v) := r'(v)g(r(v))$. Recalling that
$r(v) \sim v$ 
when $v \to 0$ and $r(v) \sim\sqrt{v}$ when $v \to
+\infty$~\cite{cojean}, (g1) implies that $f$ verifies  $\lim_{v\to
  0} \frac{f(v)}{v}<0$. Moreover, as a consequence of~(g2), $f$ is
subcritical with respect to the critical Sobolev exponent 
$\frac{N+2}{N-2}$.  Therefore $\mathcal{T}$ is well defined on
$H^1(\R^N)$ and $0$ is a local minimum of $\mathcal{T}$.  This
establishes (P2)--(P3).  Because of~(g3) there exists a $v^*$ such
that $\mathcal{T}(v^*) < 0$ but, in order to have (P1) and (P4),
we need to replace (g3) with the following stronger
assumptions (see~\cite{tacheny}):
\begin{enumerate}
\item[(g4)] $F(v)/v^2 \to +\infty$  when $v \to +\infty$
  where $F(v) := \int_0^v f(s) \intd s$;
\item[(g5)]  $v\mapsto f(v)/v$ is
  increasing on $(0,+\infty)$.
\end{enumerate}

Let us remark that when the domain  is
bounded (as in numerical experiments), it is enough to require
$\lim_{u\to 0} \frac{f(u)}{u} \le 0$
instead of $\lim_{u\to 0} \frac{f(u)}{u} < 0$. So we can work with $V=0$ 
in this case.

For our problem~\eqref{exis1} with $V\neq 0$ constant, it is easy
to see that~(g4) is satisfied. 
Let us now show that property~(g5) holds
for $p \ge 3$.  First, let us  remark that the derivative of
$v \mapsto r^k(v) / \bigl(v \sqrt{1 + 2r^2(v)}\bigr)$ is a positive
function times
\begin{equation}
  \label{eq:(f/v)'}
  \bigl(k + 2(k-1) r^2\bigr) (r' \cdot v)
  - \bigl((1 + 2r^2) \cdot r \bigr) .
\end{equation}
For $k = 1$, this quantity becomes $r' v - (1 + 2r^2)r$ which is
negative because $r'(v) \cdot v \le r$ (see
e.g.~\cite[Lemma~2.2]{cojean}).  Thus the map $v
\mapsto - V r(v) / \bigl(v \sqrt{1 + 2r^2(v)}\bigr)$ is
increasing.
As a consequence, it remains to prove that $f(v)/v$ is increasing when
$V=0$.  For this, it is enough to show that \eqref{eq:(f/v)'} is
positive for $k = p$.  Using the inequality $r'(v) \cdot v \ge
\tfrac{1}{2} r$ (see e.g.~\cite[Lemma~2.2]{cojean}), one readily
proves the assertion.

The above arguments imply that the mountain pass algorithm applied
to~$\mathcal{T}$ generates a Palais-Smale sequence
when $3<p< \frac{3N+2}{N-2}$.  
Numerically, it is natural to ``approximate'' the entire space 
by large bounded  domains $\Omega_R$ that are symmetric around
$0$. In the numerical experiments, we 
will consider $\Omega_R =(-R/2,R/2)^N$.  On  
$\Omega_R$, the Palais-Smale  condition holds and consequently the MPA
converges up to a subsequence.  This approximation is
reasonable for two reasons. 
First, the solution $v(x)$ on $\R^N$ goes exponentially fast to
$0$ when $\abs{x} \to \infty$.   Second, if we
consider a family of solutions 
$(v_R)$ on $\Omega_R$  which is bounded and stays away from
$0$ and we extend $u_R$ by $0$ on $\R^N \setminus \Omega_R$, 
we will now sketch an argument showing
that $(v_R)$ converges up to a subsequence to a non-trivial solution on
$\R^N$ (see also~\cite{DingNi} where authors prove that, for some
semilinear elliptic equations $-\Delta v + V v = f(x,v)$, ground
state solutions on large domains weakly converge to a solution
on $\R^N$). 
The boundedness ensures that, taking if necessary a subsequence,
$v_R \wto v$.
At this point, it may well be that $v = 0$.
Nevertheless, 
E.~Lieb~\cite{lieb} proved that if  a 
family of functions  $(v_R)$ is bounded away from zero in $H^1(\R^N)$
 then there exists at least one family $(x_R)\subseteq\R^N$ such
that  $v_R(\cdot - x_R)\wto v^*\ne 0$ up to a subsequence.
To conclude that $v^*\ne 0$, it suffices for example to pick up $x_R$
so that $\int_{B(x_R, 1)} \abs{v_R}^2 \intd x \not\to 0$.
Intuitively, the role of $x_R$ is to bring back  
(some of) the mass that $v_R$ may loose at infinity.  
It is then classical to prove
that $v^*$ is a non-trivial solution on $\R^N$.
If moreover $(v_R)$ is a family of ground states, no mass can be lost
at infinity and so $v_R(\cdot - x_R) \to v^*$.
In this case, $v^*$ is a ground state of the problem on $\R^N$.
Of course,
if $(x_R)$ is bounded in $\R^N$ then $v_R(\cdot - \tilde{x}_R)$ weakly
converges to  a translation of $v^*$ for any bounded family
$(\tilde{x}_R)$. In particular, for $\tilde x_R = 0$,
$v_R$ weakly converges, up to a
subsequence, to a non-trivial solution~$v$.    By regularity, $v_R\to v$ in
$C^1_{\text{loc}}(\R^N)$ up to a subsequence.
  
As the moving plane method can be applied
to the autonomous case (see~\cite{PucciSerrin}),
the positive solutions $v_R$ are even w.r.t.\
each hyperplane $x_i = 0$ and are decreasing in each direction $x_i$
from $x_i=0$ to $x_i = \pm R$.
As $(v_R)$ is bounded away from $0$, the mass of $v_R$
is located around~$0$.  Thus $v_R$
converges up to a subsequence to a non-zero positive solution
of~\eqref{exis1} on
$\R^N$.  Moreover, it is expected that the ground state solutions are
unique, up to a 
translation, which would imply the convergence of the whole family
$(v_R)$.  This   uniqueness result has been proved for $V$
sufficiently large (see~\cite{gladiali}).  

The above reasoning can be applied to the sequence generated by the MPA.
In Section~\ref{num}, the numerical experiments will provide bounded
families $(v_R)$ of positive (approximate) solutions
on $\Omega_R$ which are bounded away from 
zero (when $V \ne 0$). Thus they converge to a solution on $\R^N$ as
$R \to \infty$.

\medbreak

For the non-autonomous case ($V$ non-constant),  M.~Colin and
L.~Jeanjean~\cite{cojean} proved the existence of a positive solution for the
equation $ -\Delta u-u\Delta u^2+ V(x) u = g(u)$  on $\R^N$ when
\begin{enumerate}
\item[(V1)] there exists $V_0>0$ such that $V(x)>V_0$ on $\R^N$;
\item[(V2)] $\lim_{\abs{x}\to +\infty} V(x) =: V_\infty \in \R$ and
  $V(x)\leq V_\infty$ on $\R^N$;
\item[(g1$'$)] $g(0)=0$ and $g$ is continuous;
\item[(g2$'$)] $\lim_{u\to 0}
  \frac{g(u)}{u} =0$;
\item[(g3$'$)] there exists $p<\infty$ when $N=1,2$
  or $p< (3N+2)/(N-2)$ if $N\geq 3$
  and  $C>0$ such that $\abs{g(u)}\leq C (1+  \abs{u}^p)$ for any $u\in\R$. 
\item[(g4$'$)] $\exists
  \mu > 4$ such that $0 < \mu G(u) \leq g(u)u$ for
  any $u>0$  where
  $G(u) := \int_{0}^u g(t)\intd t$.
\end{enumerate}
These assumptions are clearly satisfied for our model nonlinearity
$g(u) = \abs{u}^{p-1}u$ with $3 < p <
\frac{3N+2}{N-2}$.
As before, to prove that $\mathcal{T}$ possesses the properties
(P1)--(P4), we need to replace (g4$'$) with a slightly stronger
assumption, namely that $v\mapsto f(x,v)/v$ is
increasing on $(0,+\infty)$ for almost every $x \in \R^N$ where $f$ is
defined by~\eqref{exis1}.  Essentially repeating the arguments that we
developed for the autonomous case, one can show
that this assumption is statisfied
for our model problem.  As a consequence, the MPA generates a
Palais-Smale sequence.  Since the Palais-Smale condition holds
for the ground state level
(even on $\R^N$ when $V$ is non-constant), the MPA sequence converges up
to a subsequence to a solution of~\eqref{exis1}.

Concerning the approximation of $\R^N$ by large domains, we can argue
similarly to the autonomous case. Numerically (see
Section~\ref{num:non-constant V}),
$(v_R)$ is bounded away from zero and the peaks of
$v_R$ are located around local minimums of $V$. So, the entire mass of $v_R$ is
not going to infinity.  Thus, as in the autonomous case, $v_R\to v$ in
$C^1_{\text{loc}}$ with $v$ being a non-trivial solution on $\R^N$.

\section{Numerical experiments}
\label{num}

In this section, we 
compute ground state solutions to problem~\eqref{gs1}
using algorithm~\ref{mpa} with the update step~\eqref{eq:stepsize}
(instead of~\eqref{eq:step-Zhou}) on problem~\eqref{exis1}.
The algorithm relies at each step on the finite element method.  Let us
remark that approximations are saddle points of the functional but,
as the algorithm is a constrained steepest descent method, it is not
guaranteed that they are gound state solutions.  Never\-the\-less, no
non-trivial solutions with lower energy have been found numerically.

Let us now give more details on the computation of various
objects intervening in the procedure.  As we already
motivated above, the numerical algorithm will seek solutions $v$
to~\eqref{exis1} on a ``large'' domain $\Omega_R := (-R/2,R/2)^2$ with
zero Dirichlet boundary conditions instead of the whole space~$\R^2$.
Functions of $H^1_0(\Omega_R)$ will be approximated using
$P^1$-finite elements on a Delaunay triangulation of $\Omega_R$
generated by Triangle~\cite{Triangle}.  The matrix of the quadratic
form $(v_1, v_2) \mapsto \int_{\Omega_R} \nabla v_1 \nabla v_2$ is
readily evaluated on the finite elements basis.  A quadratic
integration formula on each triangle is used to compute
\begin{math}
  v \mapsto 
  - \frac{1}{p+1}\int_{\Omega_R} |r(v)|^{p+1}\intd x
  + \frac{1}{2}\int_{\Omega_R} V(x)\, |r(v)|^2\intd x  
\end{math}. %
The function $r$ is approximated using a
standard adaptive ODE solver.
The gradient $g := \nabla\mathcal{T}(v)$ is computed in the usual way:
the function $g \in H^1_0(\Omega_R)$ is the solution of the equation
$\forall \varphi \in H^1_0$, $(g | \varphi)_{H^1_0} = \mathrm{d}
\mathcal{T}(v)[\varphi]$ or, equivalently, $g$ is the solution of the
following linear equation
\begin{equation*}
  \begin{cases}
    -\Delta g = -\Delta v - f(x, v)
    &\text{in } \Omega_R,\\
    g = 0& \text{on } \partial\Omega_R,
  \end{cases}
\end{equation*}
where $f$ is defined in~\eqref{exis1}.  This equation is solved using
the finite element method.

Since in our case $E = \{0\}$, the projector $P(u)$ is easily computed
for any $u \ne 0$: first we pick a $t_1 > 0$ large enough so that
$\mathcal{T}(t_1u) \le 0$ and then we use Brent's method to obtain the
point at which $t \mapsto \mathcal{T}(tu)$ achieves its maximum in
$[0, t_1]$ (alternatively, one could seek $t$ such that
$\mathrm{d}\mathcal{T}(tu)[u] = 0$).
The stepsize $s_n$ of~\eqref{eq:stepsize} is determined as follows: we
use Brent's method to compute a value of $s$ minimizing
$[0,1] \to \R : s \mapsto \mathcal{T}\bigl( P(P_K(u_n[s])) \bigr)$.
This choice guarantees that no arbitrary small steps are taken unless
required by the functional geometry.

The starting function for the MPA is always $(x_1,x_2) \mapsto \bigl(0.25 -
(x_1/R)^2\bigr) \bigl(0.25 - (x_2/R)^2\bigr)$.
The program stops when the gradient
of the energy functional at the approximation has a norm less than
$10^{-2}$.  An simple adaptive mesh refinement is performed during the
MPA iterations in order to increase accuracy while keeping the cost
reasonable.  The approximate solution is then further improved using a
few iterations of Newton's method.

\subsection{Zero potential}

To start, we study the problem
\begin{equation}
  \label{exo1u}
  -\Delta u-u\Delta u^2 =|u|^{p-1}u
\end{equation}
in the open bounded domain $(-0.5, 0.5)^2$.  Running the mountain pass
algorithm with $p = 4$ and $R = 1$, we get the results presented in
Fig.~\ref{fig:square,V=0,p=4} and Table~\ref{tableValues,V=0,p=4}.
These experiments suggest that the solutions go to $0$ as $R \to
\infty$, that is when balls become larger.
This is not surprising.
Indeed, for $N\ge 2$, the classical Poh\v ozaev identity for the semi-linear
problem $-\Delta v= f(v)$ on $\R^N$ reads
\begin{equation}
  \label{poho}
  \int_{\R^N}\Bigl[\frac{N-2}{2}f(v)v - NF(v) \Bigr] \intd x=0
\end{equation}
where $F$ is the primitive of $f$ with $F(0)=0$. In our case, 
$f$ is given by eq.~\eqref{exis1} with $V = 0$ and thus
$F(v) = {\abs{r(v)}^{p+1}}/({p+1})$.
Hence for instance for $N=2$, the condition~\eqref{poho} becomes
$$
\int_{\R^2} \abs{r(v)}^{p+1}=0.
$$
Thus $r(v)=0$ or, equivalently, $v=0$.

We also repeat the experiments with the larger exponent $p = 6$ i.e.,
we consider the problem:
\begin{equation}
  \label{exo2u}
  -\Delta u-u\Delta u^2 =|u|^{5}u.
\end{equation}
The same conclusions can be drawn in this case.

\begin{figure}[h!t]
  \vspace*{-6ex}%
  \begin{tikzpicture}[x=0.25\linewidth, y=0.18\linewidth]
    \node at (0,0) {%
      \includegraphics[width=0.25\linewidth]{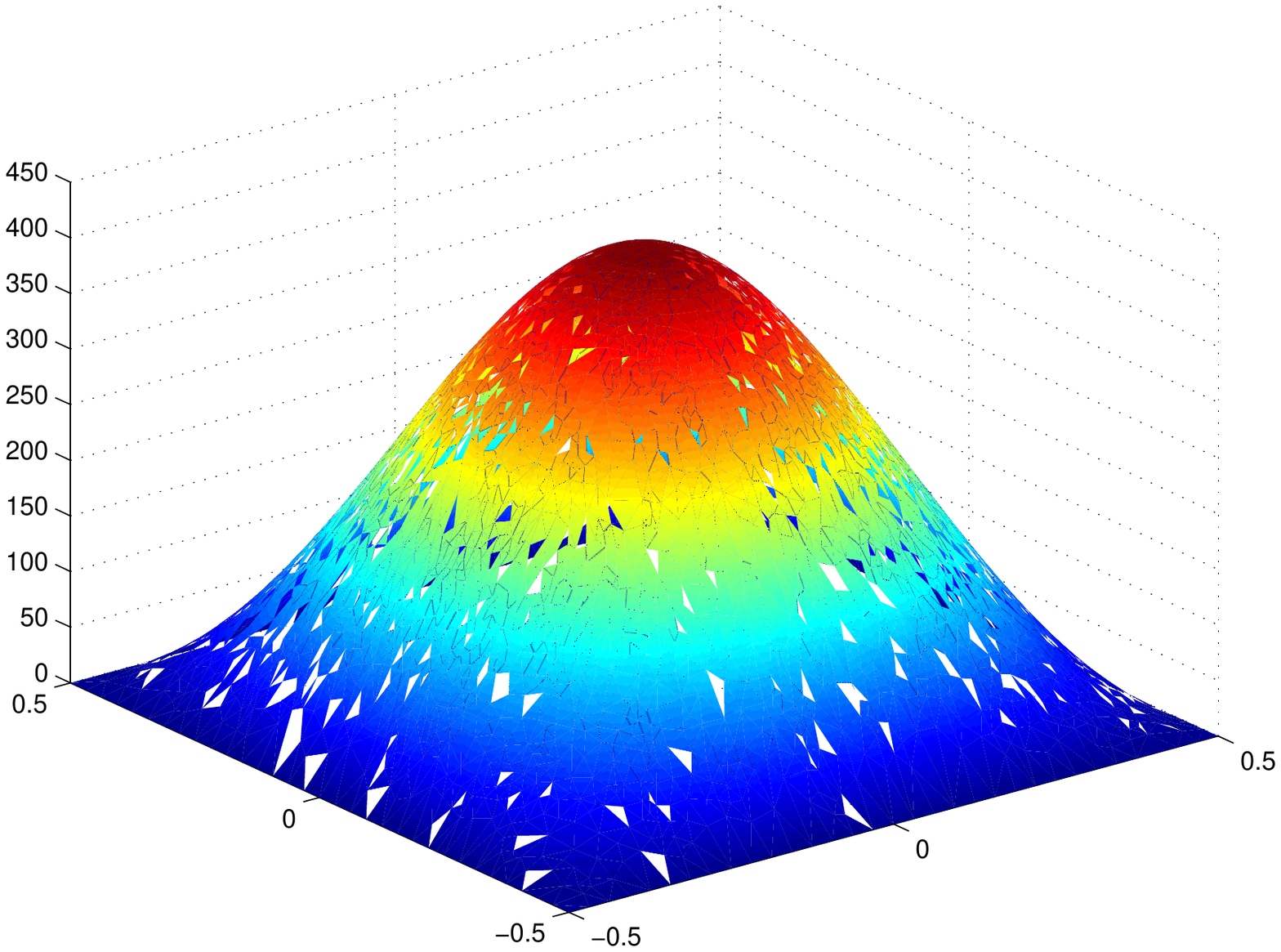}};
    \node at (0,-1) {%
      \includegraphics[width=0.25\linewidth]{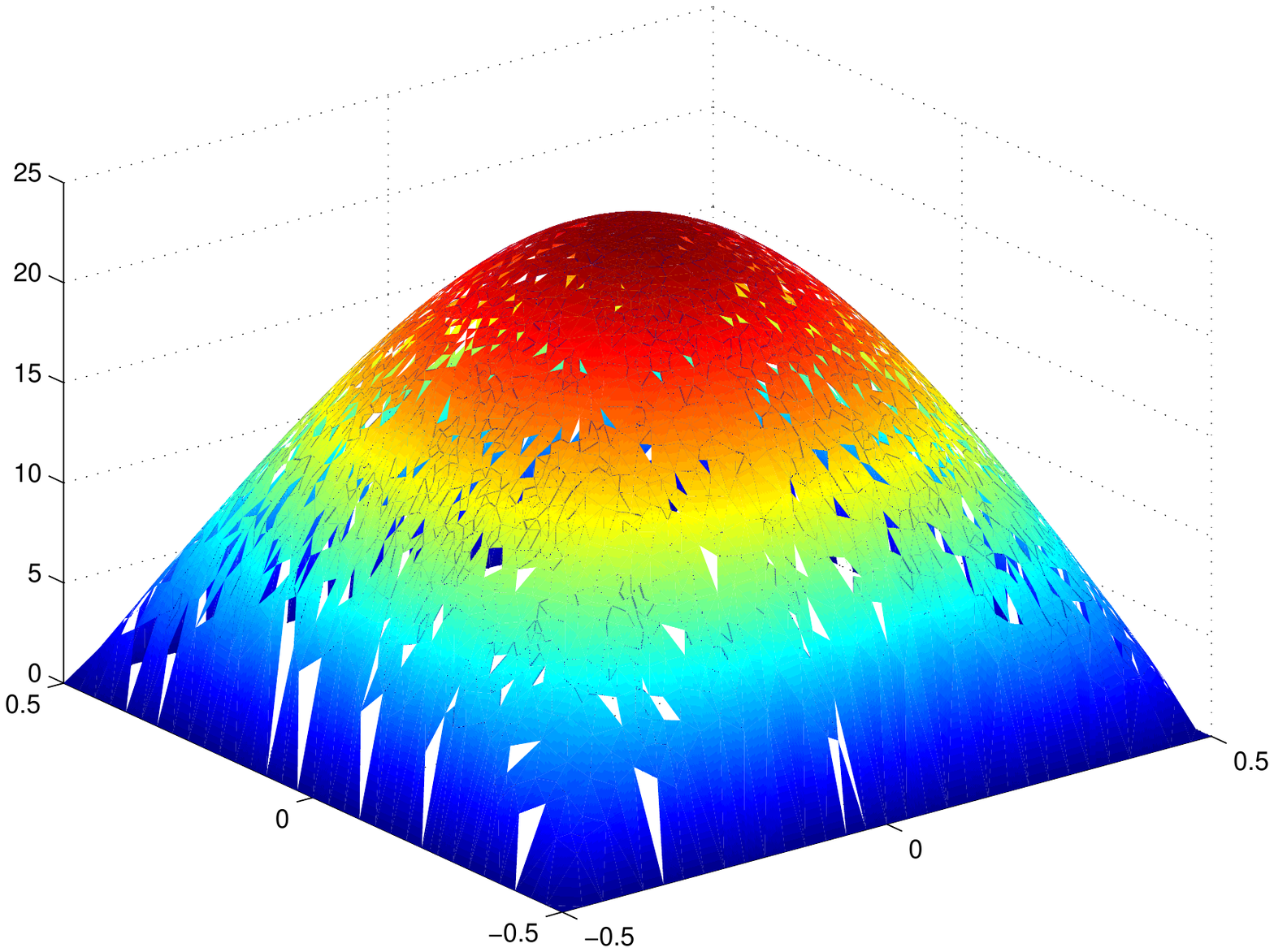}};
    \node at (0,0.5) {$R = 1$};
    \node[left] at (-0.45, 0) {$v$};
    \node[left] at (-0.45, -1) {\rotatebox{90}{$u = r(v)$}};

    \node at (1,0) {%
      \includegraphics[width=0.25\linewidth]{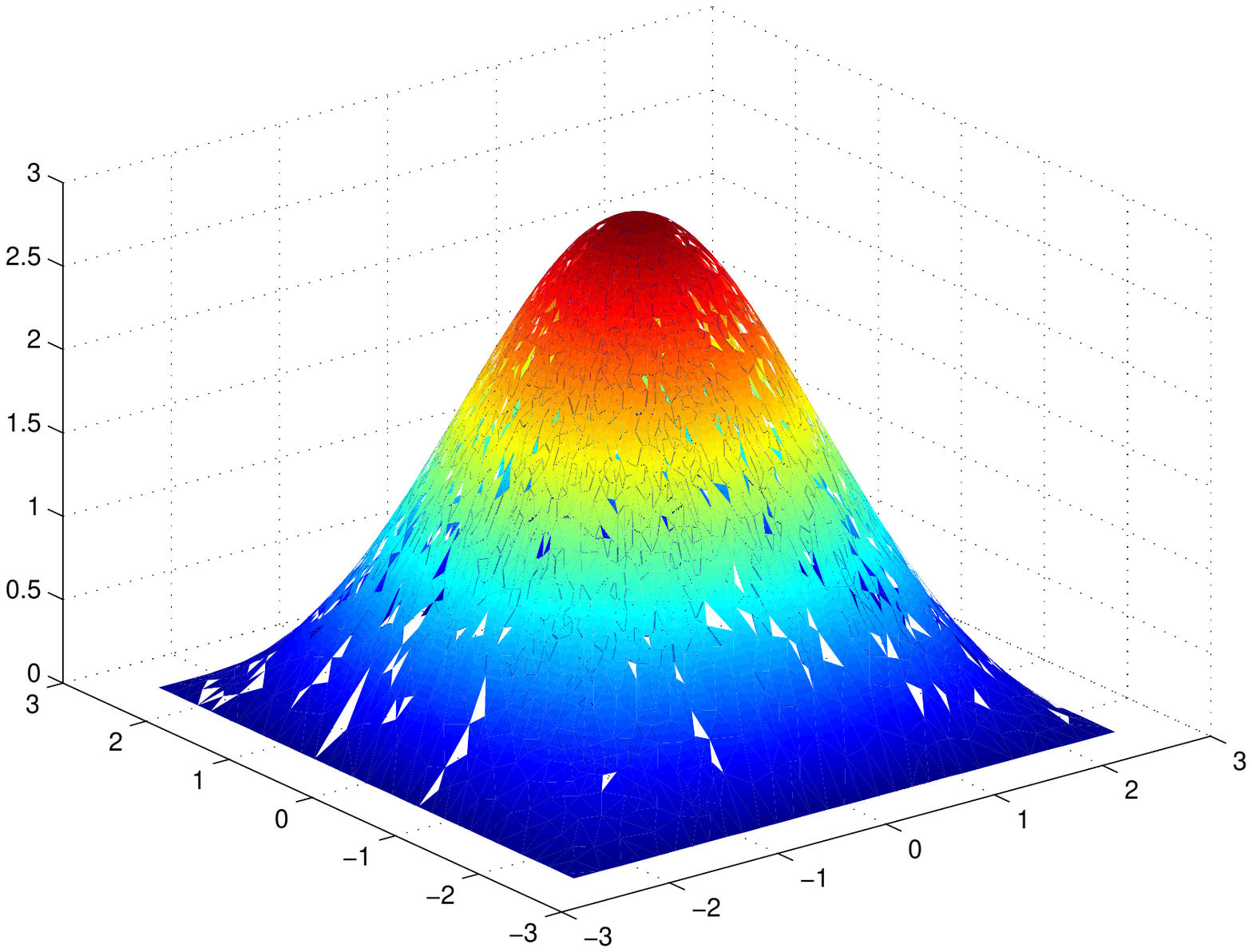}};
    \node at (1,-1) {%
      \includegraphics[width=0.25\linewidth]{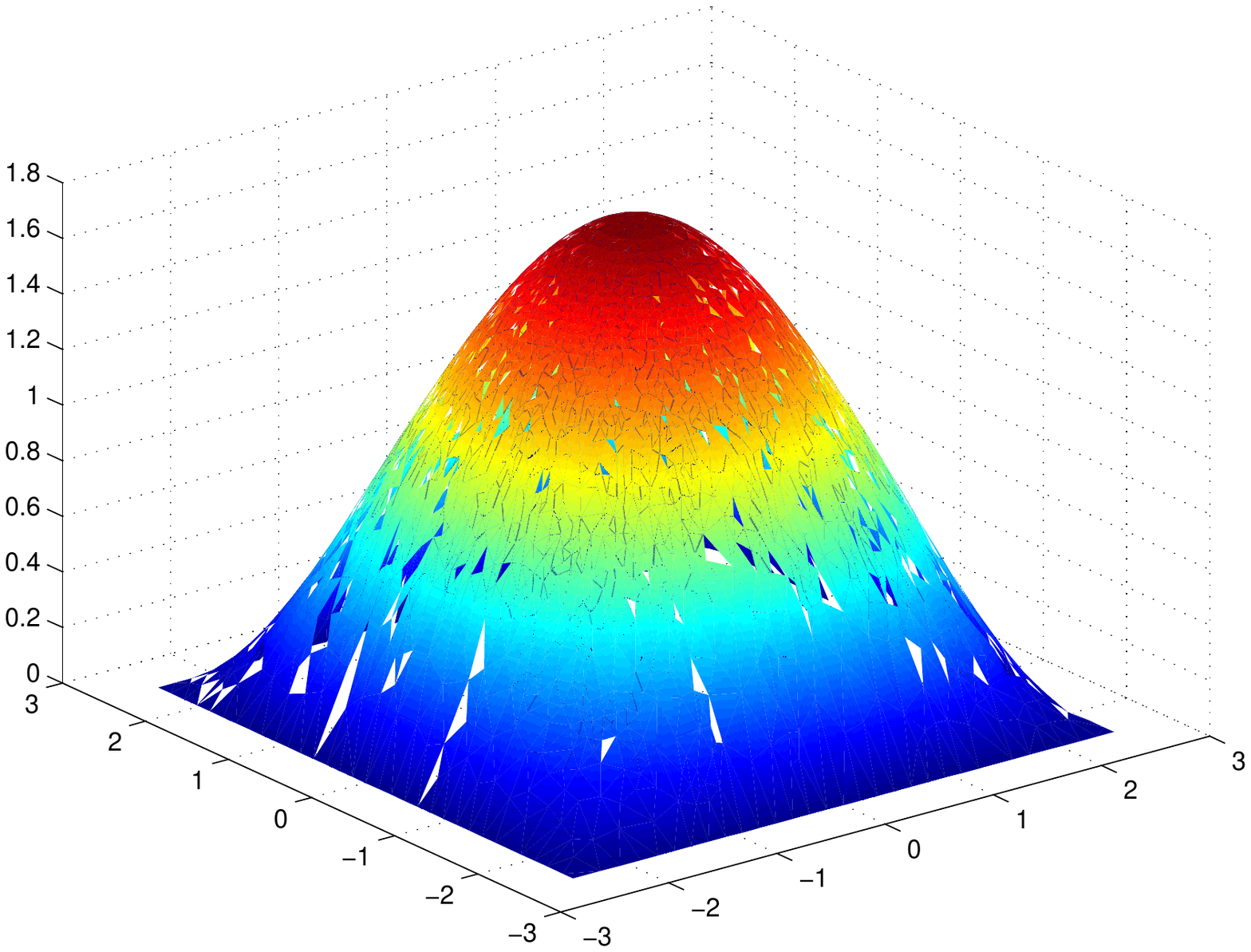}};
    \node at (1,0.5) {$R = 5$};

    \node at (2,0) {%
      \includegraphics[width=0.25\linewidth]{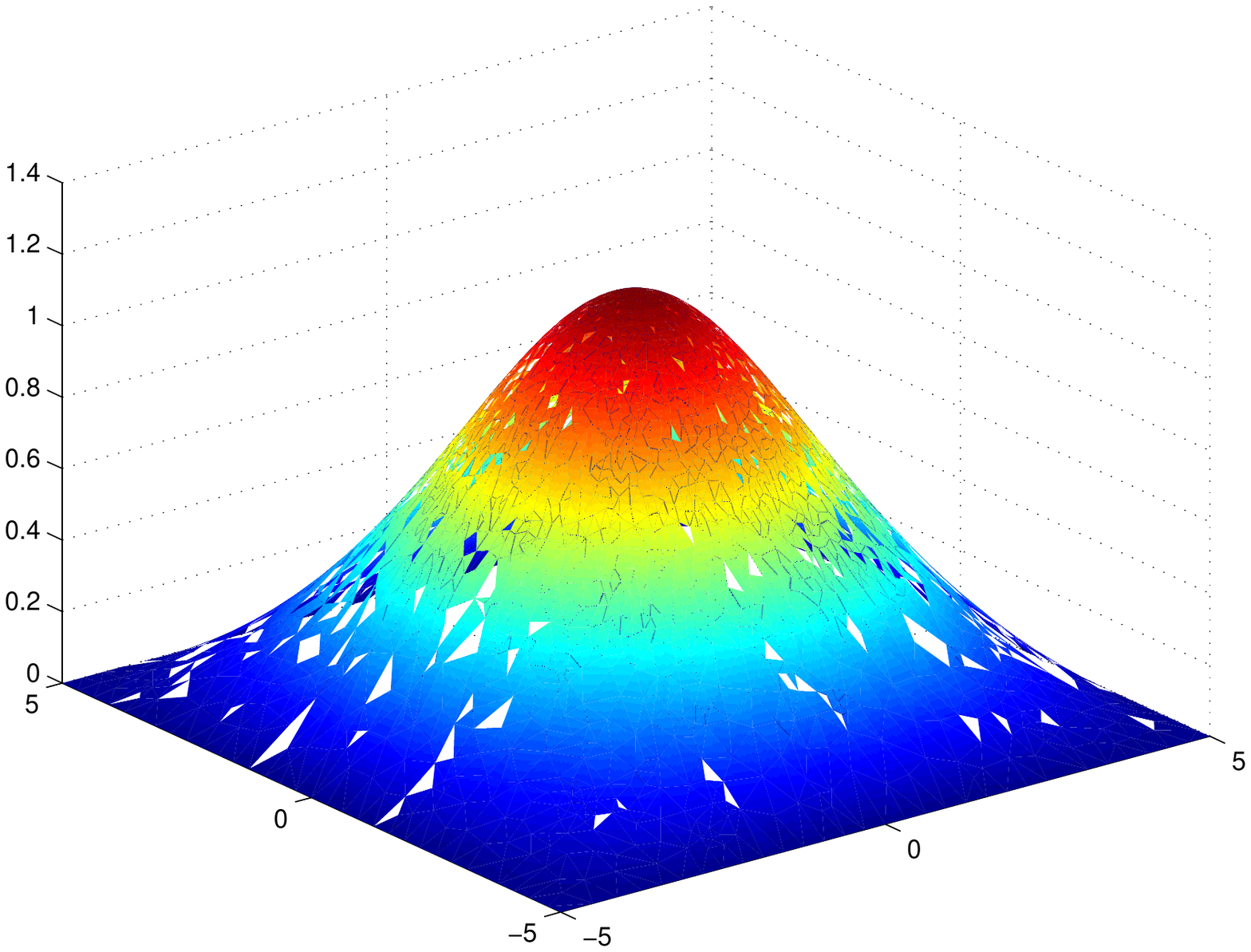}};
    \node at (2,-1) {%
      \includegraphics[width=0.25\linewidth]{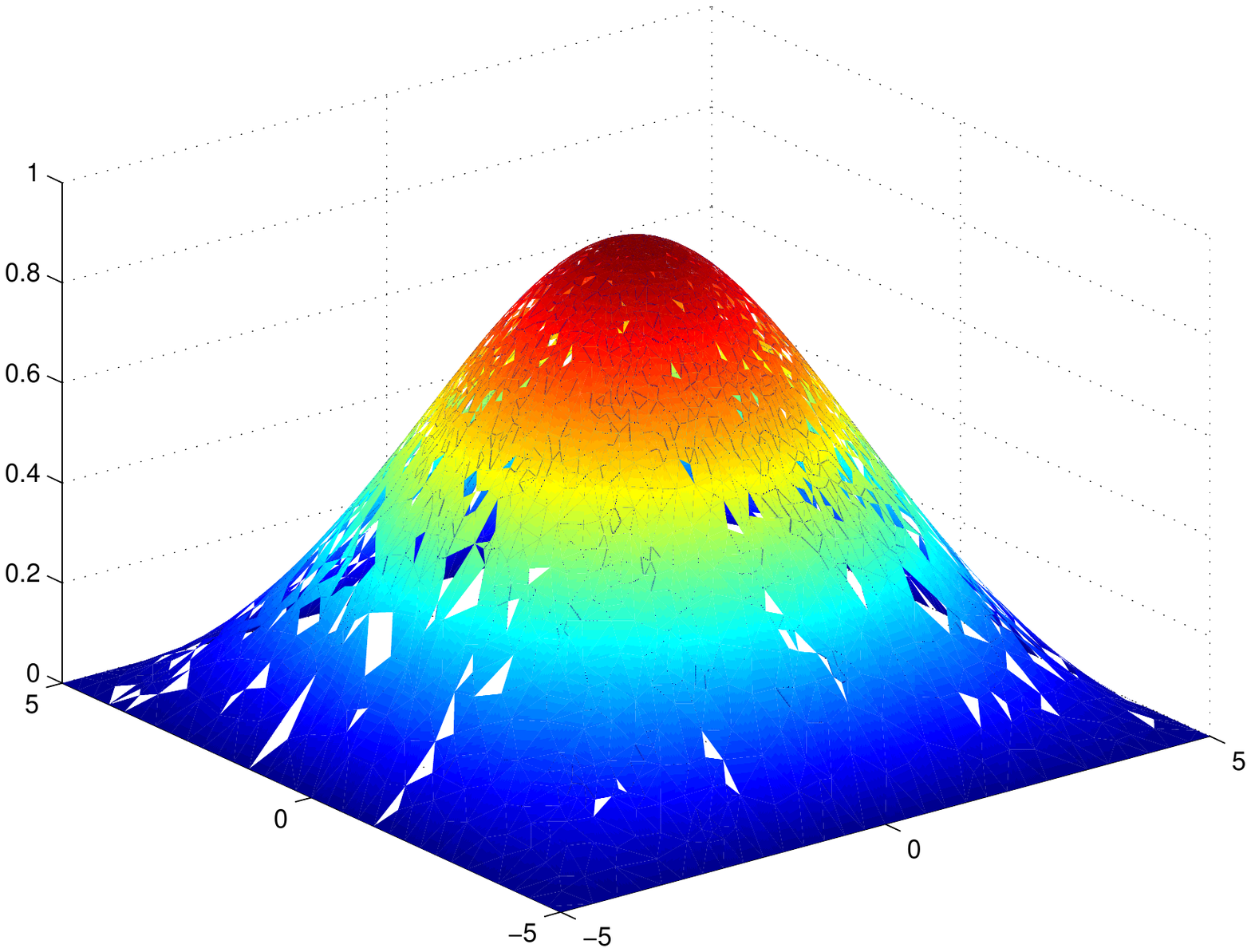}};
    \node at (2,0.5) {$R = 10$};

    \node at (3,0) {%
      \includegraphics[width=0.25\linewidth]{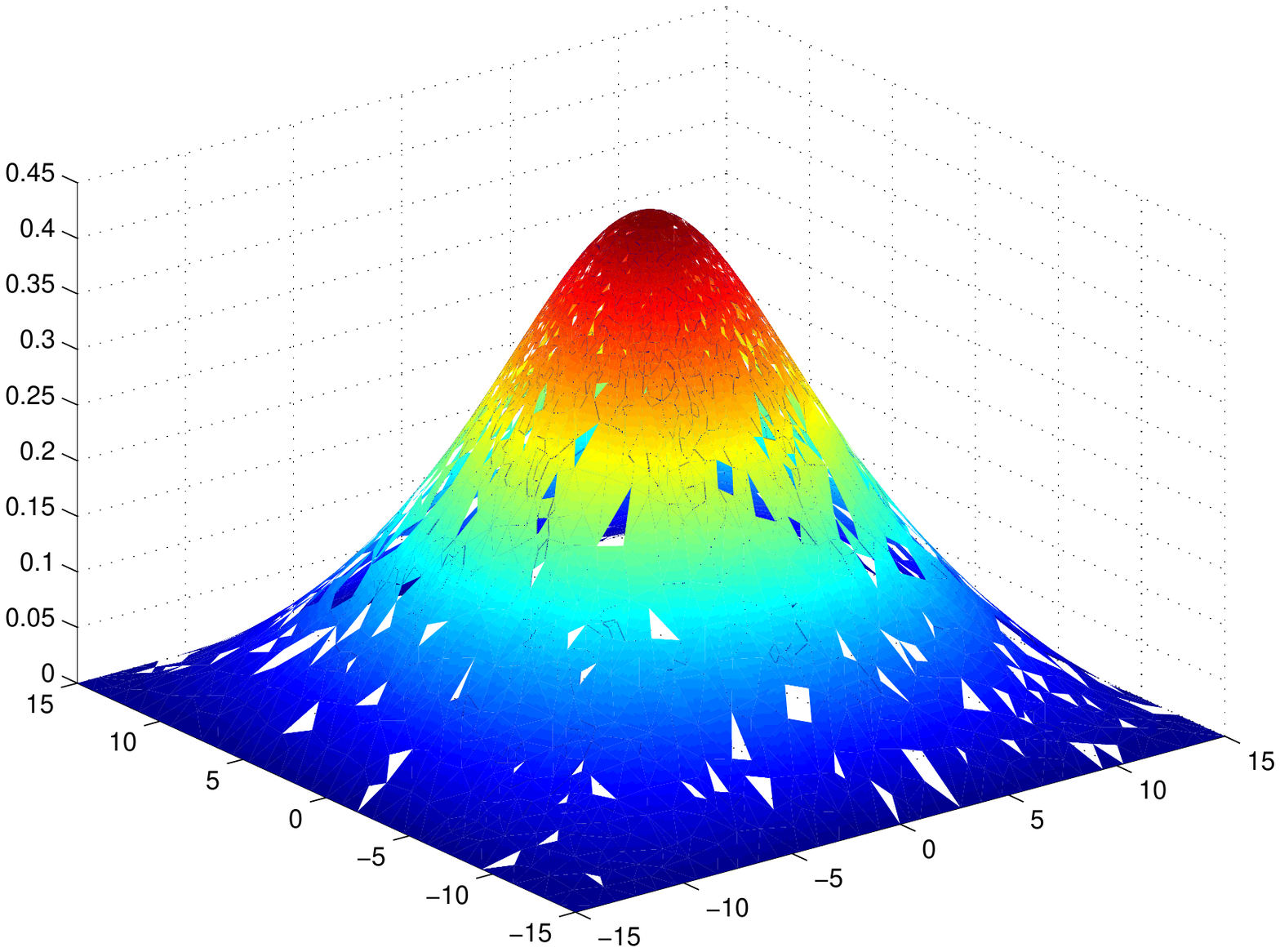}};
    \node at (3,-1) {%
      \includegraphics[width=0.25\linewidth]{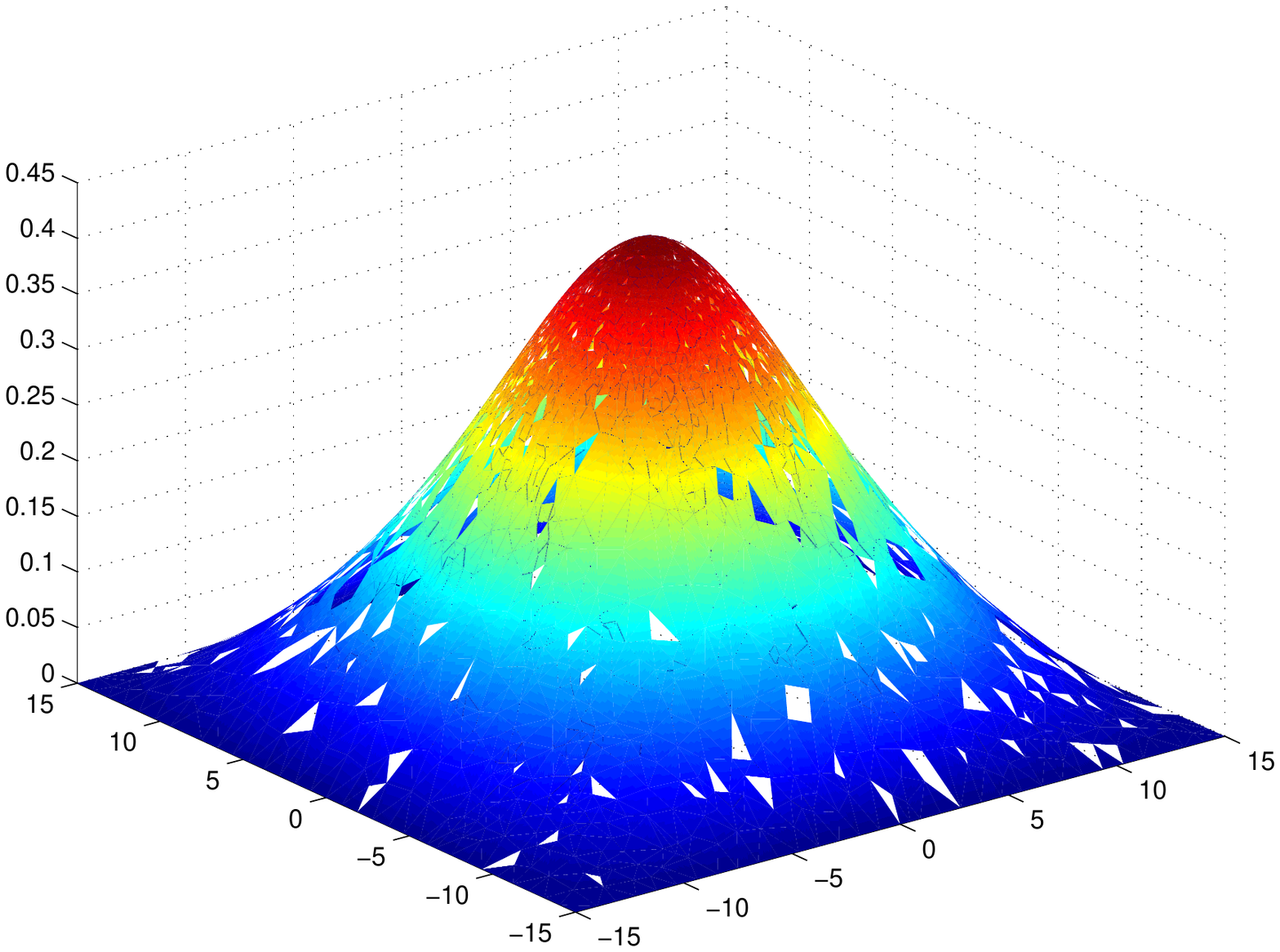}};
    \node at (3,0.5) {$R = 30$};
  \end{tikzpicture}

  \vspace*{-8ex}%
  \caption{Sol.\ $v$ to~\eqref{exis1} and $u = r(v)$ to~\eqref{gs1}
    on $\Omega_R$ for $V = 0$ and $p=4$.}
  \label{fig:square,V=0,p=4}
\end{figure}

\begin{table}[h!t]
  \begin{math}
    \begin{array}{r|cccc}
      & R = 1& R = 5& R = 10& R = 30\\
      \hline
      \max v&               415.8& 2.97&  1.17&  0.45\\
      \abs{\nabla v}_{L^2}&   875& 5.89&  2.19&  0.80\\
      \max u&                24.2& 1.77&  0.94&  0.42\\
      \mathcal T(v) =
      \mathcal E(u)&        78148&  6.75&  1.19&  0.18\\
      \norm{\nabla\mathcal{T}(v)}&
      8.2\cdot 10^{-8}& 1.4\cdot 10^{-11}& 1.6\cdot 10^{-10}& 9.6\cdot 10^{-10}
    \end{array}
  \end{math}

  \vspace{1ex}
  \caption{Characteristics of approximate solutions $v$
    to~\eqref{exis1} and $u = r(v)$ to~\eqref{gs1}
    on $\Omega_R$ for $V=0$ and $p = 4$.}
  \label{tableValues,V=0,p=4}
\end{table}

\begin{figure}[!h]
  \vspace*{-6ex}%
  \begin{tikzpicture}[x=0.25\linewidth, y=0.18\linewidth]
    \node at (0,0) {%
      \includegraphics[width=0.25\linewidth]{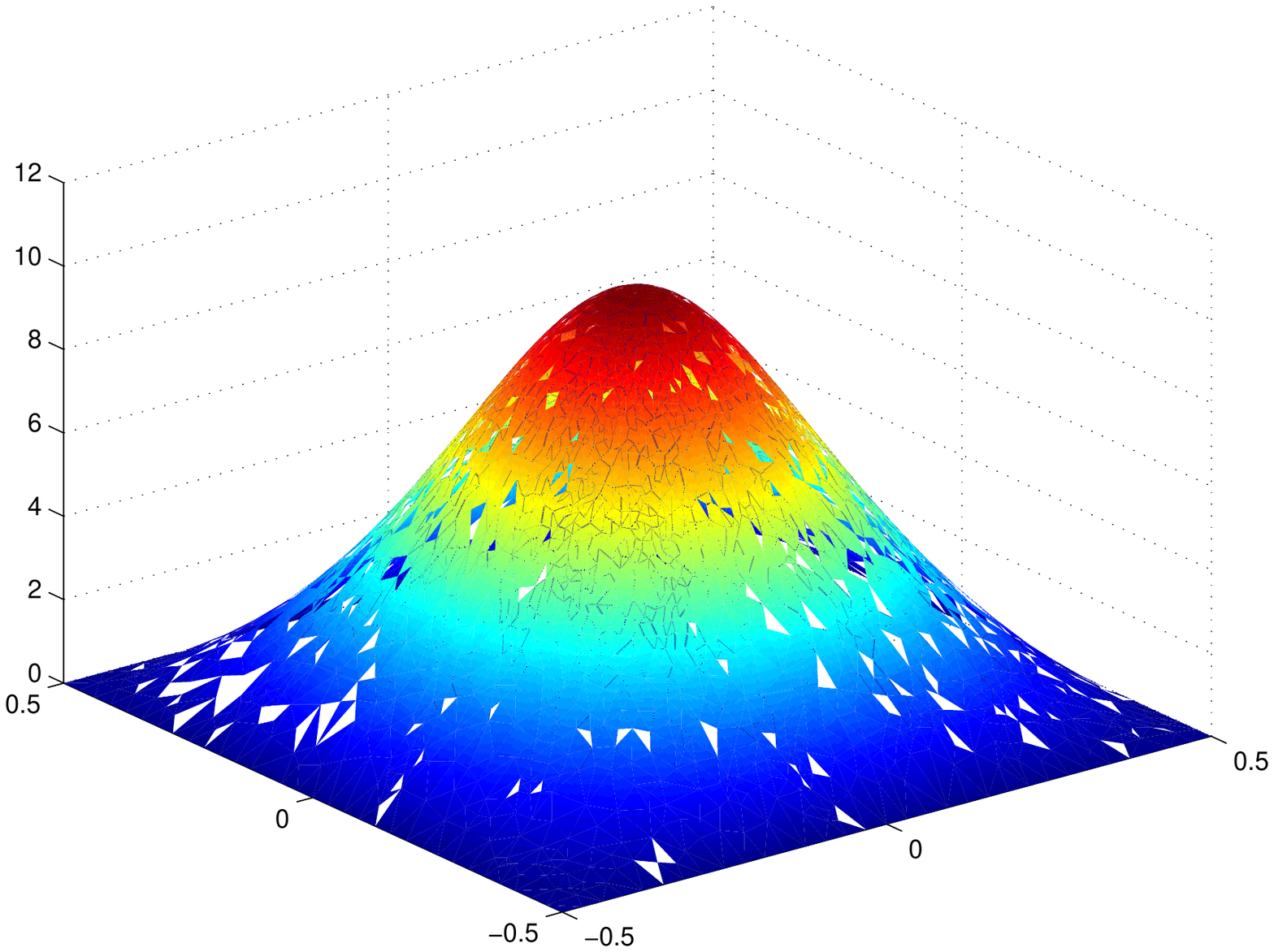}};
    \node at (0,-1) {%
      \includegraphics[width=0.25\linewidth]{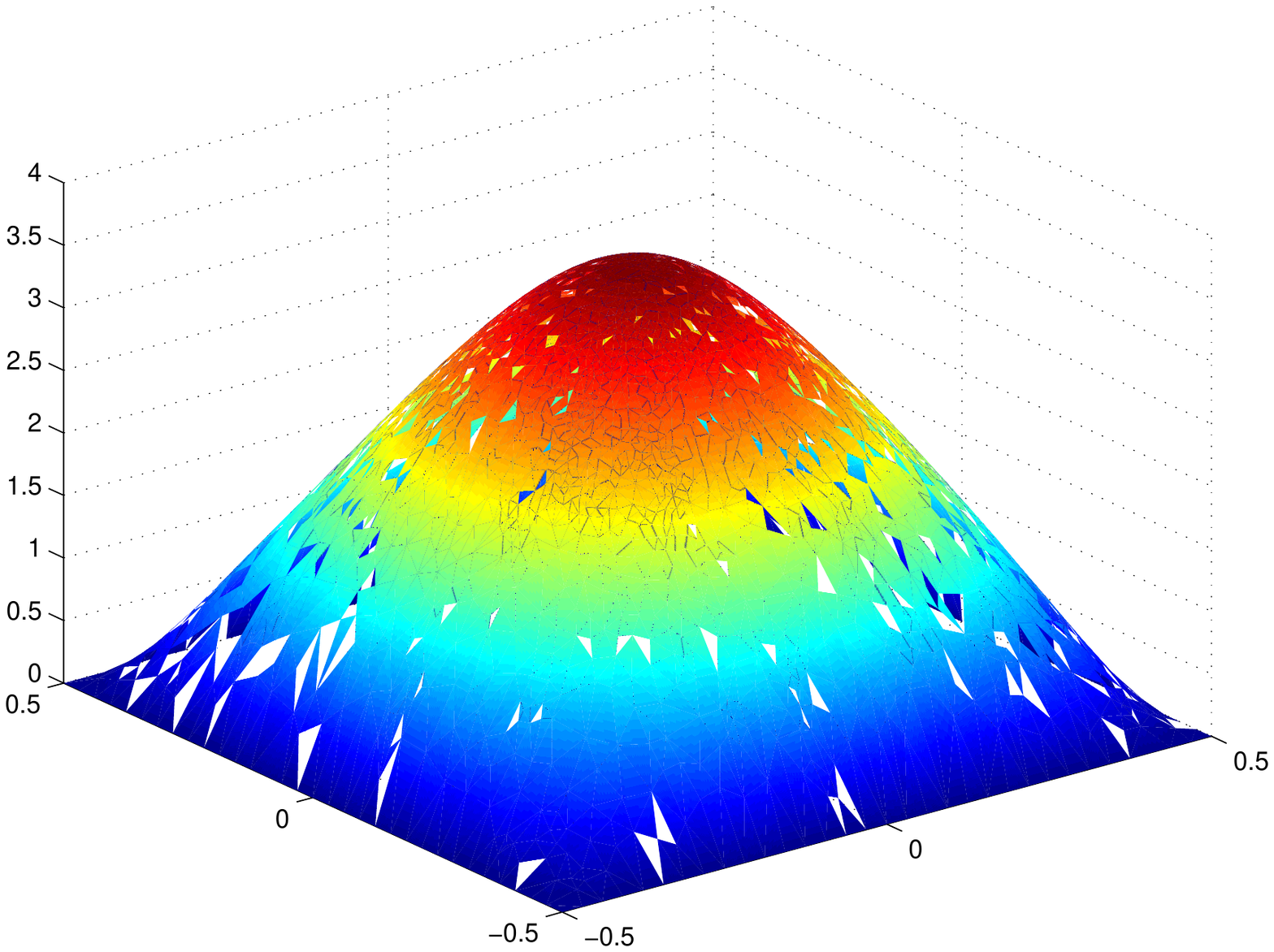}};
    \node at (0,0.5) {$R = 1$};
    \node[left] at (-0.45, 0) {$v$};
    \node[left] at (-0.45, -1) {\rotatebox{90}{$u = r(v)$}};

    \node at (1,0) {%
      \includegraphics[width=0.25\linewidth]{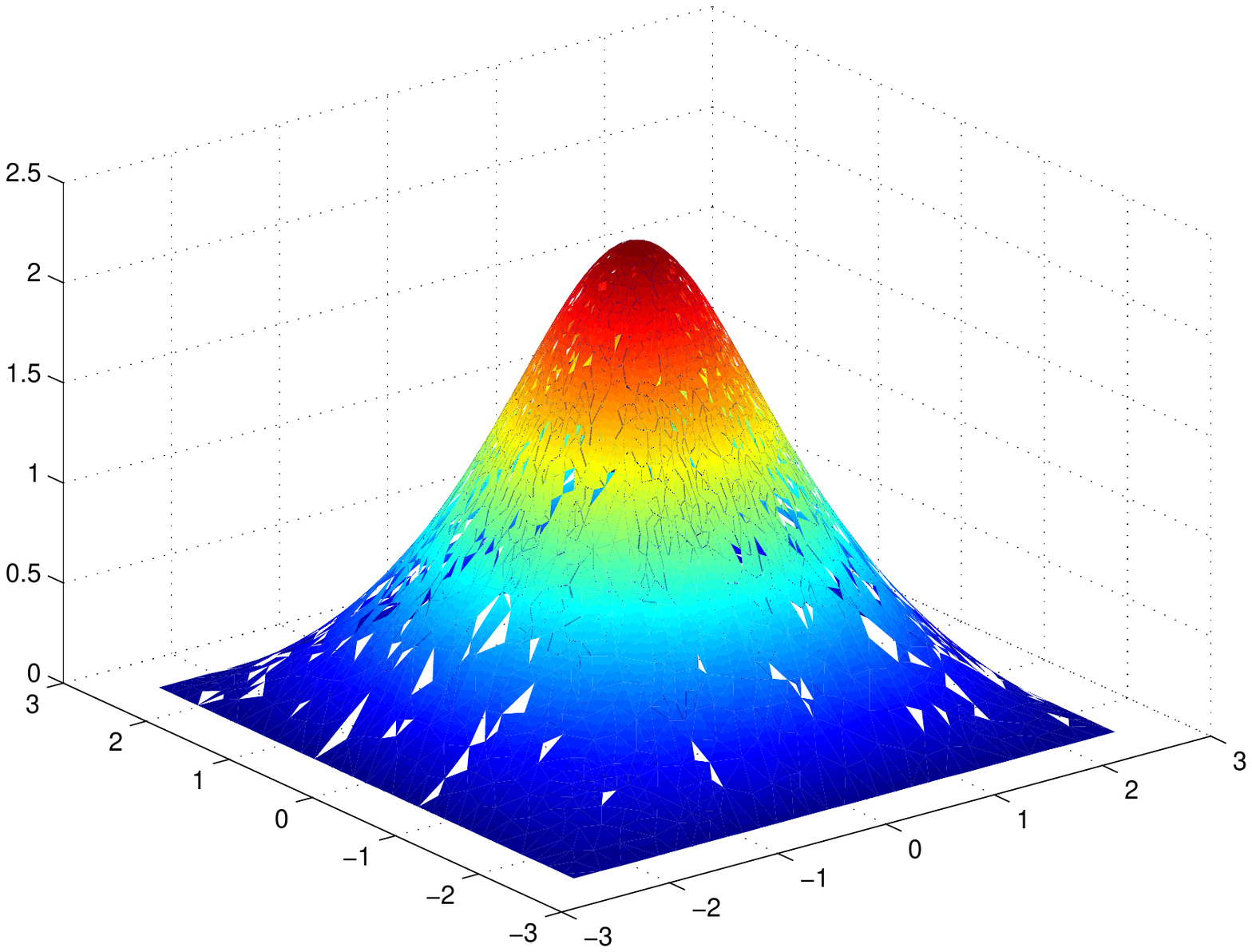}};
    \node at (1,-1) {%
      \includegraphics[width=0.25\linewidth]{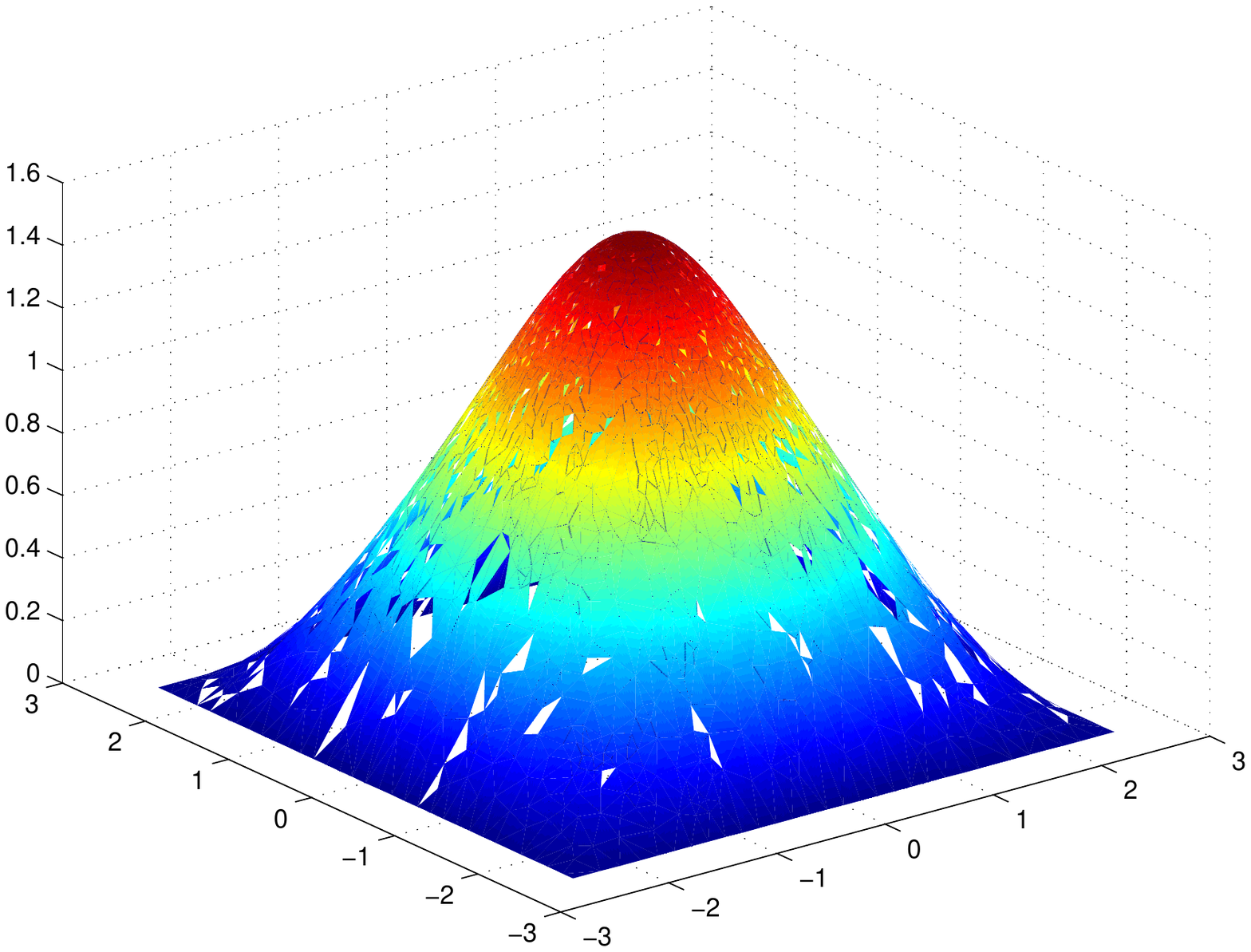}};
    \node at (1,0.5) {$R = 5$};

    \node at (2,0) {%
      \includegraphics[width=0.25\linewidth]{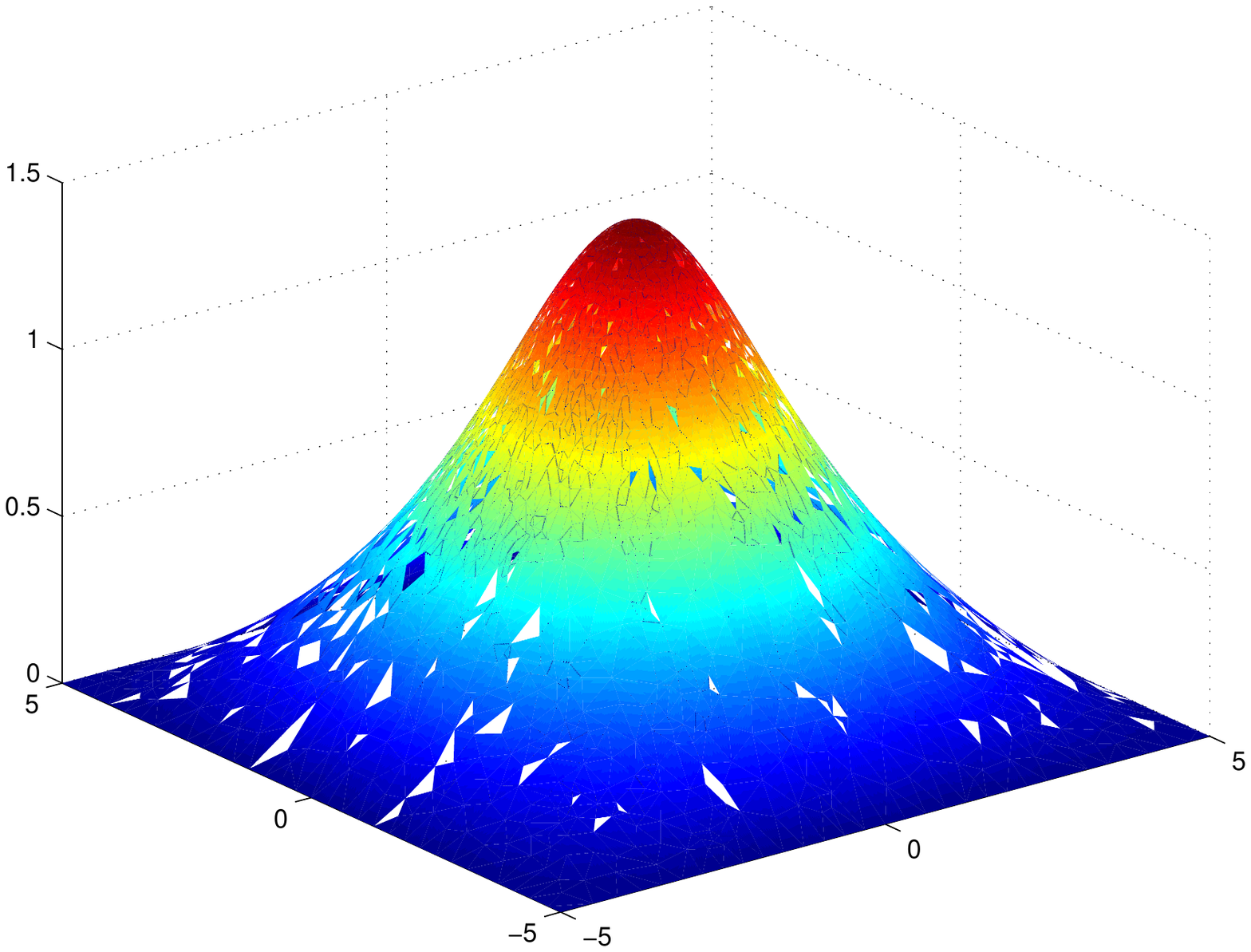}};
    \node at (2,-1) {%
      \includegraphics[width=0.25\linewidth]{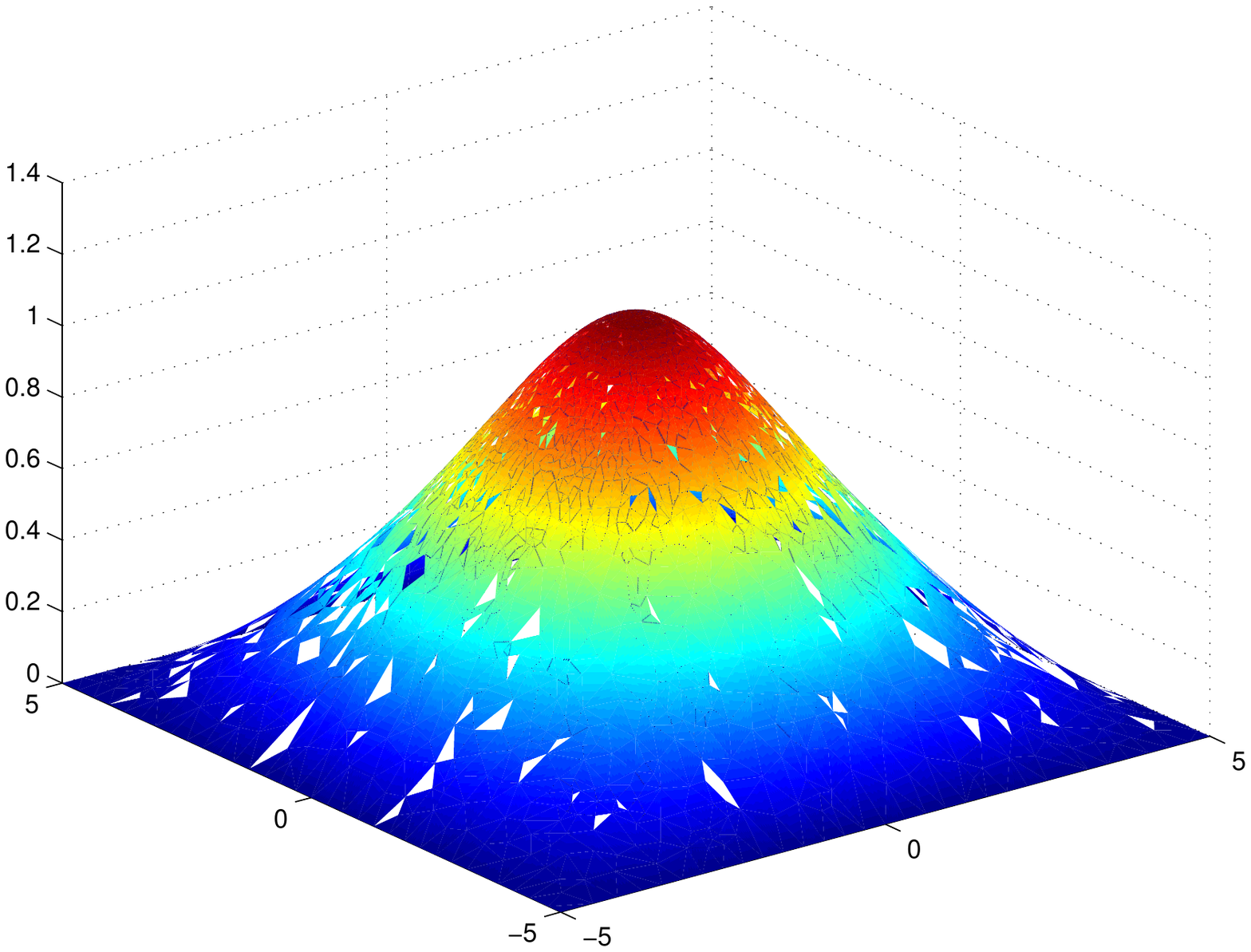}};
    \node at (2,0.5) {$R = 10$};

    \node at (3,0) {%
      \includegraphics[width=0.25\linewidth]{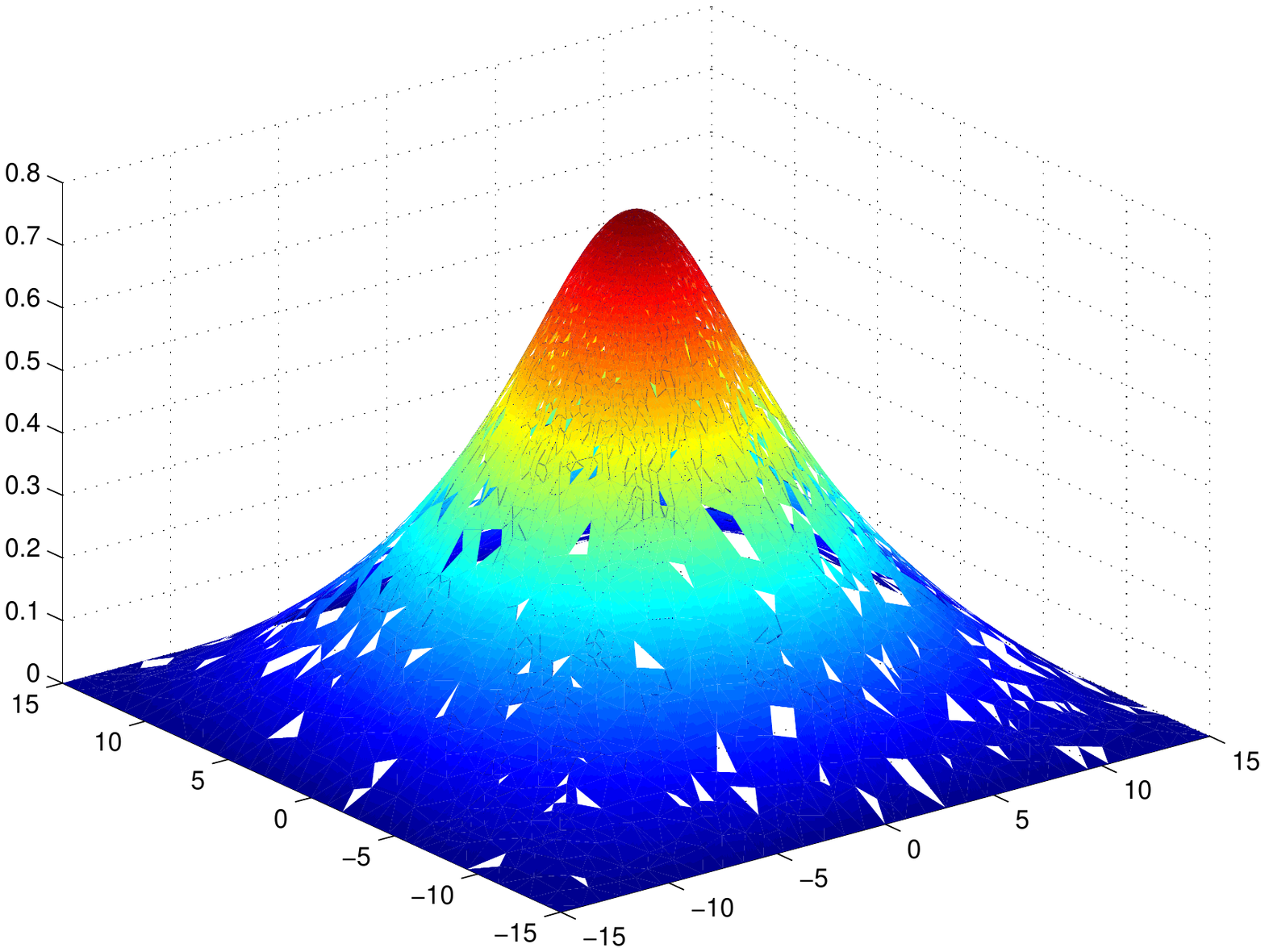}};
    \node at (3,-1) {%
      \includegraphics[width=0.25\linewidth]{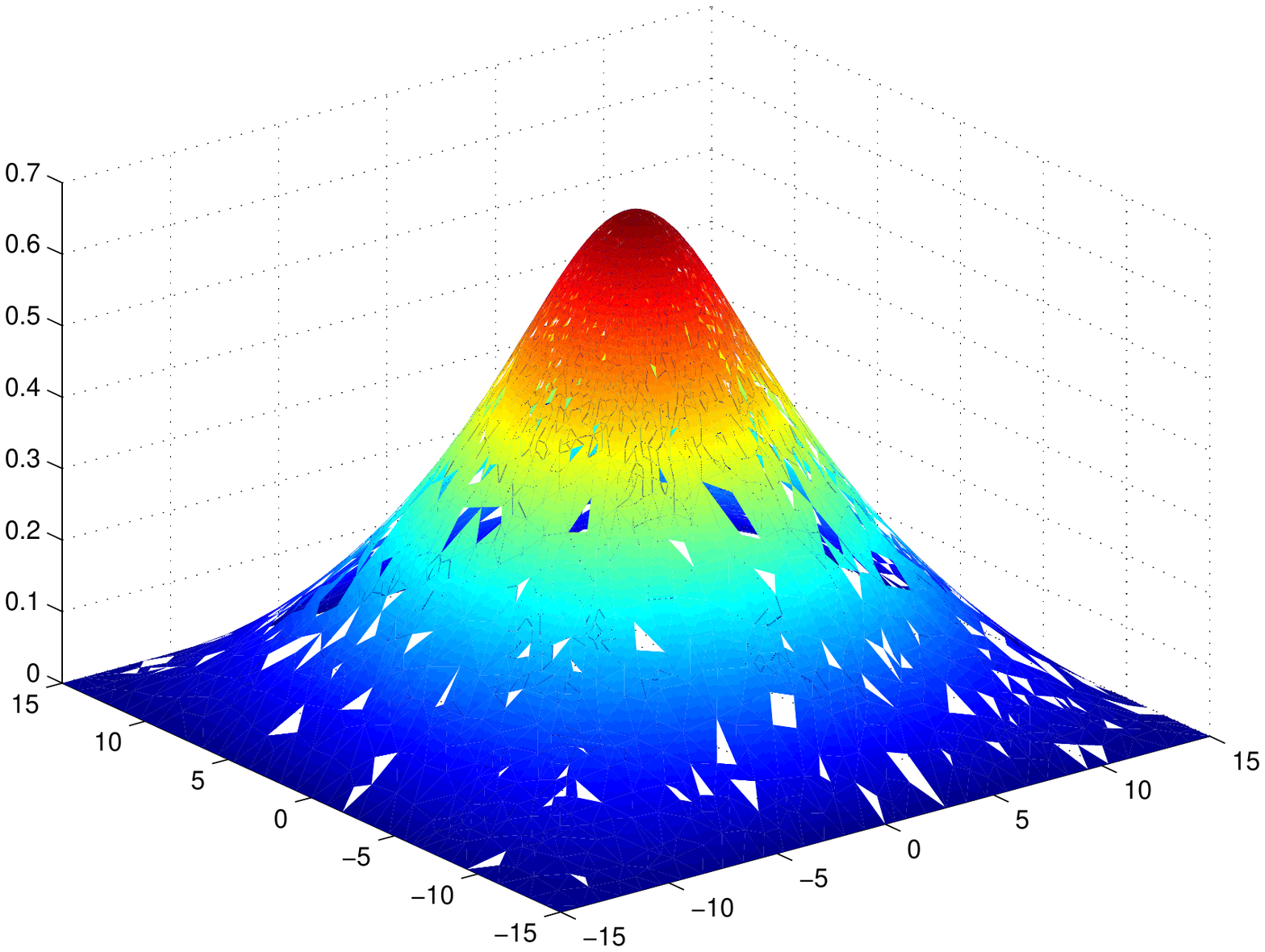}};
    \node at (3,0.5) {$R = 30$};
  \end{tikzpicture}

  \vspace*{-8ex}%
  \caption{Sol.\ $v$ to~\eqref{exis1} and $u = r(v)$ to~\eqref{gs1}
    on $\Omega_R$ for
    $V=0$ and $p=6$.}
  \label{fig:square,p=6}
\end{figure}

\begin{table}[!ht]
  \begin{math}
    \begin{array}{r|cccc}
      & R = 1& R = 5& R = 10& R = 30\\
      \hline
      \max v&               10.10&  2.34&  1.46&  0.80\\
      \abs{\nabla v}_{L^2}& 18.96&  4.15&  2.52&  1.32\\
      \max u&                3.59&  1.52&  1.11&  0.70\\
      \mathcal T(v) =
      \mathcal E(u)&        87.5&  5.00&  1.97&  0.58\\
      \norm{\nabla\mathcal{T}(v)}&
      2\cdot 10^{-8}& 4.2\cdot 10^{-11}& 3\cdot 10^{-8}& 2\cdot 10^{-9}
    \end{array}
  \end{math}

  \vspace{1ex}
  \caption{Characteristics of approximate solutions $v$
    to~\eqref{exis1} and $u = r(v)$ to~\eqref{gs1}
    on $\Omega_R$ for $V=0$ and $p = 6$.}
  \label{tableValues,V=0,p=6}
\end{table}

Then, for $p$ in the range $[4, 7]$, on 
Fig.~\ref{comparison,V=0}, we compare the energies and the
$L^\infty$-norms of the mountain pass
approximation given previously (our problem with $V =0$) with those of
the approximation  of
the problem $-\Delta u = \abs{u}^{p-1}u$. 
Equation~\eqref{exo1u} and the latter one are particular cases
of the following family.  Let $r_\delta : \R \to \R$ be the unique
solution of the Cauchy problem
\begin{equation*}
  r_\delta'(s) = \frac{1}{\sqrt{1+ \delta r^2(s)}},
  \qquad
  r_\delta(0) = 0.
\end{equation*}
If $v$ is a solution of~\eqref{exis1} with $r = r_\delta$, then $u =
r_\delta(v)$ is a solution to~\eqref{gs1} for $\delta = 2$ and to the
Lane-Emden equation for $\delta=0$.  On Figure~\ref{comparison,V=0},
we have in addition draw the curve for $\delta = 1$ to shed some light
on this homotopy.

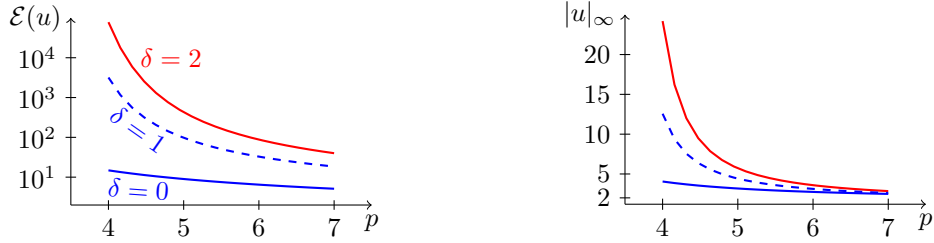
\begin{figure}[h]
  \centering
  \newcommand{\xmin}{3.5}
  \newcommand{\ymin}{0.3}
  \begin{tikzpicture}[y=3.5ex]
    \draw[->] (\xmin, \ymin) -- (7.5, \ymin) node[below]{$p$};
    \foreach \p in {4, 5, 6, 7} {%
      \draw (\p, \ymin + 0.12) -- (\p, \ymin-0.12) node[below]{$\p$};
    }
    \draw[->] (\xmin, \ymin) -- (\xmin, 5) node[left]{$\E(u)$};
    \foreach \y in {1, 2, 3, 4} {%
      \draw (\xmin + 0.07, \y) -- (\xmin - 0.07, \y) node[left]{$10^\y$};
    }
    \draw[thick,color=red] plot file{pe_V0_R1.dat};
    \node[right,color=red] at (4.3,4) {$\delta=2$};
    \draw[thick,color=blue,dashed] plot file{pe_V0_R1-delta1.dat};
    \node[color=blue, rotate=-35] at (4.4, 2.2) {$\delta=1$};
    \draw[thick,color=blue] plot file{pe_V0_R1-delta0.dat};
    \node[color=blue, rotate=-4] at (4.4, 0.7) {$\delta=0$};
  \end{tikzpicture}
  \hspace{6em}%
  \renewcommand{\ymin}{1.2}%
  \begin{tikzpicture}[y=0.7ex]
    \draw[->] (\xmin, \ymin) -- (7.5, \ymin) node[below]{$p$};
    \foreach \p in {4, 5, 6, 7} {%
      \draw (\p, \ymin + 0.7) -- (\p, \ymin - 0.7) node[below]{$\p$};
    }
    \draw[->] (\xmin, \ymin) -- (\xmin, 25) node[left]{$\abs{u}_\infty$};
    \foreach \y in {2, 5, 10,..., 23} {%
      \draw (\xmin + 0.07, \y) -- (\xmin - 0.07, \y) node[left]{$\y$};
    }
    \draw[thick,color=red] plot file{pn_V0_R1.dat};
    \draw[thick,color=blue,dashed] plot file{pn_V0_R1-delta1.dat};
    \draw[thick,color=blue] plot file{pn_V0_R1-delta0.dat};
  \end{tikzpicture}

  \vspace*{-2ex}
  \caption{Comparison between~\eqref{gs1} with $V=0$ and the problem
    $-\Delta u= \abs{u}^{p-1}u$ on $(-0.5,0.5)^2$.}
  \label{comparison,V=0}
\end{figure}

\subsection{Non-zero potentials}

In this section, we study the case $V =10$ and $p=4$ or $6$. We repeat
the experiments above, see Figures~\ref{fig:square,V=10,p=4}
and~\ref{fig:square,V=10,p=6} as well as
Tables~\ref{tableValues,V=10,p=4}
and~\ref{tableValues,V=10,p=6}.
As before, we also compare the energies and the $L^\infty$-norms of
the mountain pass approximation given previously (our problem with
$V=10$) with those of the approximation of the problem $-\Delta u +10u
= \abs{u}^{p-1}u$ for $p \in [4, 7]$, see Fig.~\ref{comparison,V=10}.
This time, when $R \to \infty$, the solutions no longer vanish but
seem to converge to a non-trivial positive solution on $\R^N$.
This is what is expected in view of the
argumentation in section~\ref{sec:conv-MP}.

We conclude this section by examining the behavior of the solutions as
the potential $V$ goes to~$0$.  Figure~\ref{V->0,p=4} depicts the
energy of the solution to~\eqref{gs1} obtained by the Mountain Pass
Algorithm on $\Omega_R$ with $R = 10$ for various values of~$V$.
It clearly suggests that $\mathcal{E}(u) \to 0$ as $V \to 0$,
indicating that the ground-state solutions bifurcate from $0$
as $V$ approaches the bottom of the spectrum.

\begin{figure}[h!t]
  \vspace*{-6ex}%
  \begin{tikzpicture}[x=0.25\linewidth, y=0.18\linewidth]
    \node at (0,0) {%
      \includegraphics[width=0.25\linewidth]{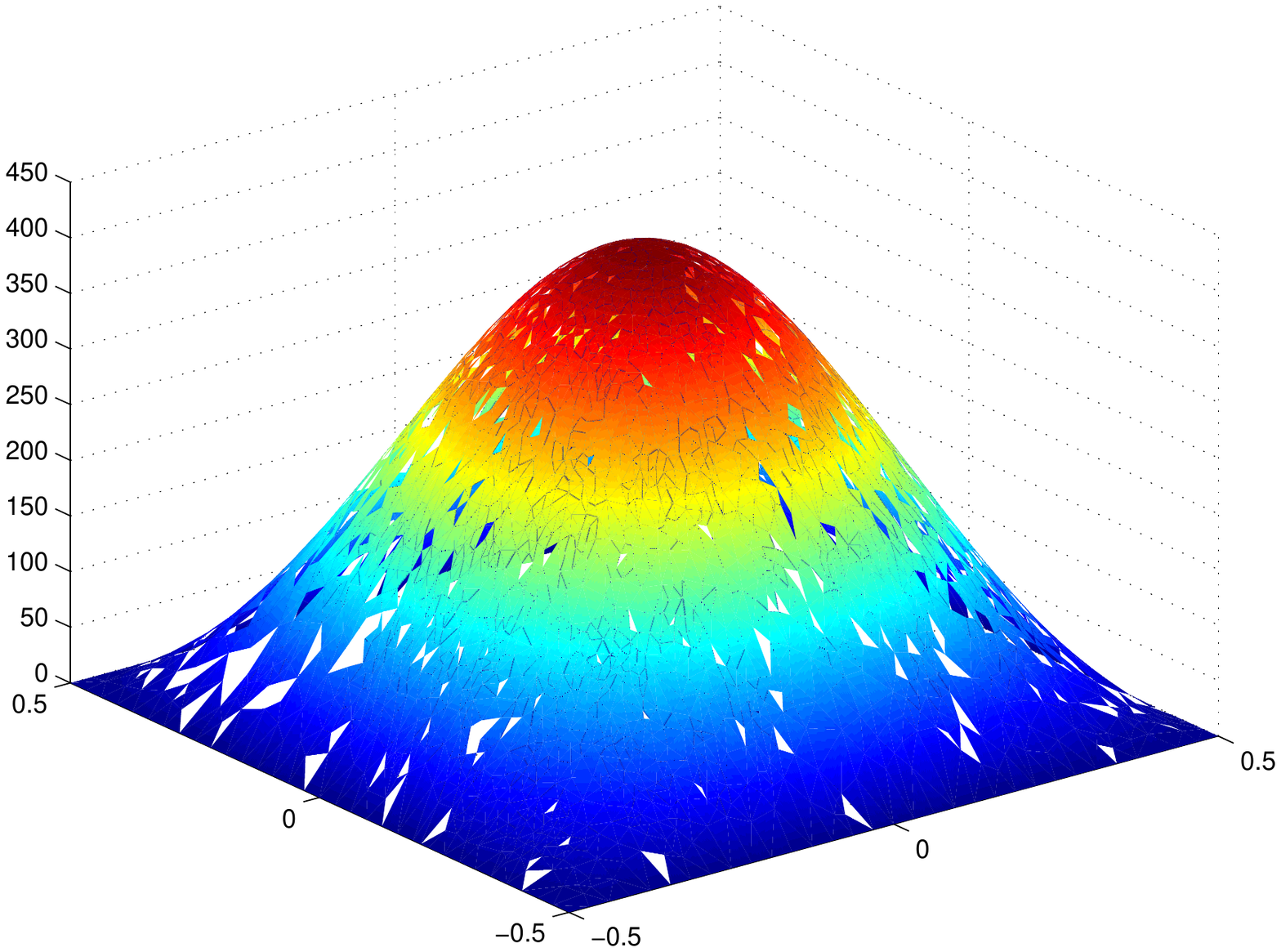}};
    \node at (0,-1) {%
      \includegraphics[width=0.25\linewidth]{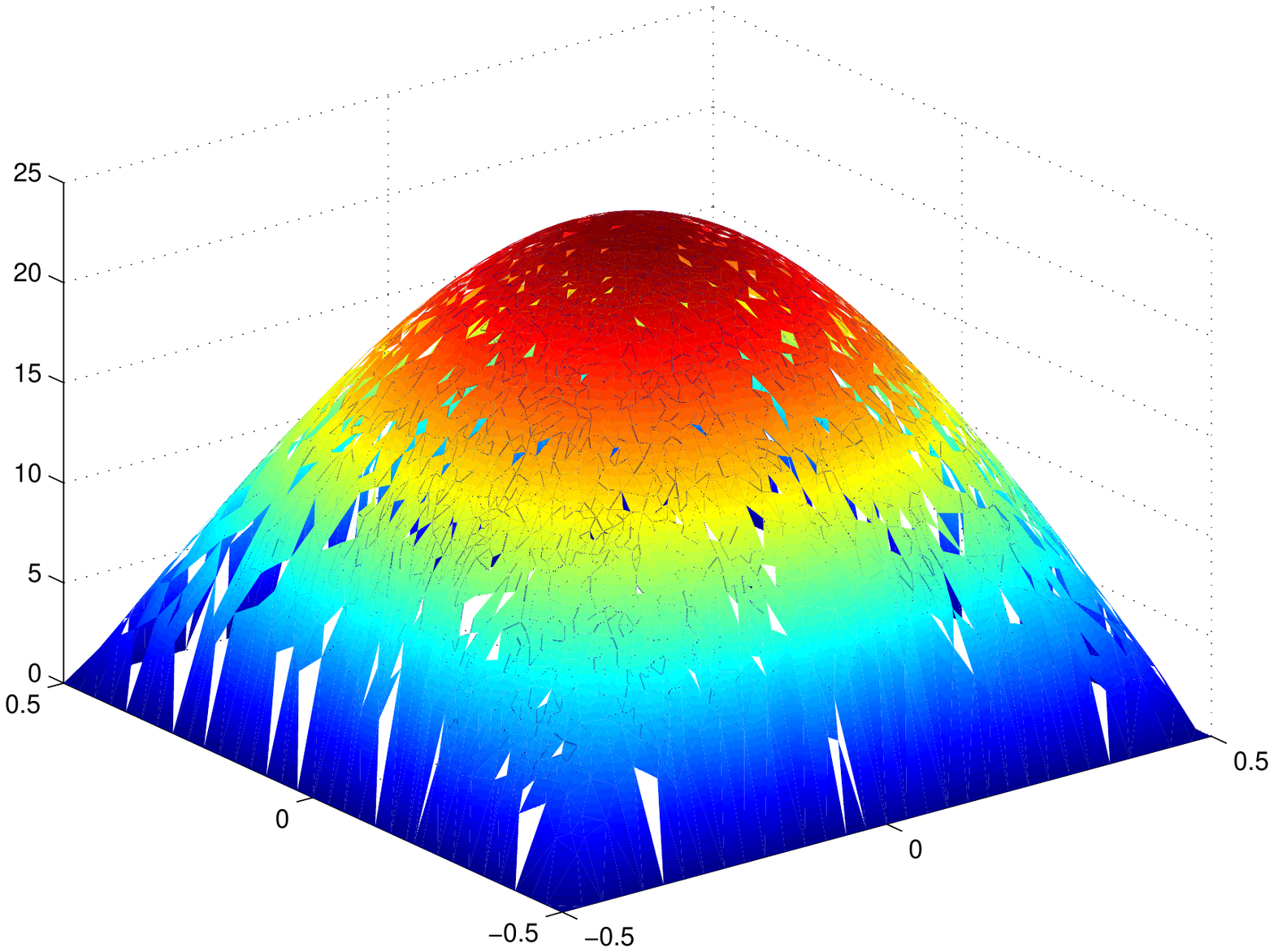}};
    \node at (0,0.5) {$R = 1$};
    \node[left] at (-0.45, 0) {$v$};
    \node[left] at (-0.45, -1) {\rotatebox{90}{$u = r(v)$}};

    \node at (1,0) {%
      \includegraphics[width=0.25\linewidth]{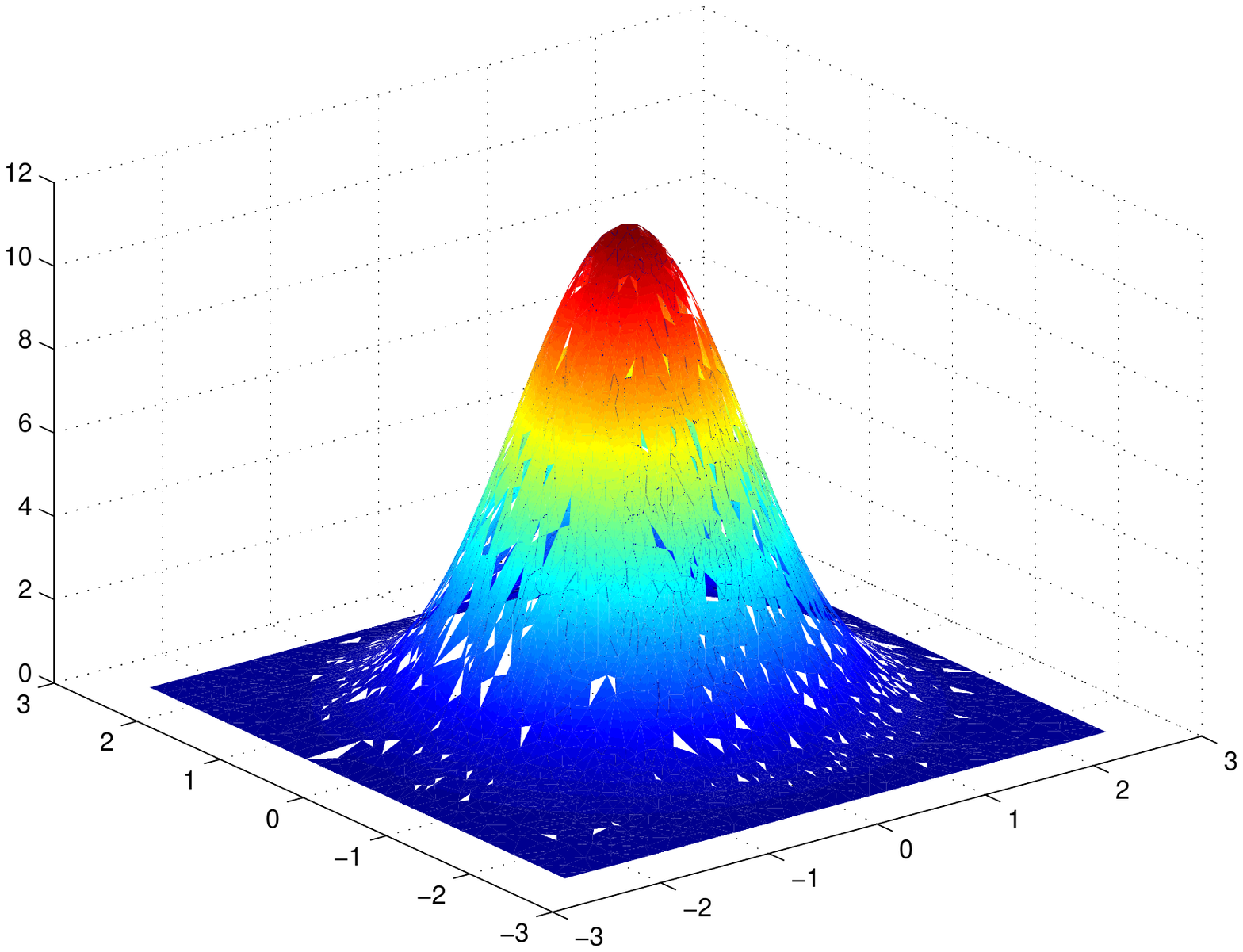}};
    \node at (1,-1) {%
      \includegraphics[width=0.25\linewidth]{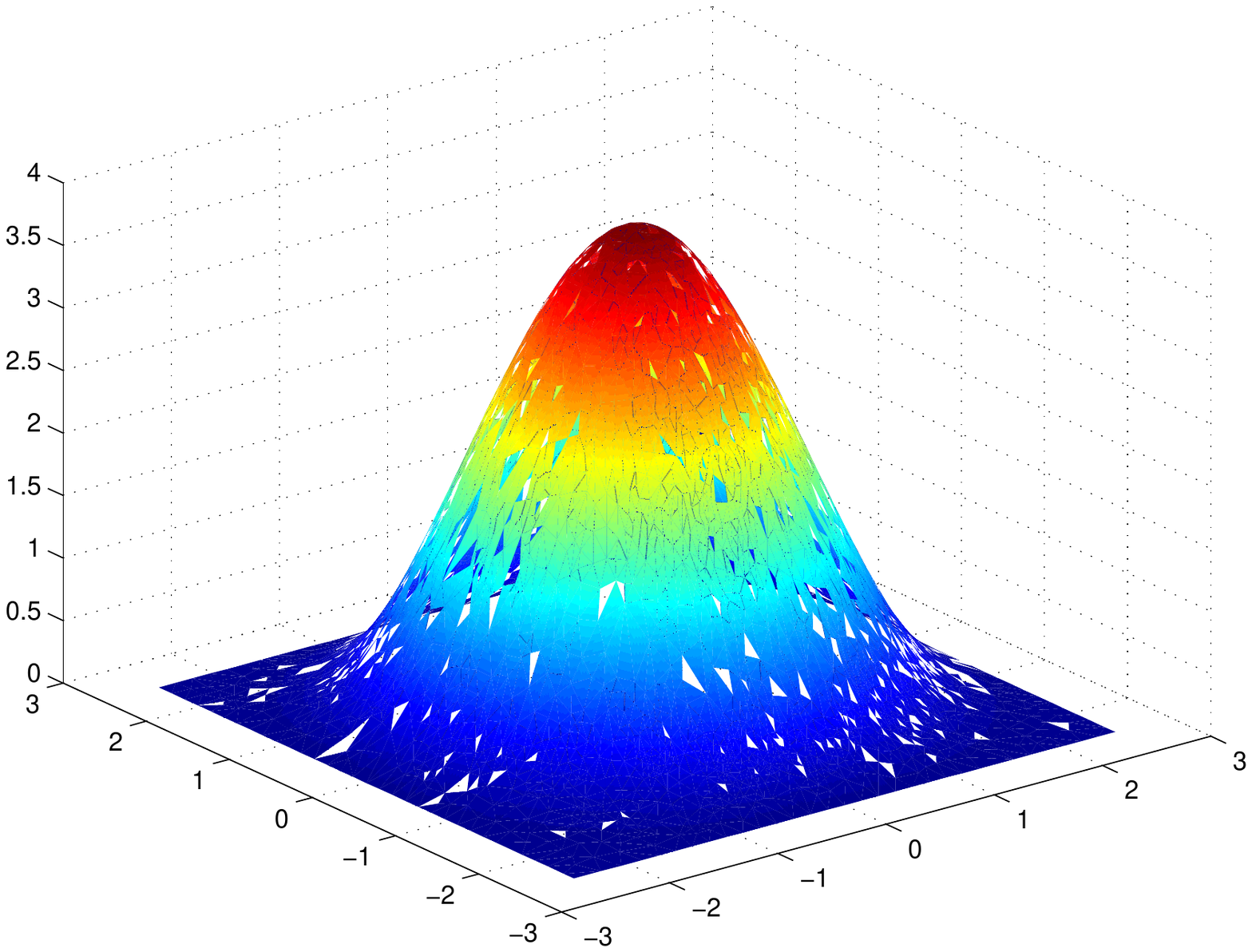}};
    \node at (1,0.5) {$R = 5$};

    \node at (2,0) {%
      \includegraphics[width=0.25\linewidth]{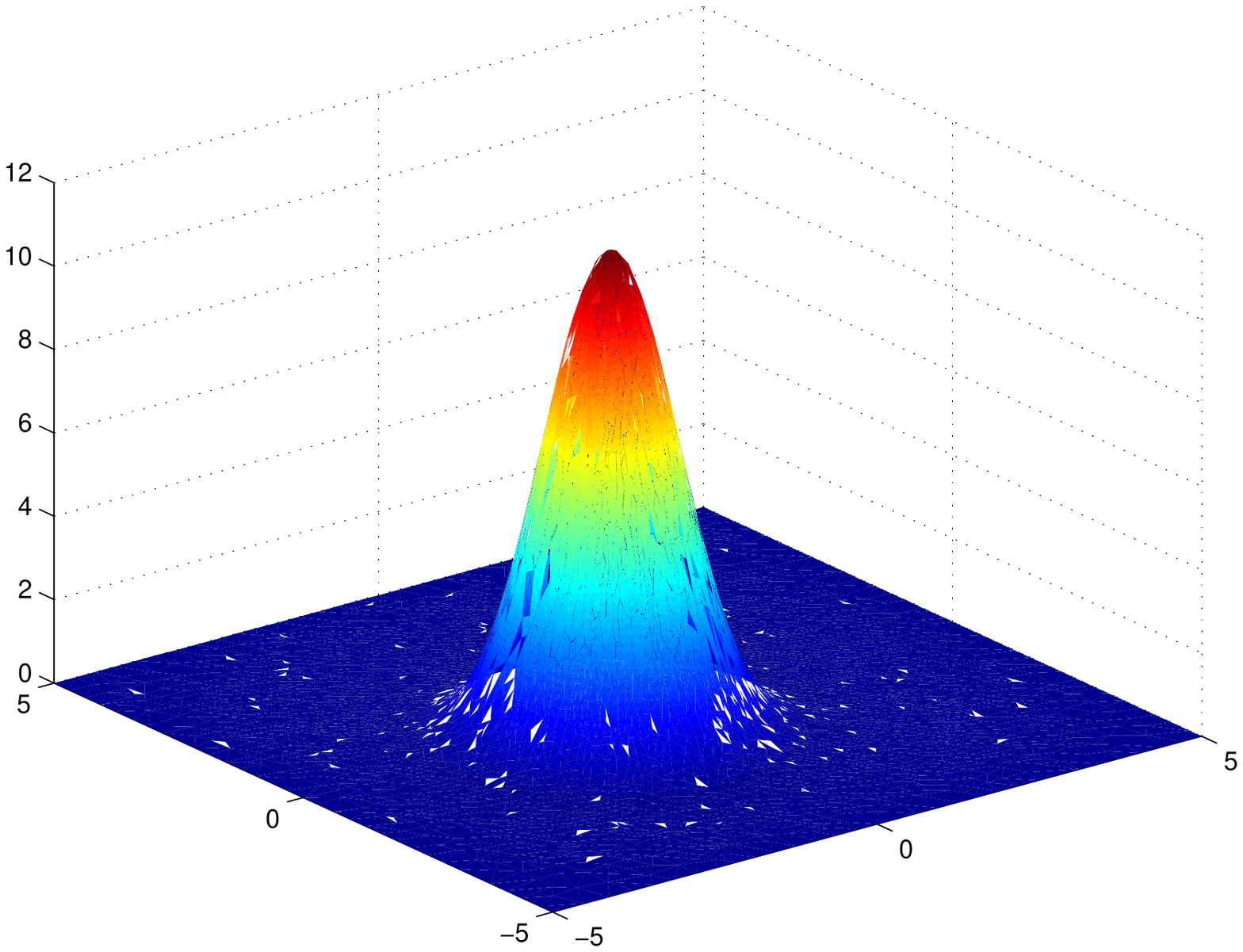}};
    \node at (2,-1) {%
      \includegraphics[width=0.25\linewidth]{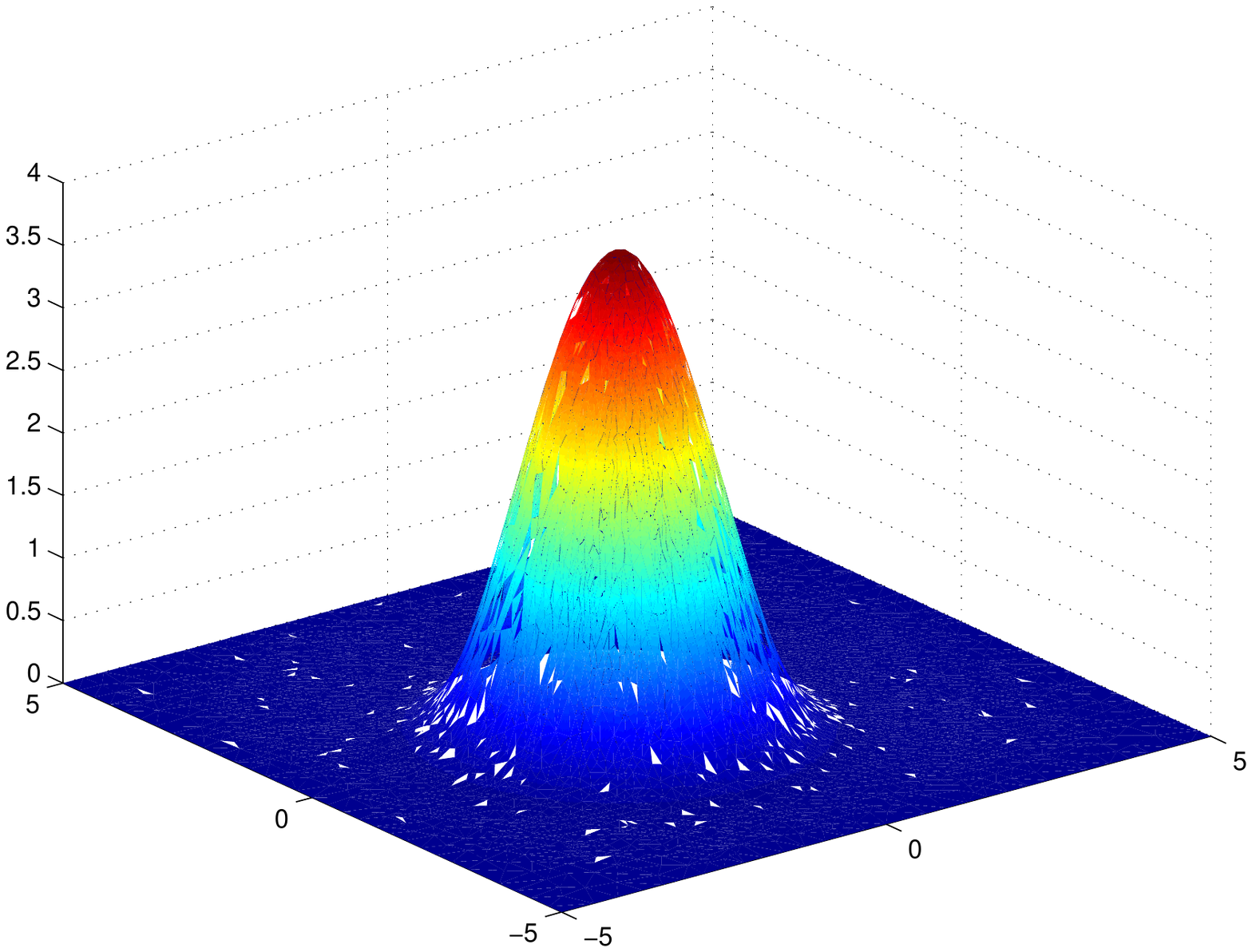}};
    \node at (2,0.5) {$R = 10$};

    \node at (3,0) {%
      \includegraphics[width=0.25\linewidth]{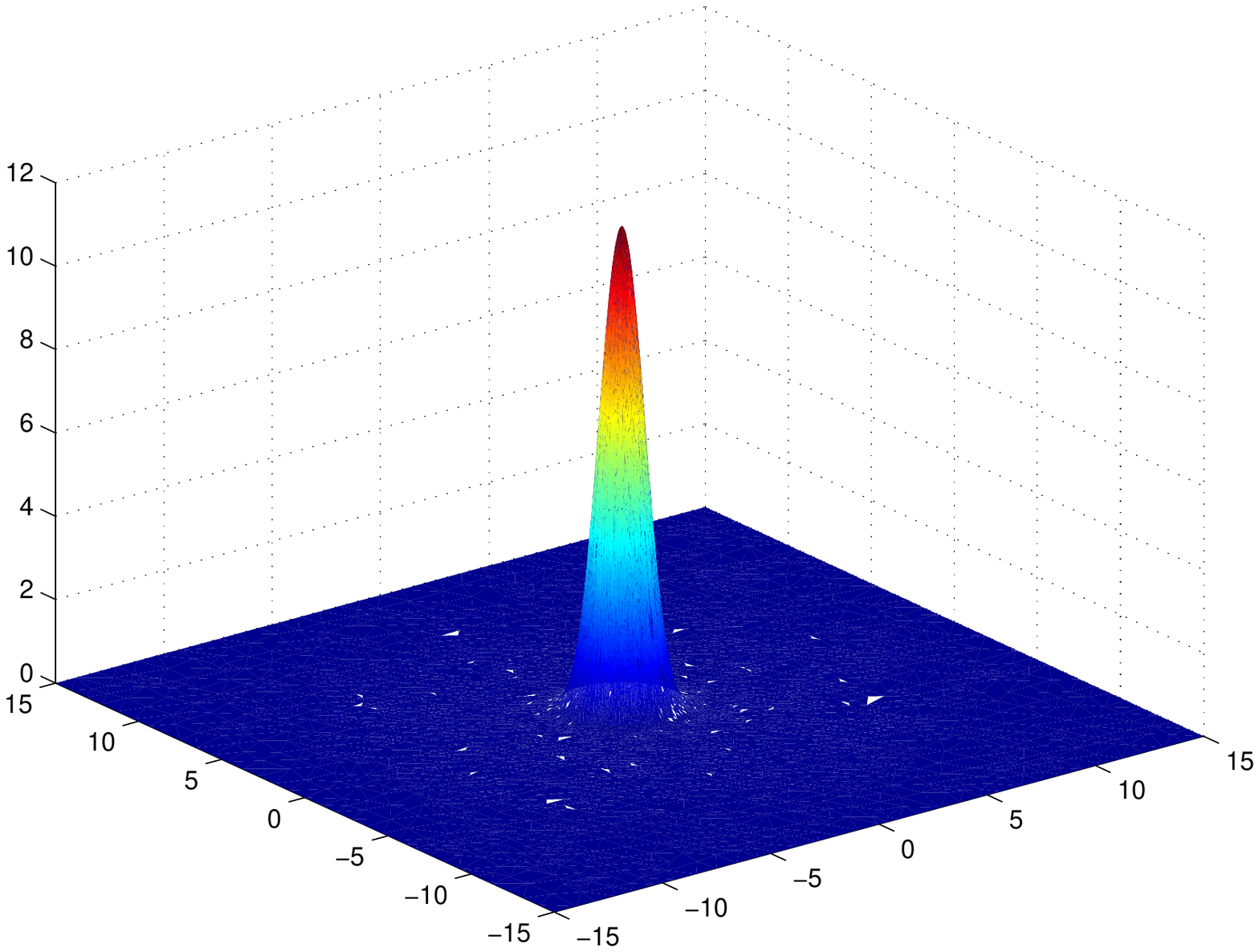}};
    \node at (3,-1) {%
      \includegraphics[width=0.25\linewidth]{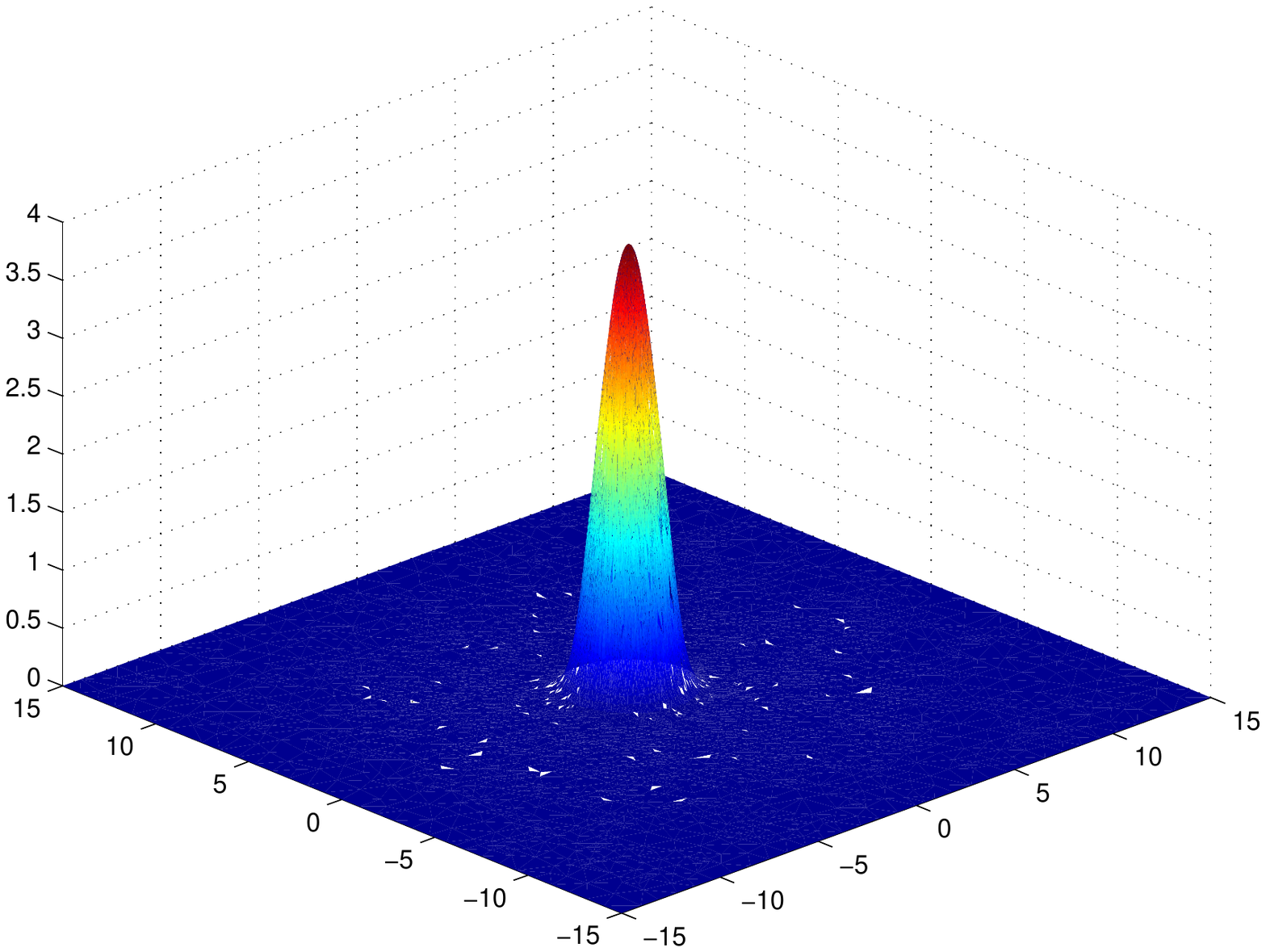}};
    \node at (3,0.5) {$R = 30$};
  \end{tikzpicture}

  \vspace*{-8ex}%
  \caption{Sol.\ $v$ to~\eqref{exis1} and $u = r(v)$ to~\eqref{gs1}
    on $\Omega_R$ for $V=10$ and $p=4$.}
  \label{fig:square,V=10,p=4}
\end{figure}

\begin{table}[ht]
  \begin{math}
    \begin{array}{r|cccc}
      & R = 1& R = 5& R = 10& R = 30\\
      \hline
      \max v&               417&   11.6&  11.6& 11.6\\
      \abs{\nabla v}_{L^2}& 877&   20.9&  20.8& 20.8\\
      \max u&               24.24& 3.87&  3.87& 3.87\\
      \mathcal T(v) =
      \mathcal E(u)&        79183&  217& 217& 216.7\\
      \norm{\nabla\mathcal{T}(v)}&
      2\cdot 10^{-8}& 4\cdot 10^{-10}&2\cdot 10^{-8}& 1.6\cdot 10^{-8}
    \end{array}
  \end{math}

  \vspace{1ex}
  \caption{Characteristics of approximate solutions $v$
    to~\eqref{exis1} and $u = r(v)$ to~\eqref{gs1} on $\Omega_R$
    for $V =10$ and $p = 4$.}
  \vspace{-1ex}
  \label{tableValues,V=10,p=4}
\end{table}

\begin{figure}[ht]
  \vspace*{-6ex}%
  \begin{tikzpicture}[x=0.25\linewidth, y=0.18\linewidth]
    \node at (0,0) {%
      \includegraphics[width=0.25\linewidth]{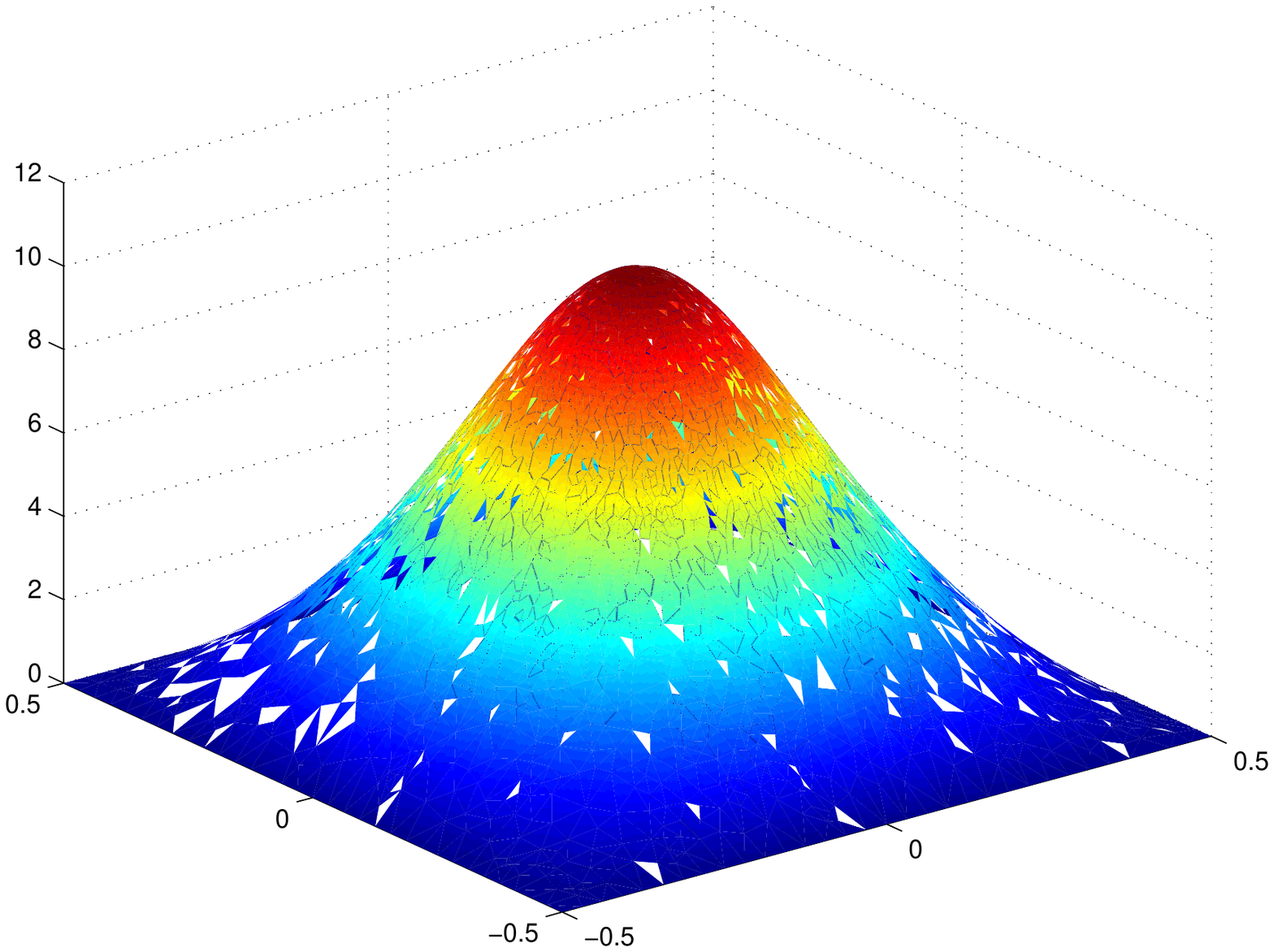}};
    \node at (0,-1) {%
      \includegraphics[width=0.25\linewidth]{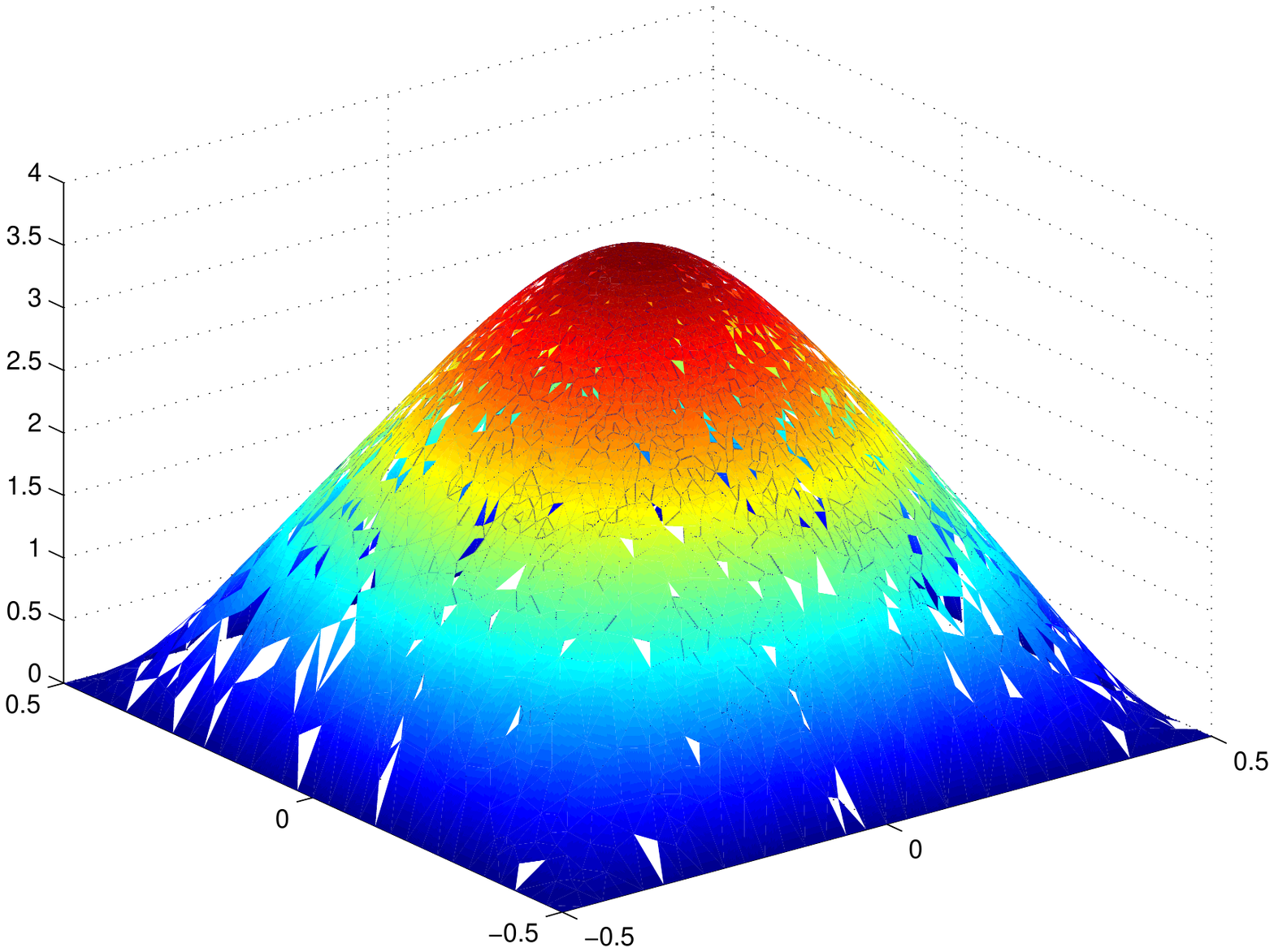}};
    \node at (0,0.5) {$R = 1$};
    \node[left] at (-0.45, 0) {$v$};
    \node[left] at (-0.45, -1) {\rotatebox{90}{$u = r(v)$}};

    \node at (1,0) {%
      \includegraphics[width=0.25\linewidth]{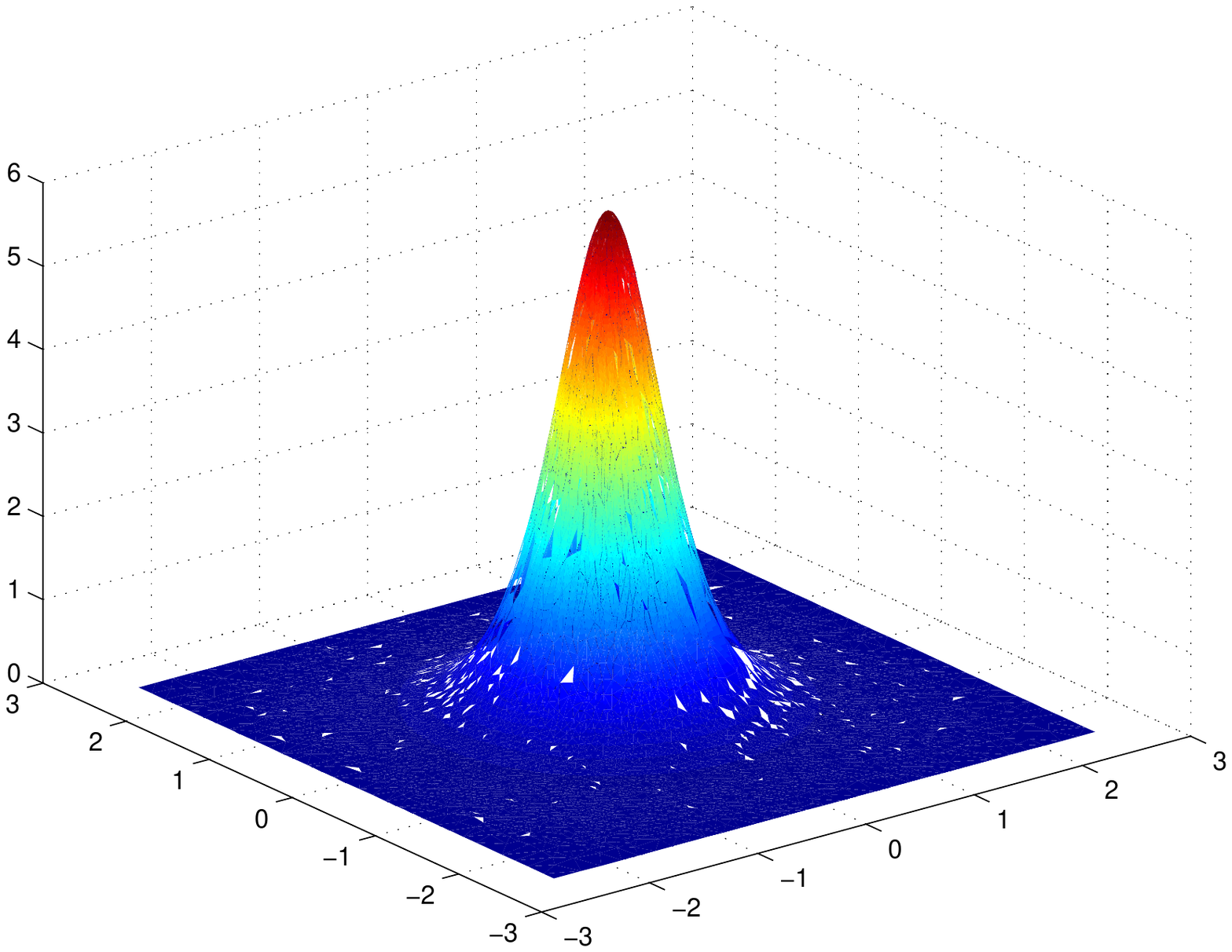}};
    \node at (1,-1) {%
      \includegraphics[width=0.25\linewidth]{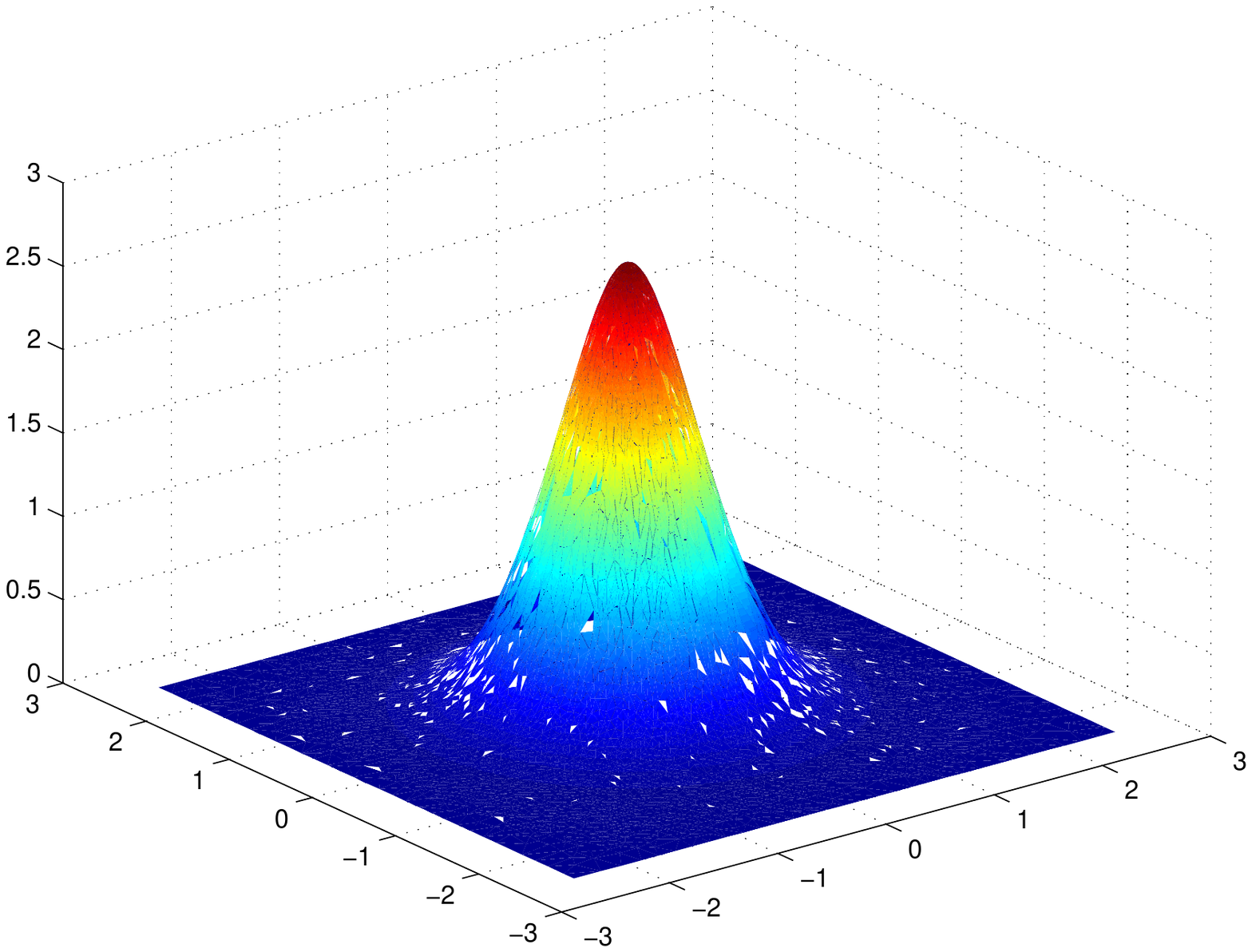}};
    \node at (1,0.5) {$R=5$};

    \node at (2,0) {%
      \includegraphics[width=0.25\linewidth]{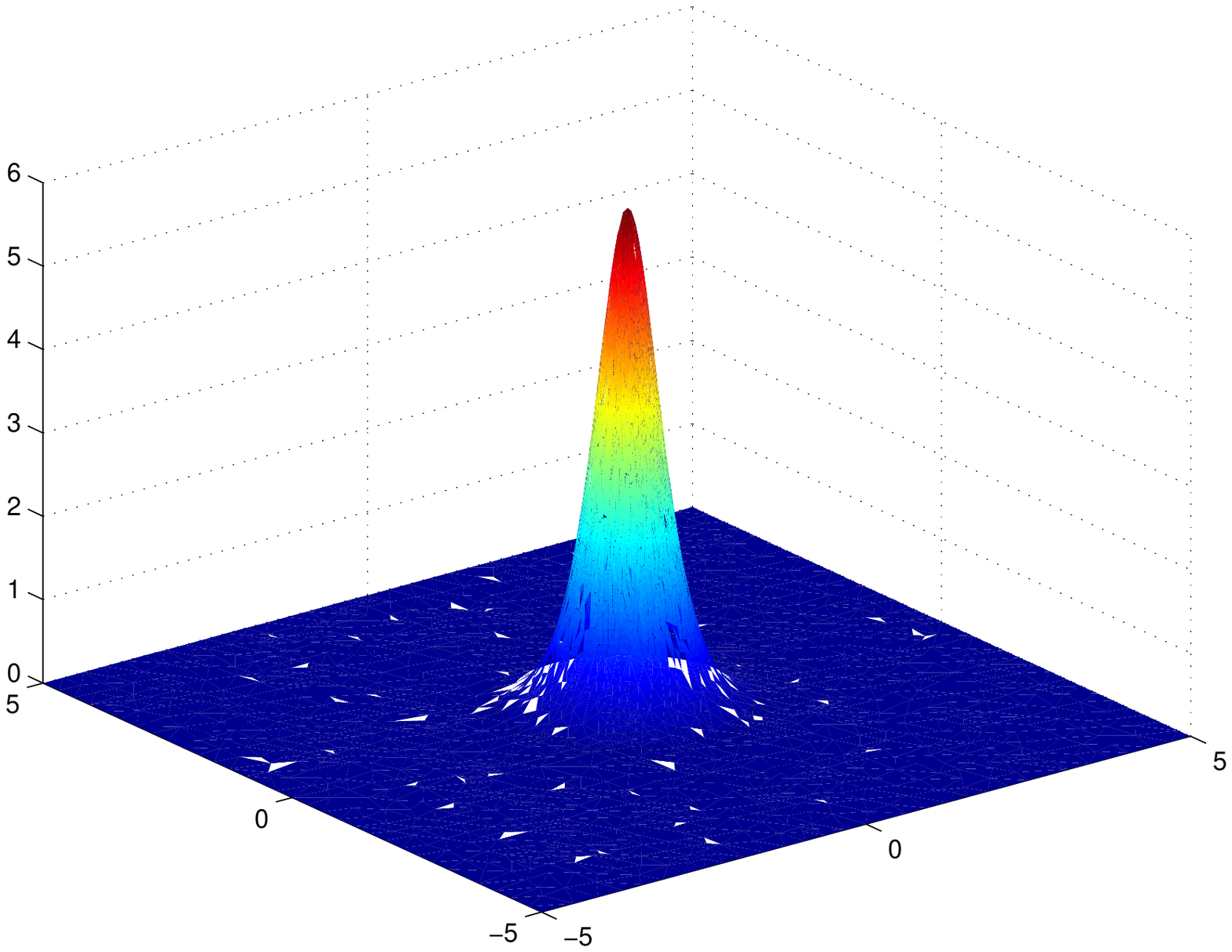}};
    \node at (2,-1) {%
      \includegraphics[width=0.25\linewidth]{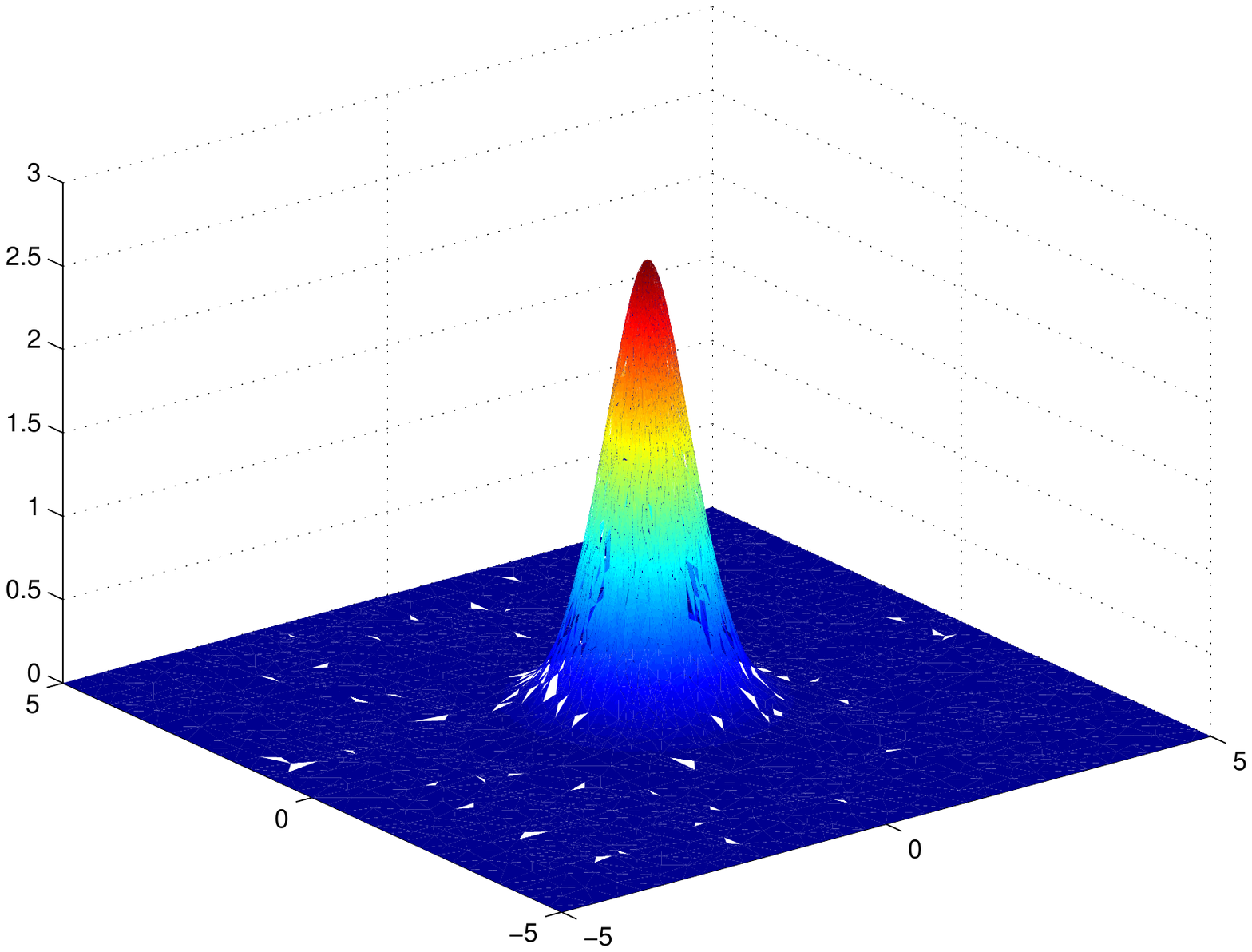}};
    \node at (2,0.5) {$R=10$};

    \node at (3,0) {%
      \includegraphics[width=0.25\linewidth]{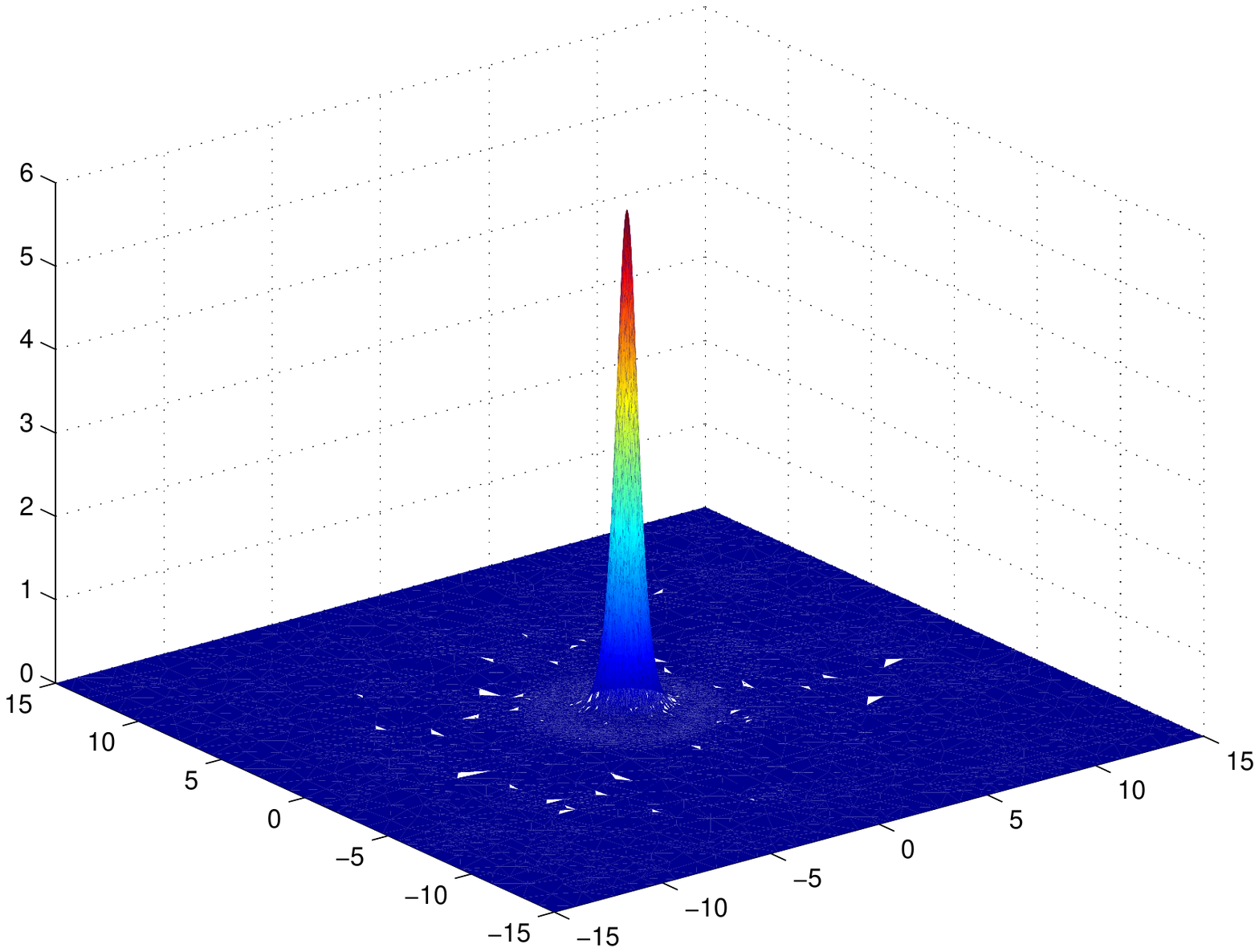}};
    \node at (3,-1) {%
      \includegraphics[width=0.25\linewidth]{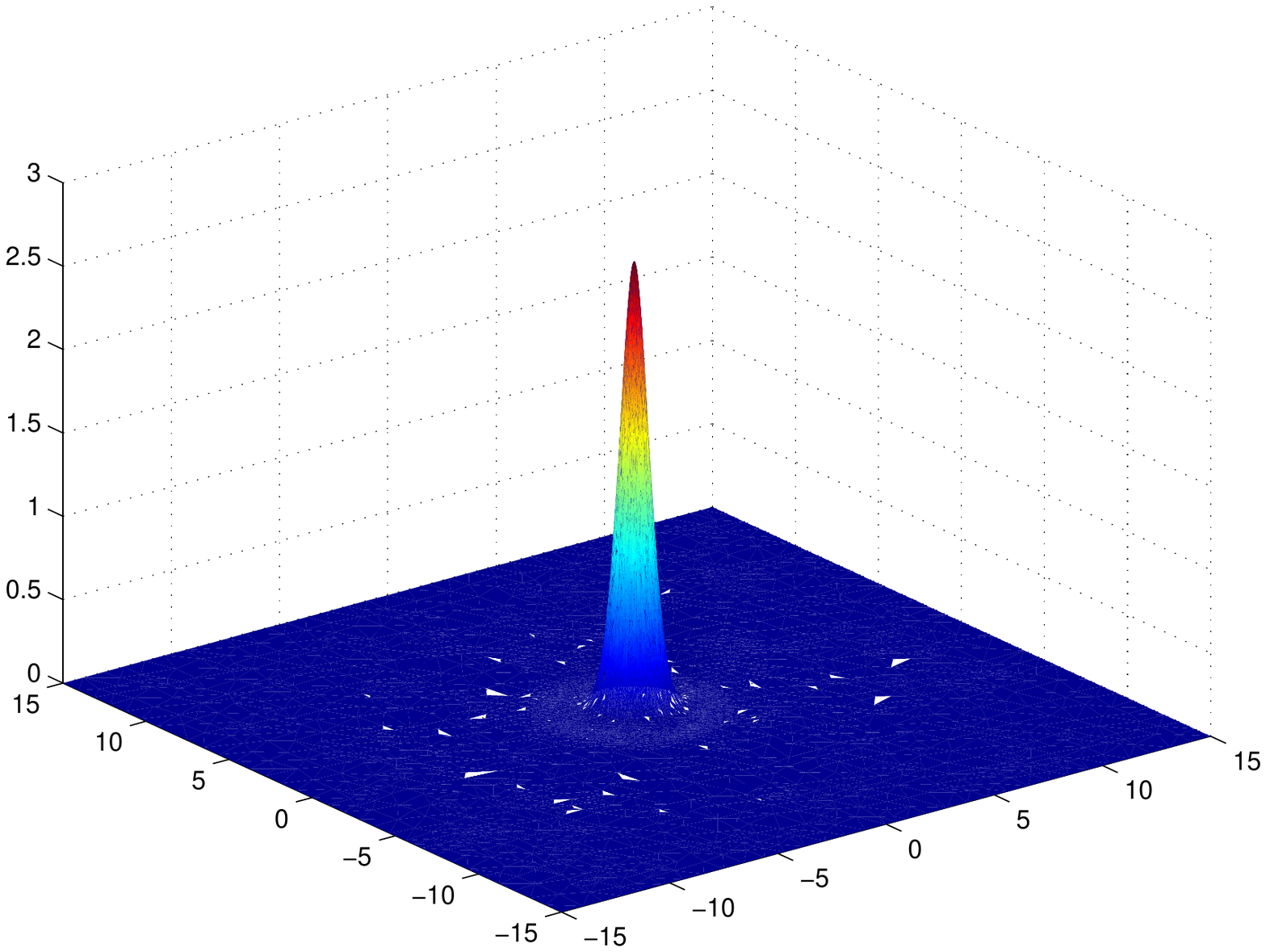}};
    \node at (3,0.5) {$R=30$};
  \end{tikzpicture}

  \vspace*{-8ex}%
  \caption{Sol.\ $v$ to~\eqref{exis1} and $u = r(v)$ to~\eqref{gs1}
    on $\Omega_R$ for $V = 10$ and $p=6$.}
  \label{fig:square,V=10,p=6}
\end{figure}

\begin{table}[ht]
  \begin{math}
    \begin{array}{r|cccc}
      & R = 1& R = 5& R = 10& R = 30\\
      \hline
      \max v&               10.57&  5.98&  5.98&  5.98\\
      \abs{\nabla v}_{L^2}& 19.56&  9.73&  9.73&  9.72\\
      \max u&               3.68&   2.68&  2.68&  2.68\\
      \mathcal T(v) =
      \mathcal E(u)&        105.1&  47.3&  47.3&  47.3\\
      \norm{\nabla\mathcal{T}(v)}&
      7\cdot 10^{-12}& 8.8\cdot 10^{-9}& 5.6\cdot 10^{-8}& 1.2\cdot 10^{-8}
    \end{array}
  \end{math}

  \vspace{1ex}
  \caption{Characteristics of approximate solutions $v$
    to~\eqref{exis1} and $u = r(v)$ to~\eqref{gs1} on $\Omega_R$
    for $V = 10$ and $p = 6$.}
  \vspace{-1ex}
  \label{tableValues,V=10,p=6}
\end{table}

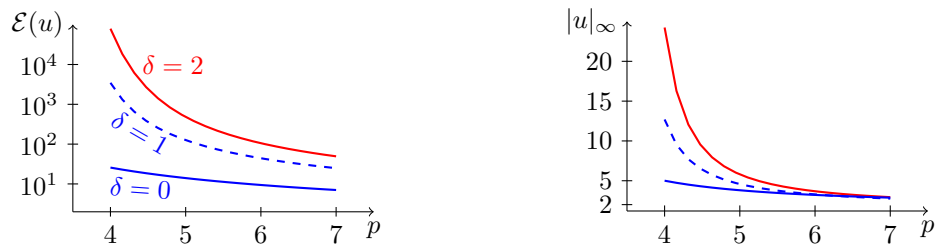
\begin{figure}[ht]
  \centering
  \newcommand{\xmin}{3.5}
  \newcommand{\ymin}{0.3}
  \begin{tikzpicture}[y=3.5ex]
    \draw[->] (\xmin, \ymin) -- (7.5, \ymin) node[below]{$p$};
    \foreach \p in {4, 5, 6, 7} {%
      \draw (\p, \ymin + 0.12) -- (\p, \ymin-0.12) node[below]{$\p$};
    }
    \draw[->] (\xmin, \ymin) -- (\xmin, 5) node[left]{$\E(u)$};
    \foreach \y in {1, 2, 3, 4} {%
      \draw (\xmin + 0.07, \y) -- (\xmin - 0.07, \y) node[left]{$10^\y$};
    }
    \draw[thick,color=red] plot file{pe_V10_R1.dat};
    \node[right,color=red] at (4.3,4) {$\delta=2$};
    \draw[thick,color=blue,dashed] plot file{pe_V10_R1-delta1.dat};
    \node[color=blue, rotate=-30] at (4.4, 2.3) {$\delta=1$};
    \draw[thick,color=blue] plot file{pe_V10_R1-delta0.dat};
    \node[color=blue, rotate=-4] at (4.4, 0.9) {$\delta=0$};
  \end{tikzpicture}
  \hspace{6em}%
  \renewcommand{\ymin}{1.2}%
  \begin{tikzpicture}[y=0.7ex]
    \draw[->] (\xmin, \ymin) -- (7.5, \ymin) node[below]{$p$};
    \foreach \p in {4, 5, 6, 7} {%
      \draw (\p, \ymin + 0.7) -- (\p, \ymin - 0.7) node[below]{$\p$};
    }
    \draw[->] (\xmin, \ymin) -- (\xmin, 25) node[left]{$\abs{u}_\infty$};
    \foreach \y in {2, 5, 10,..., 23} {%
      \draw (\xmin + 0.07, \y) -- (\xmin - 0.07, \y) node[left]{$\y$};
    }
    \draw[thick,color=red] plot file{pn_V10_R1.dat};
    \draw[thick,color=blue,dashed] plot file{pn_V10_R1-delta1.dat};
    \draw[thick,color=blue] plot file{pn_V10_R1-delta0.dat};
  \end{tikzpicture}

  \vspace*{-2ex}
  \caption{Comparison between~\eqref{gs1} with $V=0$ and the problem
    $-\Delta u + 10u= \abs{u}^{p-1}u$ on $(-0.5,0.5)^2$.}
  \label{comparison,V=10}
\end{figure}

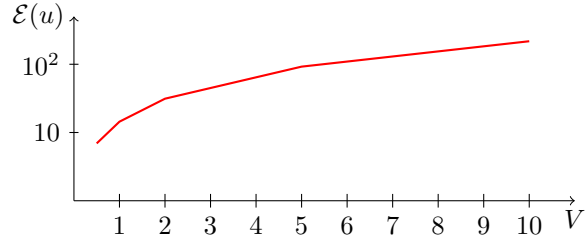
\begin{figure}[ht]
  \centering
  \begin{tikzpicture}[x=4ex, y=6ex]
    \draw[->] (0,0) -- (11, 0) node[below]{$V$};
    \draw[->] (0,0) -- (0, 2.7) node[left]{$\mathcal{E}(u)$};
    \foreach \x in {1,...,10}{
      \draw (\x, 0.1) -- (\x, -0.1) node[below]{$\x$};
    }
    \foreach \y/\l in {1/10,2/10^2}{
      \draw (0.1, \y) -- (-0.1,\y) node[left]{$\l$};
    }
    \draw[thick,color=red] plot file{bifurcation.dat};
  \end{tikzpicture}
  \caption{Energy of the solutions to~\eqref{gs1} on $\Omega_{10}$
    as $V \to 0$.}
  \label{V->0,p=4}
\end{figure}

\subsection{Exponential nonlinearity}

Since our numerical experiments are performed in two dimensions, it
is natural to wonder what shape the solutions are expected to have
when the nonlinearity has exponential growth.  
To that aim, we consider in this
section the equation
\begin{equation}
  \label{eq:exp}
  -\Delta u - u \Delta u^2 + V u = \exp(u^2) - 1. 
\end{equation}
The outcome of the numerical computations is presented on
Figures~\ref{fig:square,V=0,exp}--\ref{fig:square,V=10,exp}
and Tables~\ref{tableValues,V=0,exp}--\ref{tableValues,V=10,exp}.
Unsurprisingly, one may see that the qualitative behavior of the
solutions is similar to what could be observed before with power
nonlinearities.

\begin{figure}[ht]
  \vspace*{-6ex}%
  \begin{tikzpicture}[x=0.25\linewidth, y=0.18\linewidth]
    \node at (0,0) {%
      \includegraphics[width=0.25\linewidth]{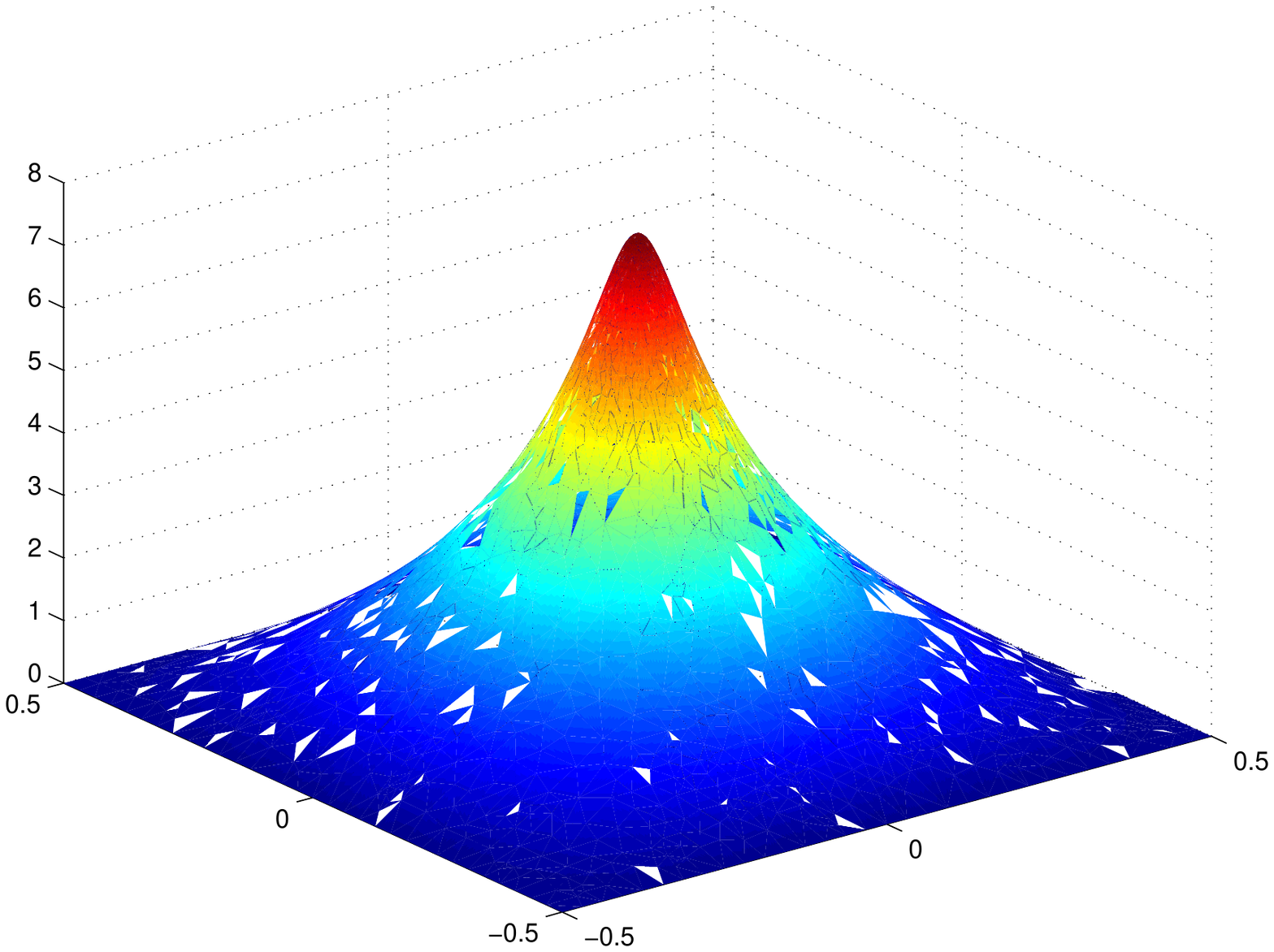}};
    \node at (0,-1) {%
      \includegraphics[width=0.25\linewidth]{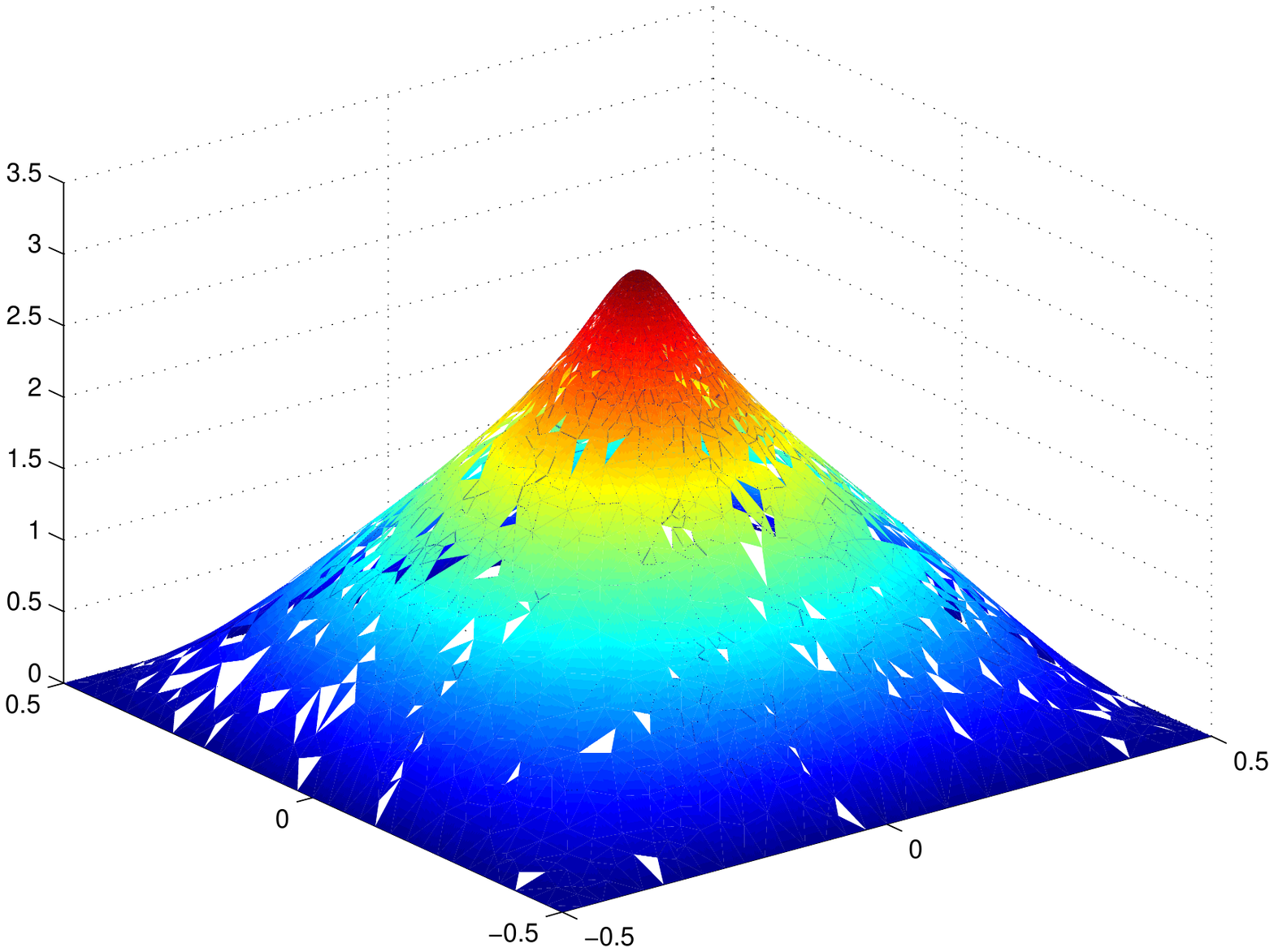}};
    \node at (0,0.5) {$R = 1$};
    \node[left] at (-0.45, 0) {$v$};
    \node[left] at (-0.45, -1) {\rotatebox{90}{$u = r(v)$}};

    \node at (1,0) {%
      \includegraphics[width=0.25\linewidth]{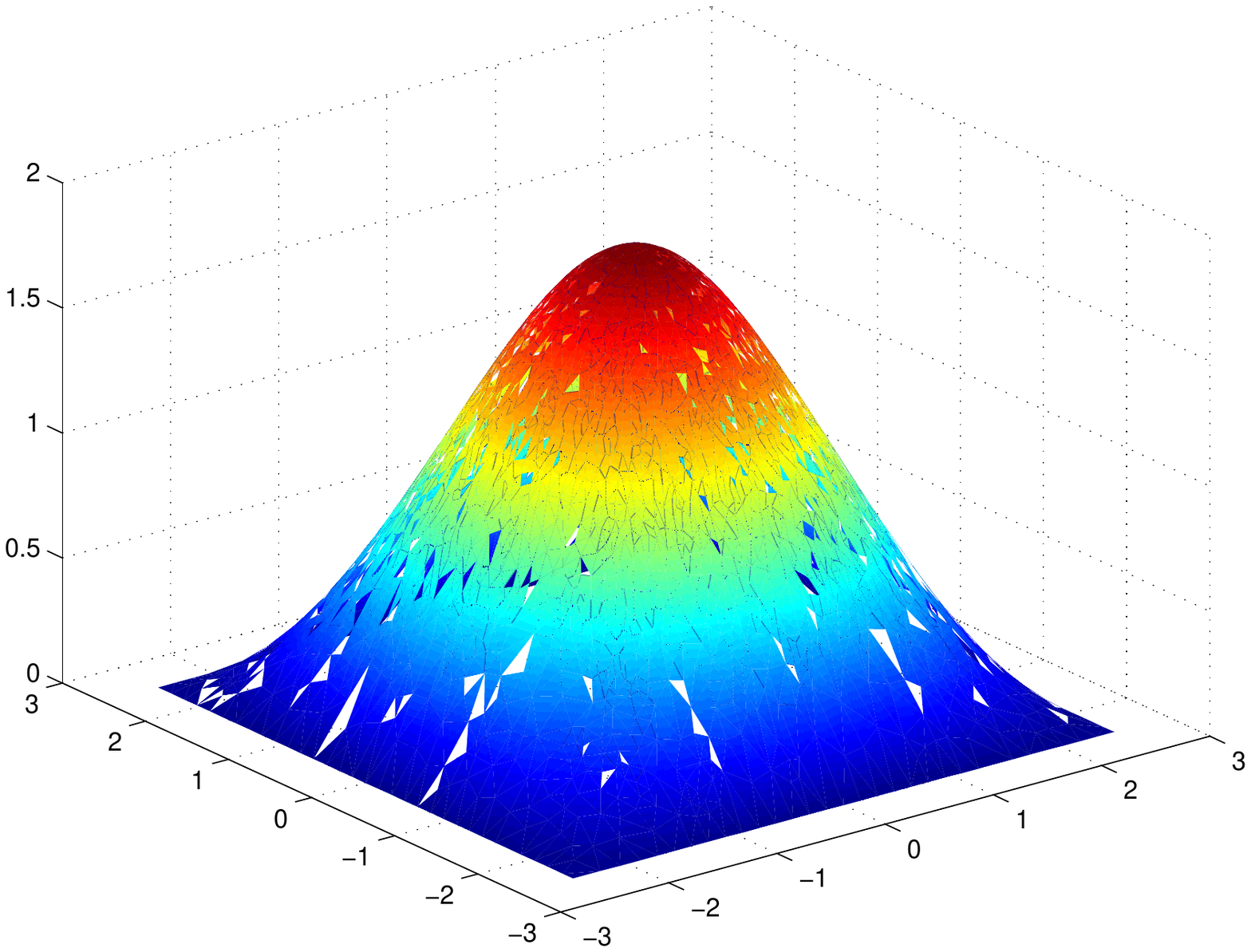}};
    \node at (1,-1) {%
      \includegraphics[width=0.25\linewidth]{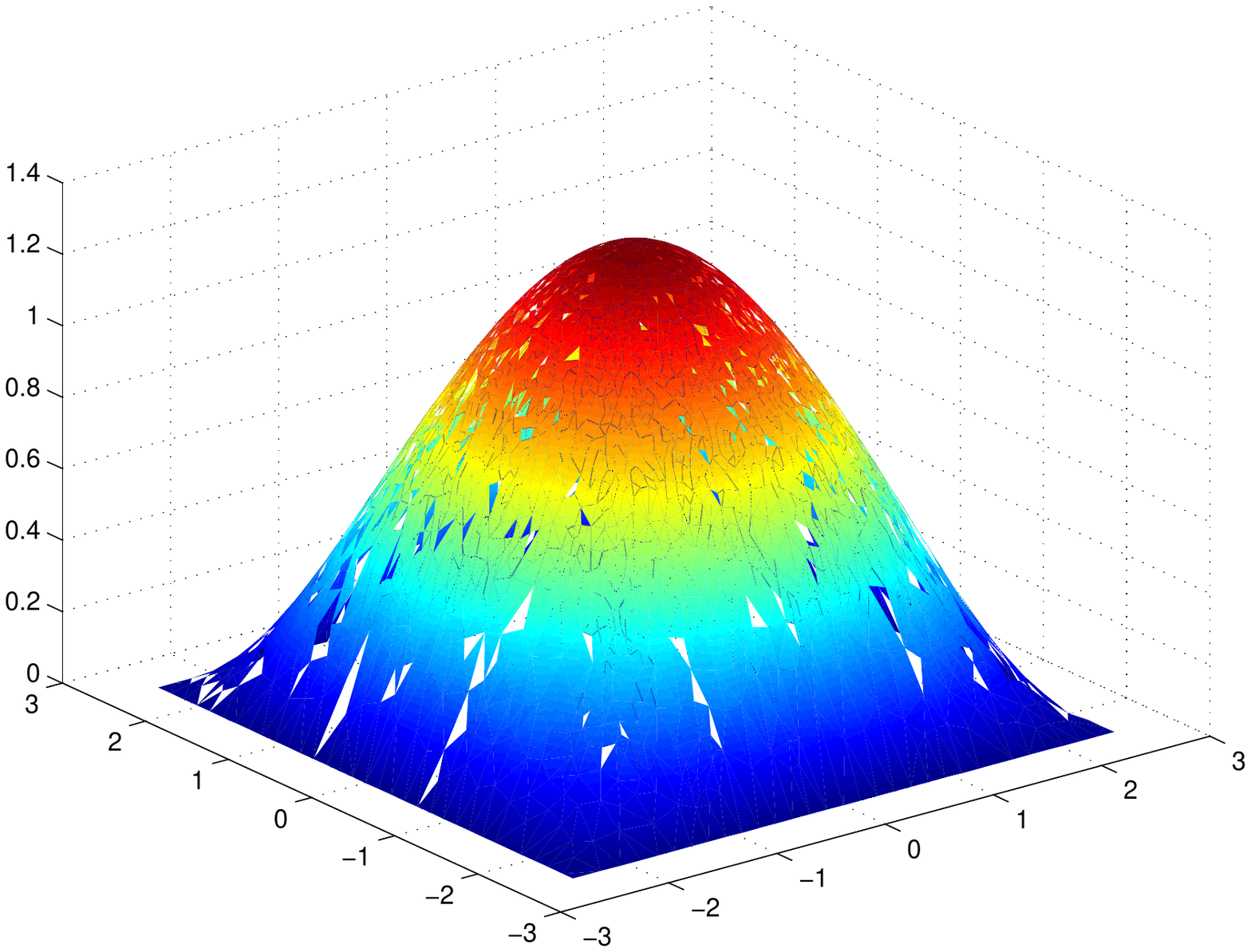}};
    \node at (1,0.5) {$R = 5$};

    \node at (2,0) {%
      \includegraphics[width=0.25\linewidth]{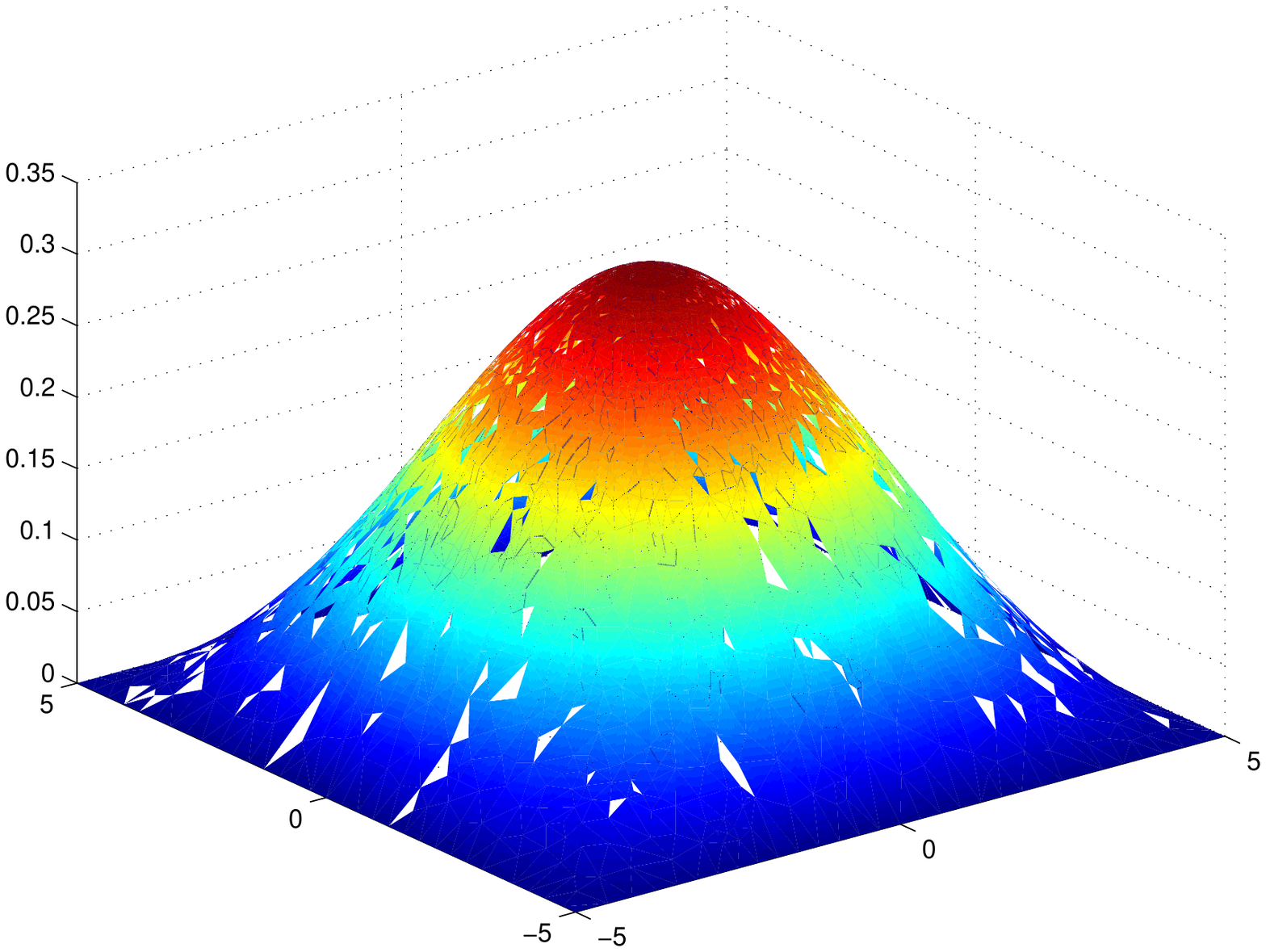}};
    \node at (2,-1) {%
      \includegraphics[width=0.25\linewidth]{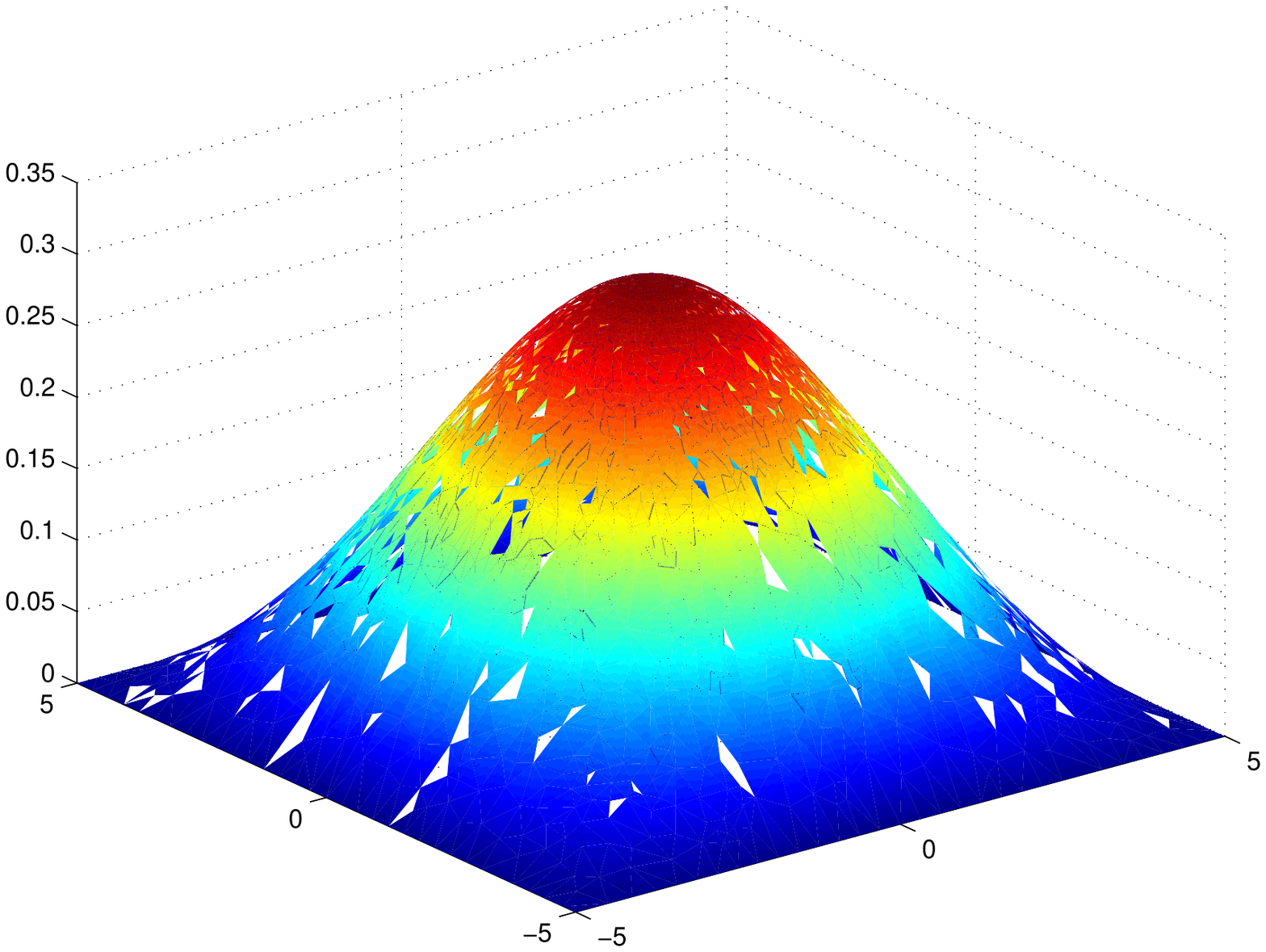}};
    \node at (2,0.5) {$R = 10$};

    \node at (3,0) {%
      \includegraphics[width=0.25\linewidth]{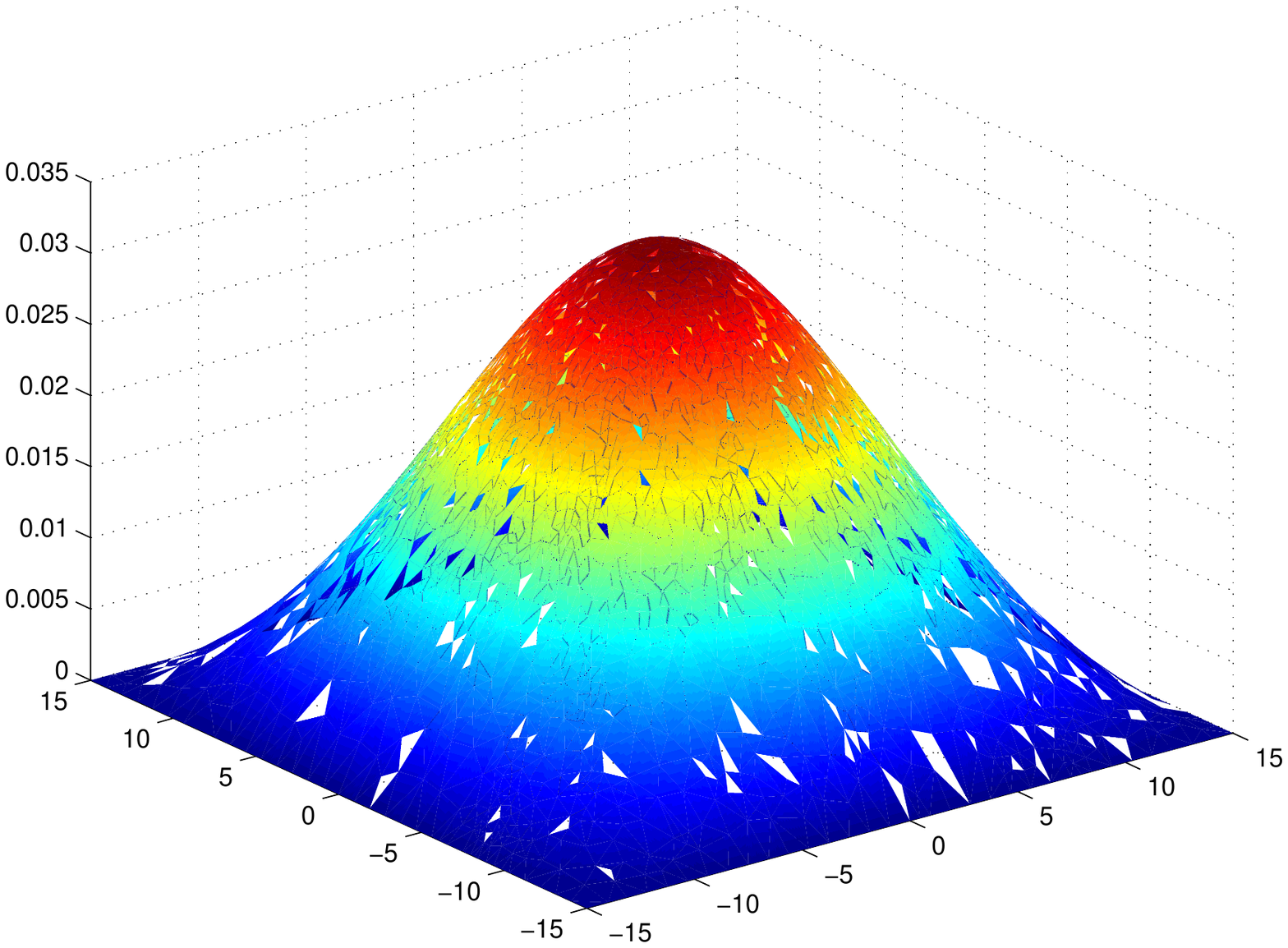}};
    \node at (3,-1) {%
      \includegraphics[width=0.25\linewidth]{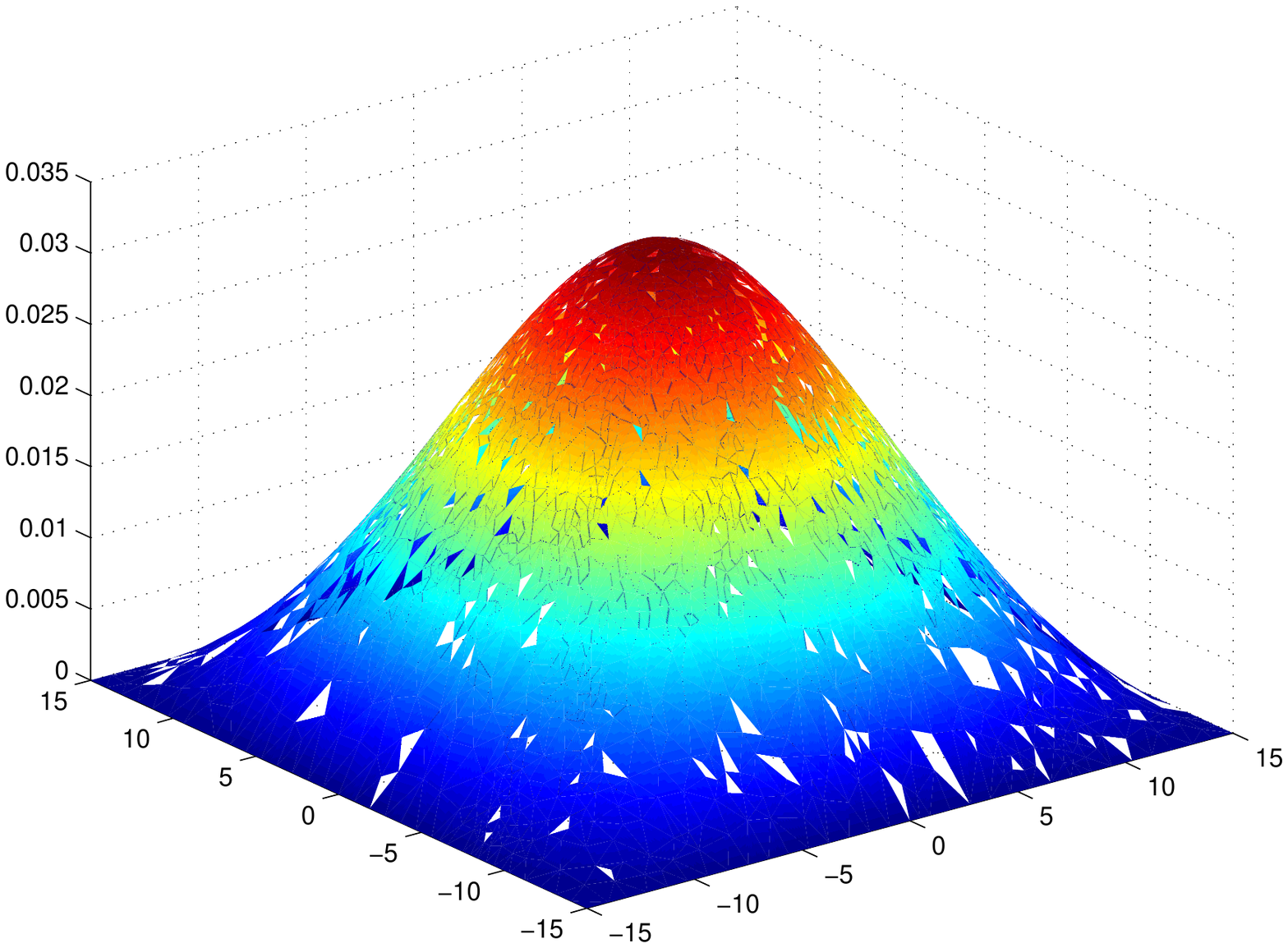}};
    \node at (3,0.5) {$R = 30$};
  \end{tikzpicture}

  \vspace*{-8ex}%
  \caption{Sol.\ to~\eqref{eq:exp} on $\Omega_R$ for $V=0$.}
  \label{fig:square,V=0,exp}
\end{figure}

\begin{table}[ht]
  \begin{math}
    \begin{array}{r|cccc}
      & R = 1& R = 5& R = 10& R = 30\\
      \hline
      \max v&                7.6&  1.85&  0.31&  0.03\\
      \abs{\nabla v}_{L^2}&   11&  3.76&  0.63&  0.065\\
      \max u&                3.1&  1.30&  0.30&  0.03\\
      \mathcal T(v) =
      \mathcal E(u)&        44.4&  2.02&  0.06&  0.0007\\
      \norm{\nabla\mathcal{T}(v)}&
      3\cdot 10^{-8}& 7\cdot 10^{-11}& 1.7\cdot 10^{-7}& 1.3\cdot 10^{-9}
    \end{array}
  \end{math}

  \vspace{1ex}
  \caption{Characteristics of approximate solutions $v$ and $u =
    r(v)$ to~\eqref{eq:exp} on $\Omega_R$ for $V =0$.}
  \vspace{-1ex}
  \label{tableValues,V=0,exp}
\end{table}

\begin{figure}[ht]
  \vspace*{-6ex}%
  \begin{tikzpicture}[x=0.25\linewidth, y=0.18\linewidth]
    \node at (0,0) {%
      \includegraphics[width=0.25\linewidth]{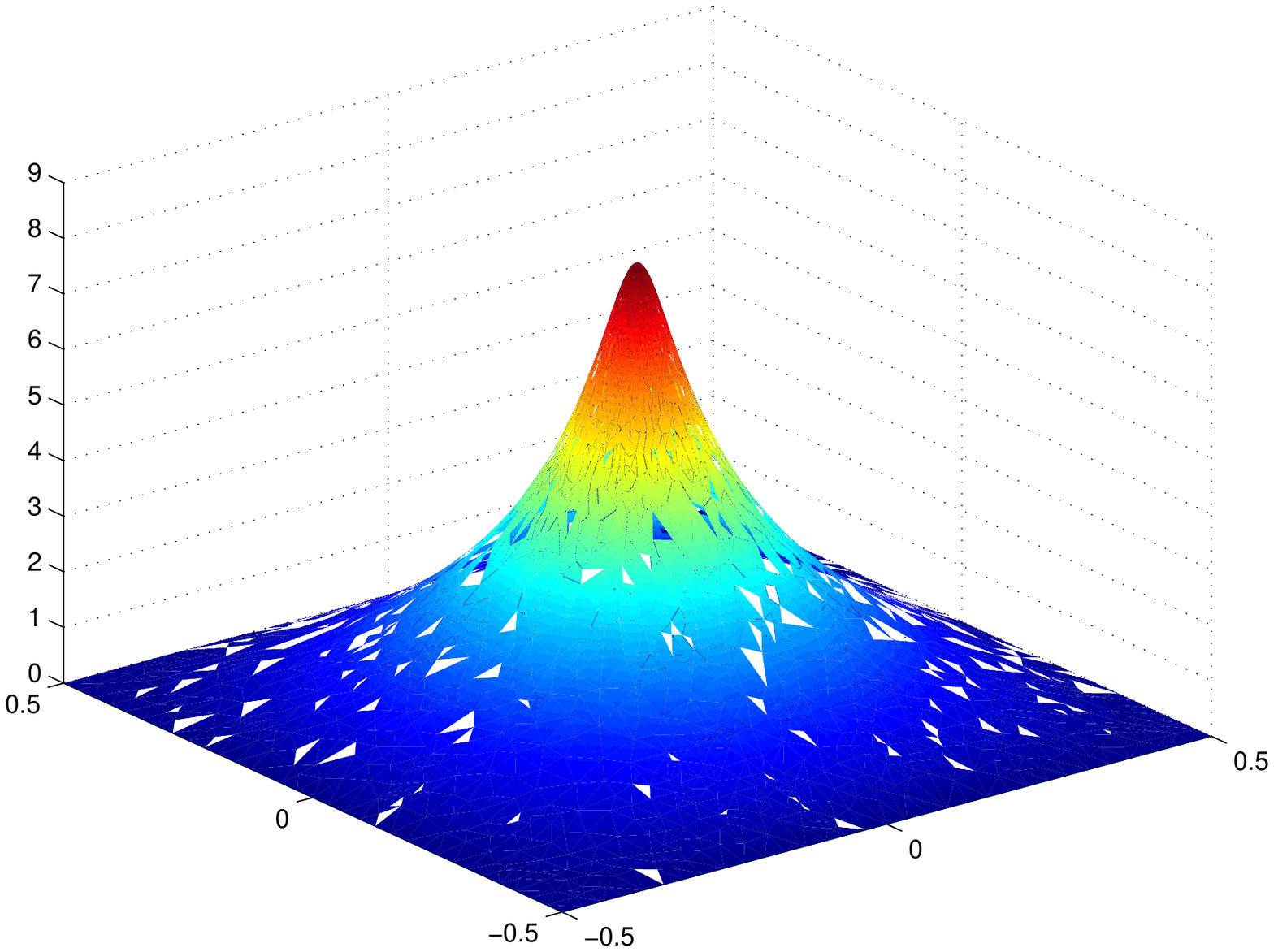}};
    \node at (0,-1) {%
      \includegraphics[width=0.25\linewidth]{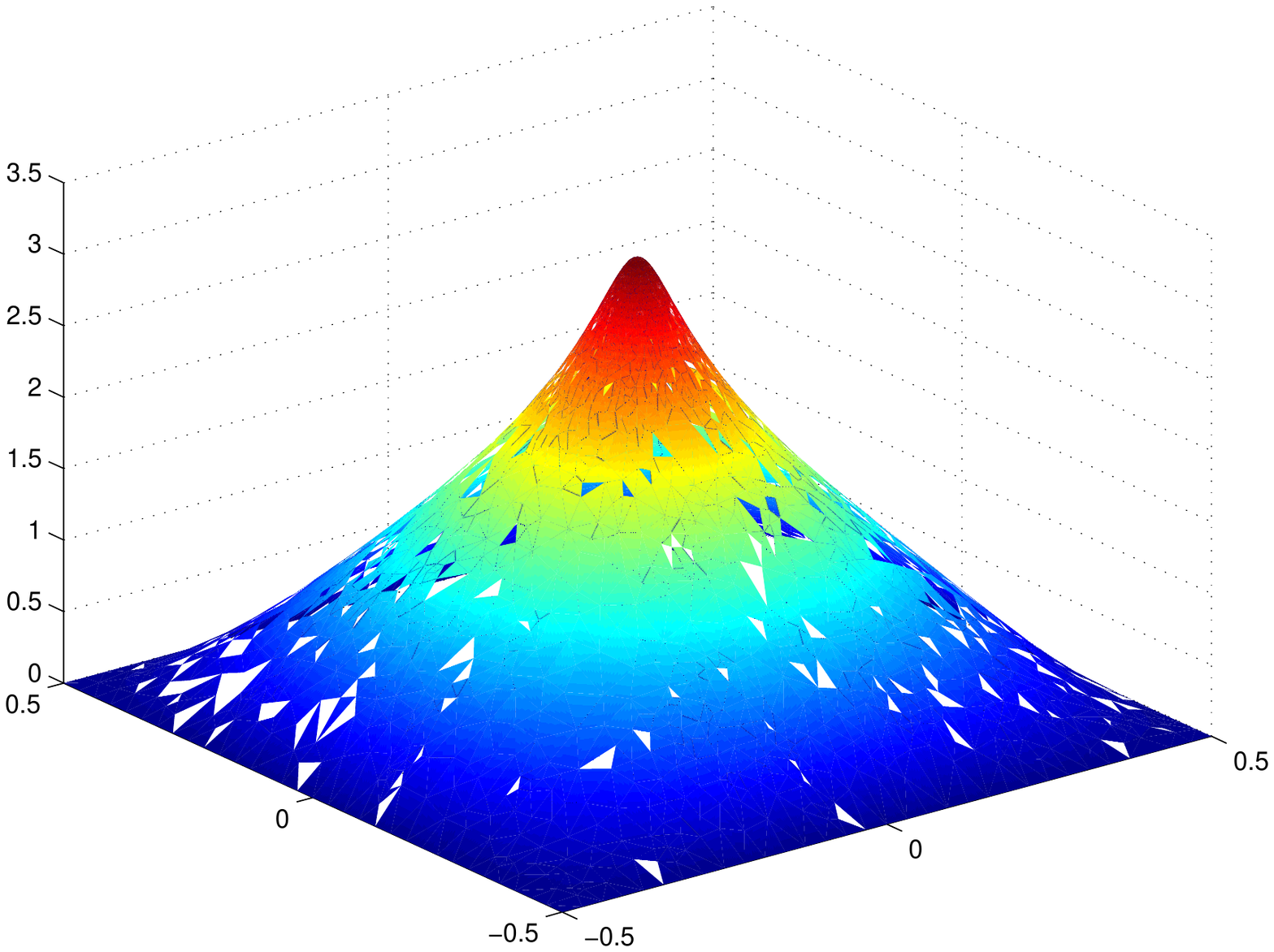}};
    \node at (0,0.5) {$R = 1$};
    \node[left] at (-0.45, 0) {$v$};
    \node[left] at (-0.45, -1) {\rotatebox{90}{$u = r(v)$}};

    \node at (1,0) {%
      \includegraphics[width=0.25\linewidth]{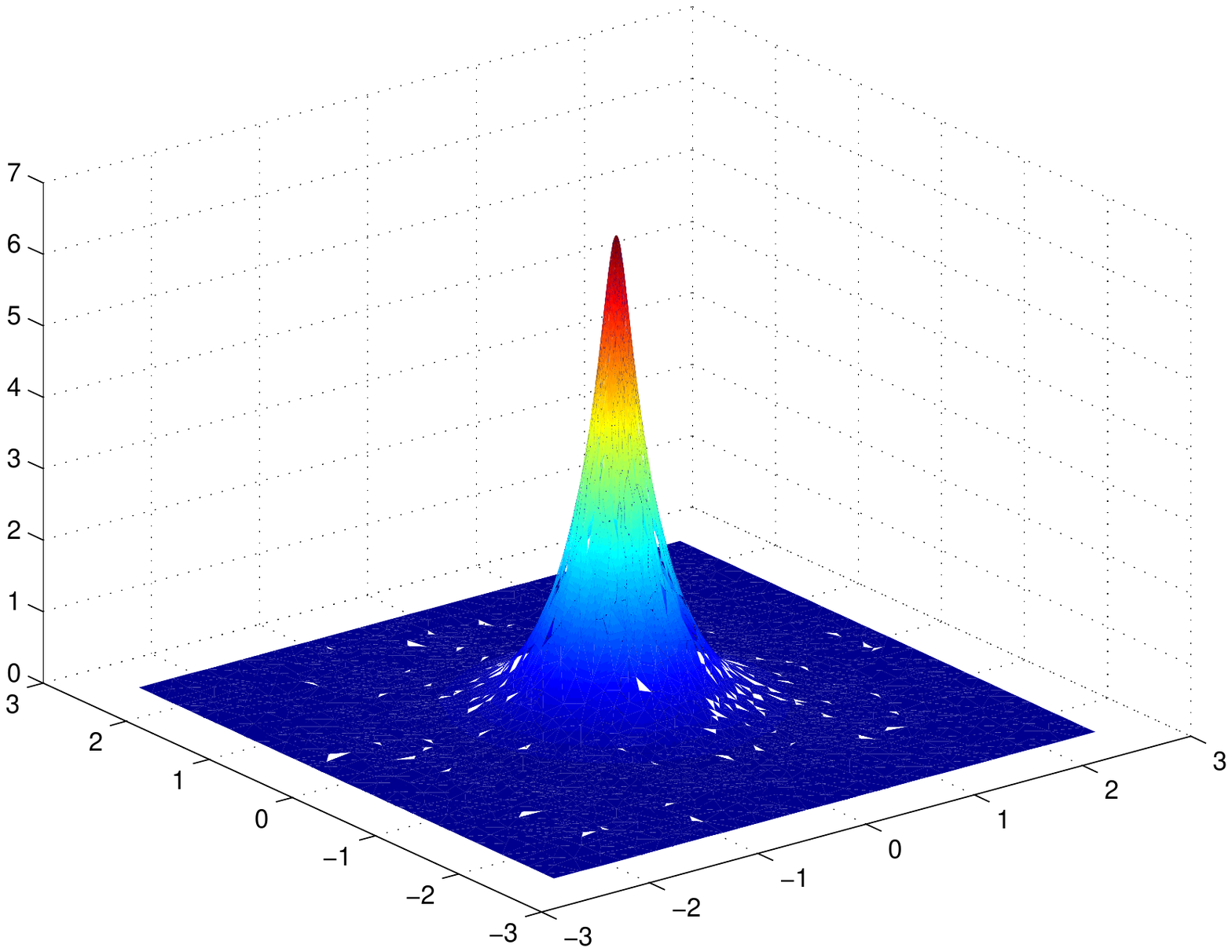}};
    \node at (1,-1) {%
      \includegraphics[width=0.25\linewidth]{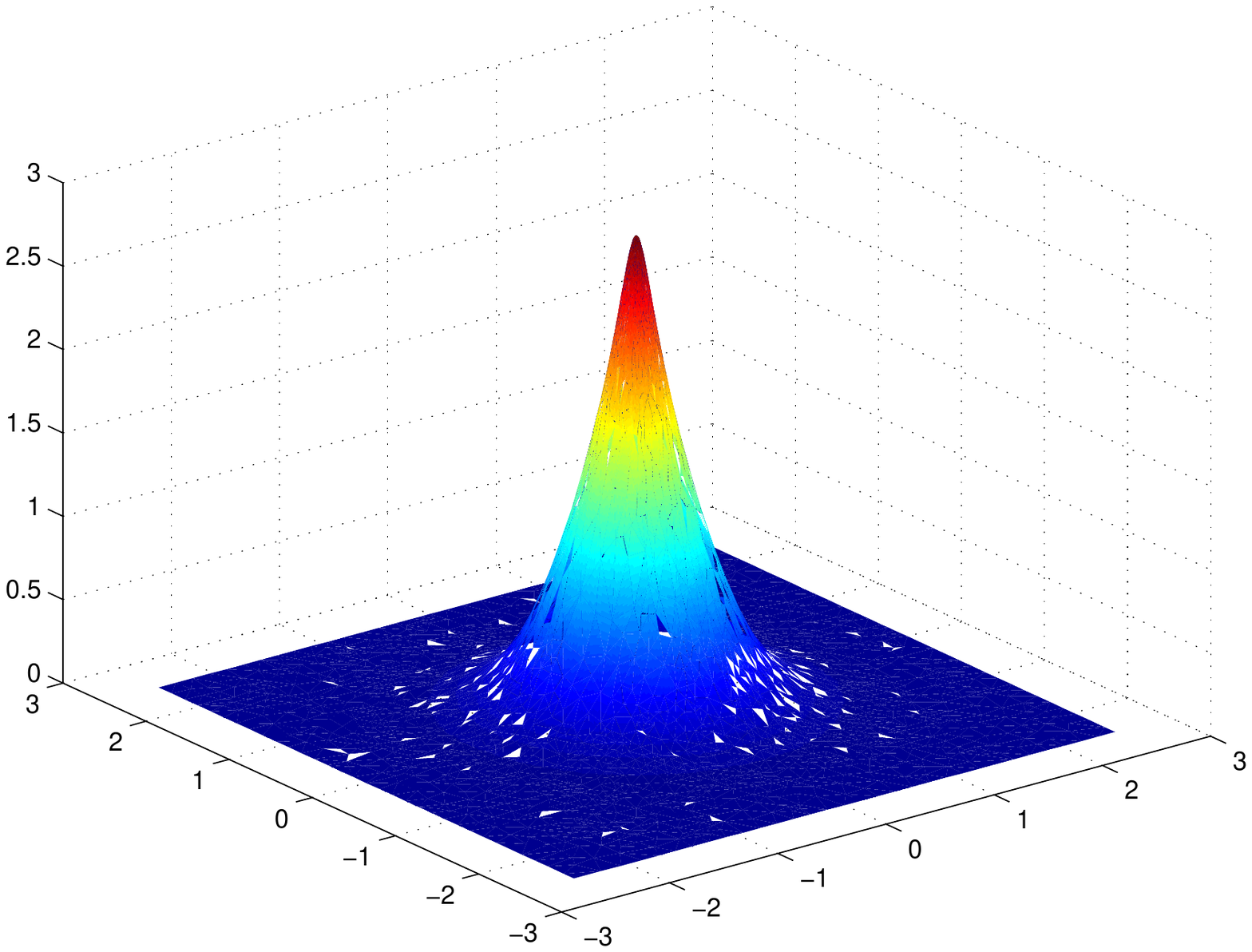}};
    \node at (1,0.5) {$R = 5$};

    \node at (2,0) {%
      \includegraphics[width=0.25\linewidth]{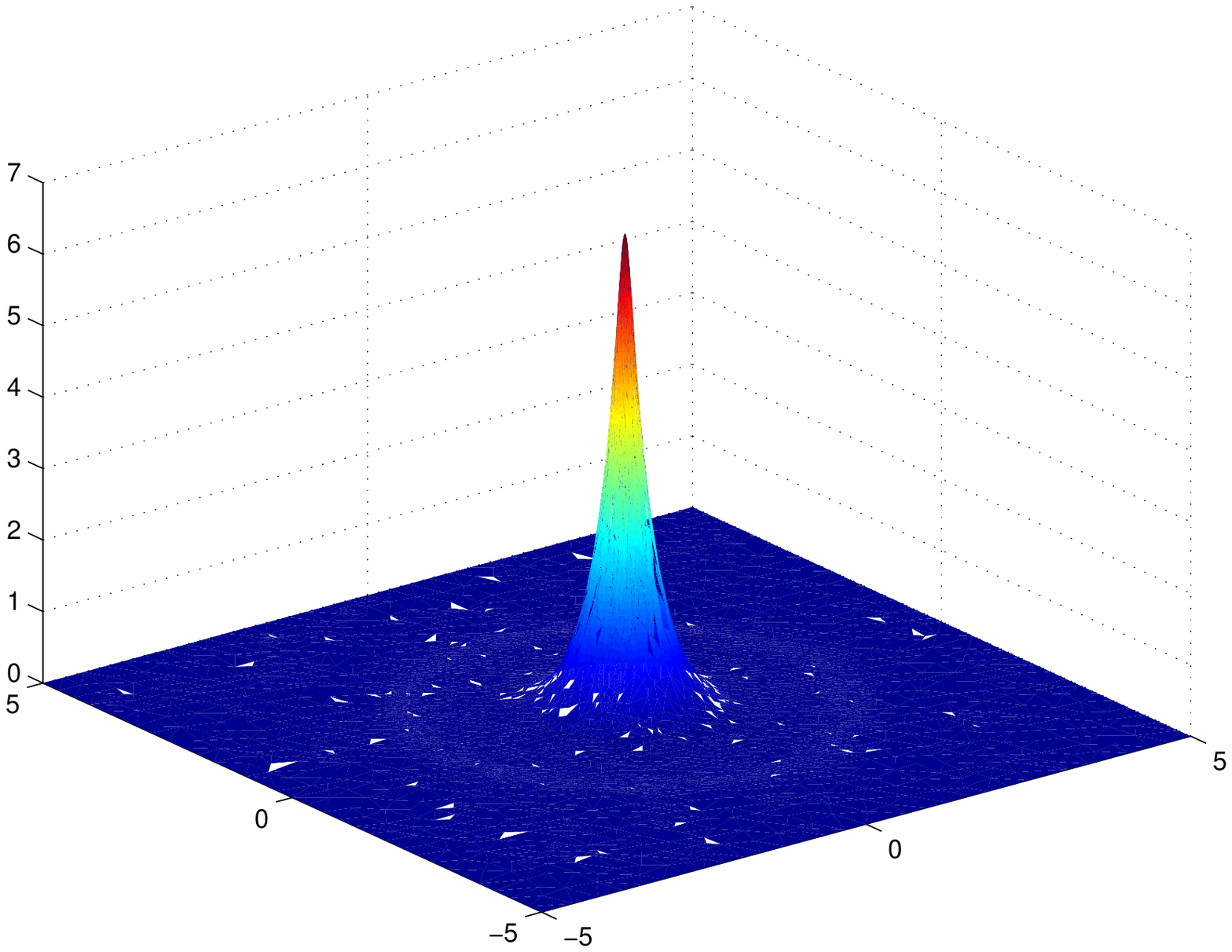}};
    \node at (2,-1) {%
      \includegraphics[width=0.25\linewidth]{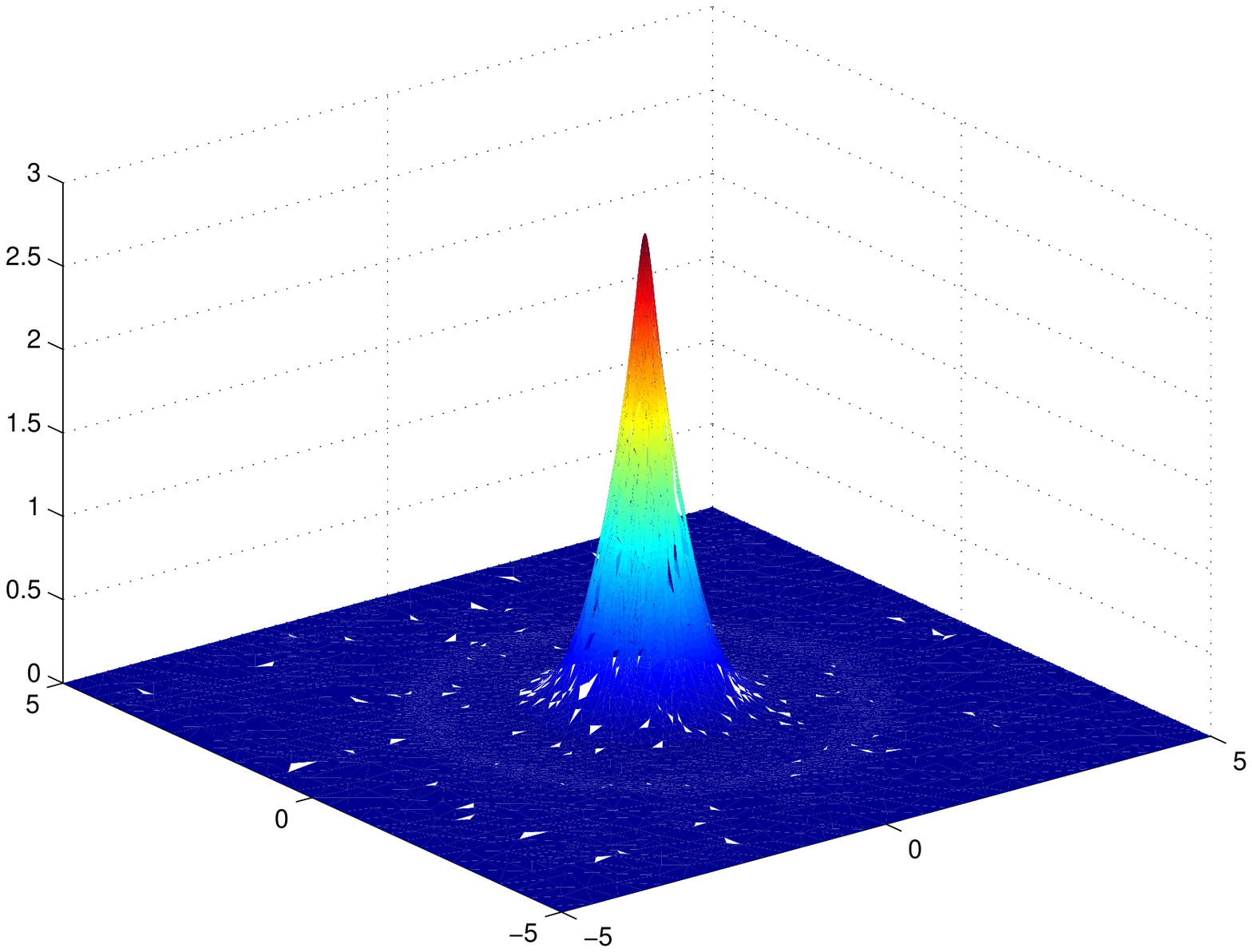}};
    \node at (2,0.5) {$R = 10$};

    \node at (3,0) {%
      \includegraphics[width=0.25\linewidth]{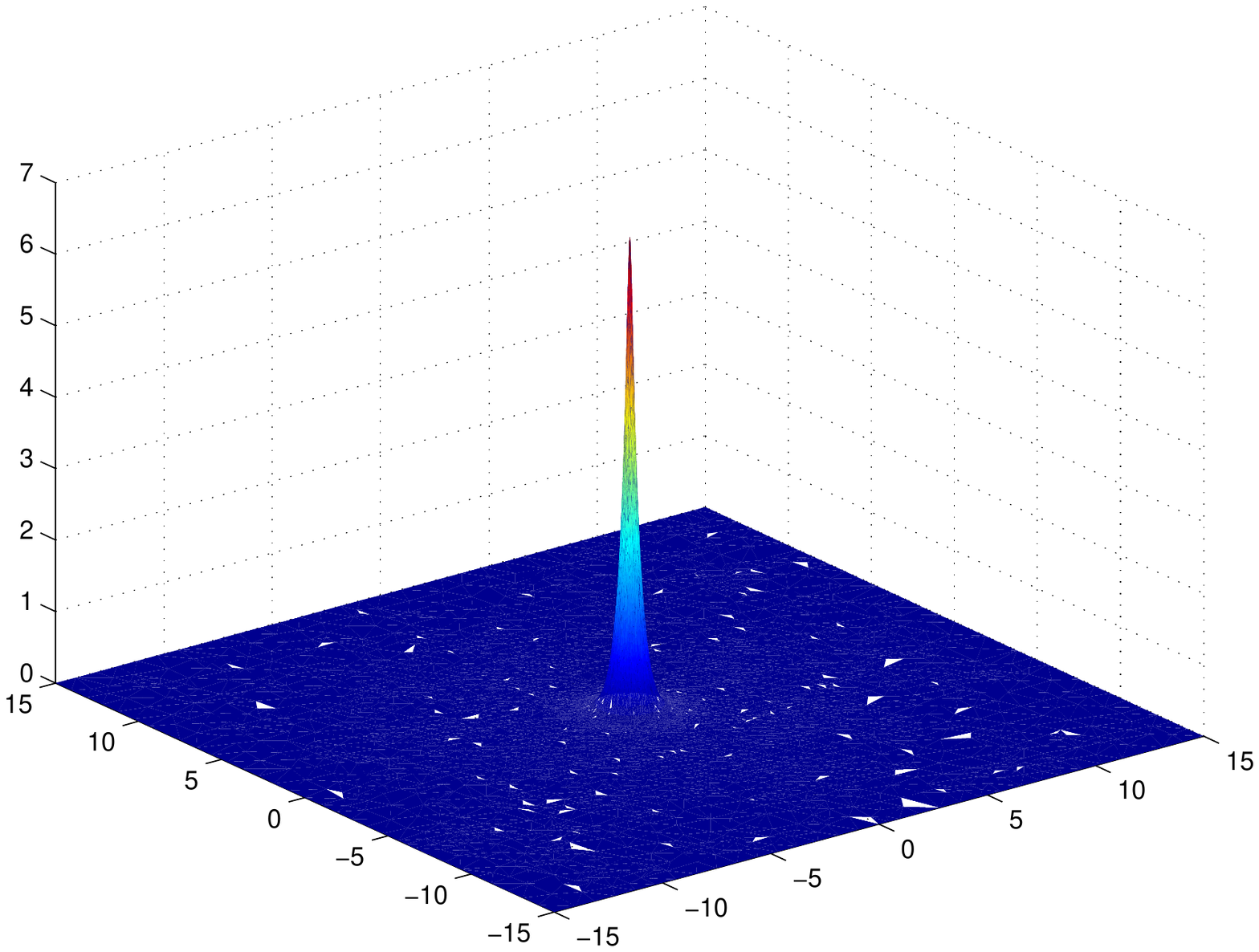}};
    \node at (3,-1) {%
      \includegraphics[width=0.25\linewidth]{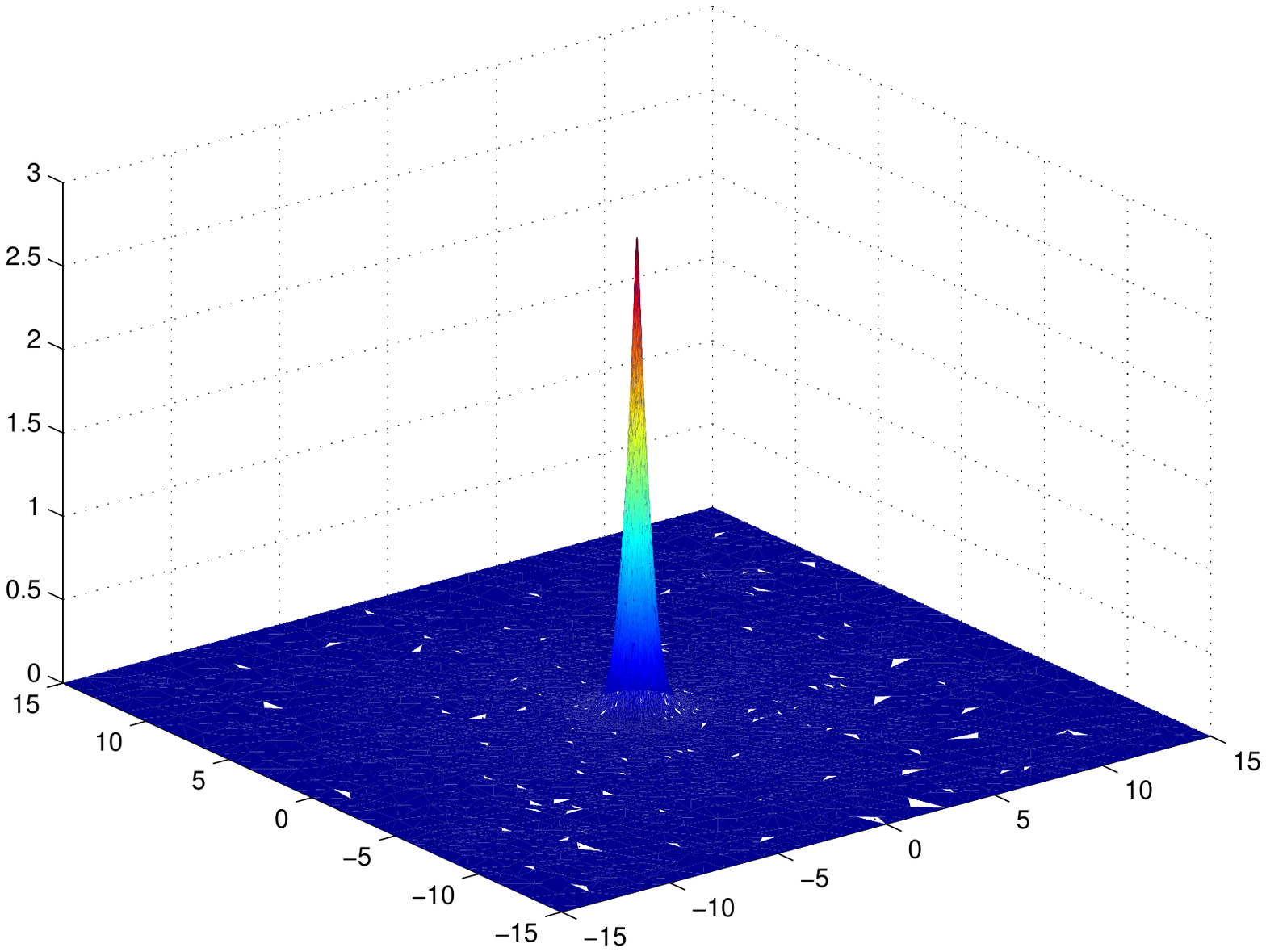}};
    \node at (3,0.5) {$R = 30$};
  \end{tikzpicture}

  \vspace*{-8ex}%
  \caption{Sol.\ to~\eqref{eq:exp} on $\Omega_R$ for $V=10$.}
  \label{fig:square,V=10,exp}
\end{figure}

\begin{table}[ht]
  \begin{math}
    \begin{array}{r|cccc}
      & R = 1& R = 5& R = 10& R = 30\\
      \hline
      \max v&                8.04&  6.63& 6.63&  6.62\\
      \abs{\nabla v}_{L^2}&  10.98& 9.10& 9.10&  9.10\\
      \max u&                3.17&  2.84& 2.84&  2.84\\
      \mathcal T(v) =
      \mathcal E(u)&         50.60& 41.47& 41.45&  41.40\\
      \norm{\nabla\mathcal{T}(v)}&
      3.6\cdot 10^{-7}& 8.6\cdot 10^{-8}& 8.5\cdot 10^{-9}& 9\cdot 10^{-3}
    \end{array}
  \end{math}

  \vspace{1ex}
  \caption{Characteristics of the approximate solutions $v$ and $u =
    r(v)$ to~\eqref{eq:exp} for $V =10$.}
  \vspace{-1ex}
  \label{tableValues,V=10,exp}
\end{table}

\subsection{Non-constant potentials}
\label{num:non-constant V}

To conclude this investigation, let us examine the case of a variable
potential.  E.~Gloss~\cite{Gloss2010} in dimension $N \ge 3$ and
J.~M.~do \'O and U.~Severo~\cite{O-Severo2010} for $N = 2$ showed that
the equation
\begin{equation}
  \label{eq:concentration}
  -\eps^2 (\Delta u + u \Delta u^2) + V(x) u = g(u)
  \quad\text{in } \R^N
\end{equation}
possesses solutions concentrating around local minima of
the potential $V$ when $\eps \to 0$.
When the potential is constant, this corresponds to $R = 1/\eps$
and, for a positive potential, the above graphs show a concentration
around the origin on a bounded domain (which suggests that
any point will do on $\R^N$).
In this section we have considered~\eqref{eq:concentration}
on~$(-0.5,0.5)^2$ with $g(u) = \abs{u}^3 u$ and the double well
potential
\begin{equation*}
  V(x) := 10 - 8 \exp\bigl(-20\, \abs{x - c}^2\bigr)
  - 5 \exp\bigl(-30\, \abs{x - c'}^2\bigr)
\end{equation*}
where $c = (-0.2, 0.2)$ and $c' = (0.3, -0.2)$.  It is pictured on
Fig.~\ref{fig:square,Vx,p=4}.  Note that the two wells have different
depths.  For $\eps = 0.05$, the MPA returns two different solutions
(see Fig.~\ref{fig:square,Vx,p=4}) depending on the initial function.
(We have chosen not to display the function $v$ as its shape is
similar to the one of $u$.)
The one with the lower value of $\mathcal T$ (hence $\mathcal E$) is 
the solution located around the lower well of~$V$.
It is obtained by using the MPA with the same initial guess as
before.  For the right one, the MPA is applied starting with a
function localized at the other well, namely
\begin{equation}
  \label{eq:guess-loc}
  (x_1, x_2) \mapsto \max\bigl\{0,\
  0.1 - (x_1 - c'_1)^2 - (x_2 - c'_2)^2 \bigr\} .
\end{equation}
For larger $\eps$, such as $\eps = 0.25$,
the MPA returns only one solution with a maximum point
not too far from the point at which $V$ achieves its global minimum,
see the rightmost graph on Fig.~\ref{fig:square,Vx,p=4}
(the graph depicted is the outcome of the MPA with the usual initial
function but even using the function~\eqref{eq:guess-loc}
as a starting point gives the same output).
For even larger $\eps$, such as $\eps=1$, the solution is very similar
to the one for $R=1$ displayed on Fig.~\ref{fig:square,V=10,p=4}.
\begin{figure}[h!t]
  \vspace*{-6ex}%
  \begin{tikzpicture}[x=0.25\linewidth, y=0.18\linewidth]
    \node at (0, -0.1) {%
      \includegraphics[width=0.25\linewidth]{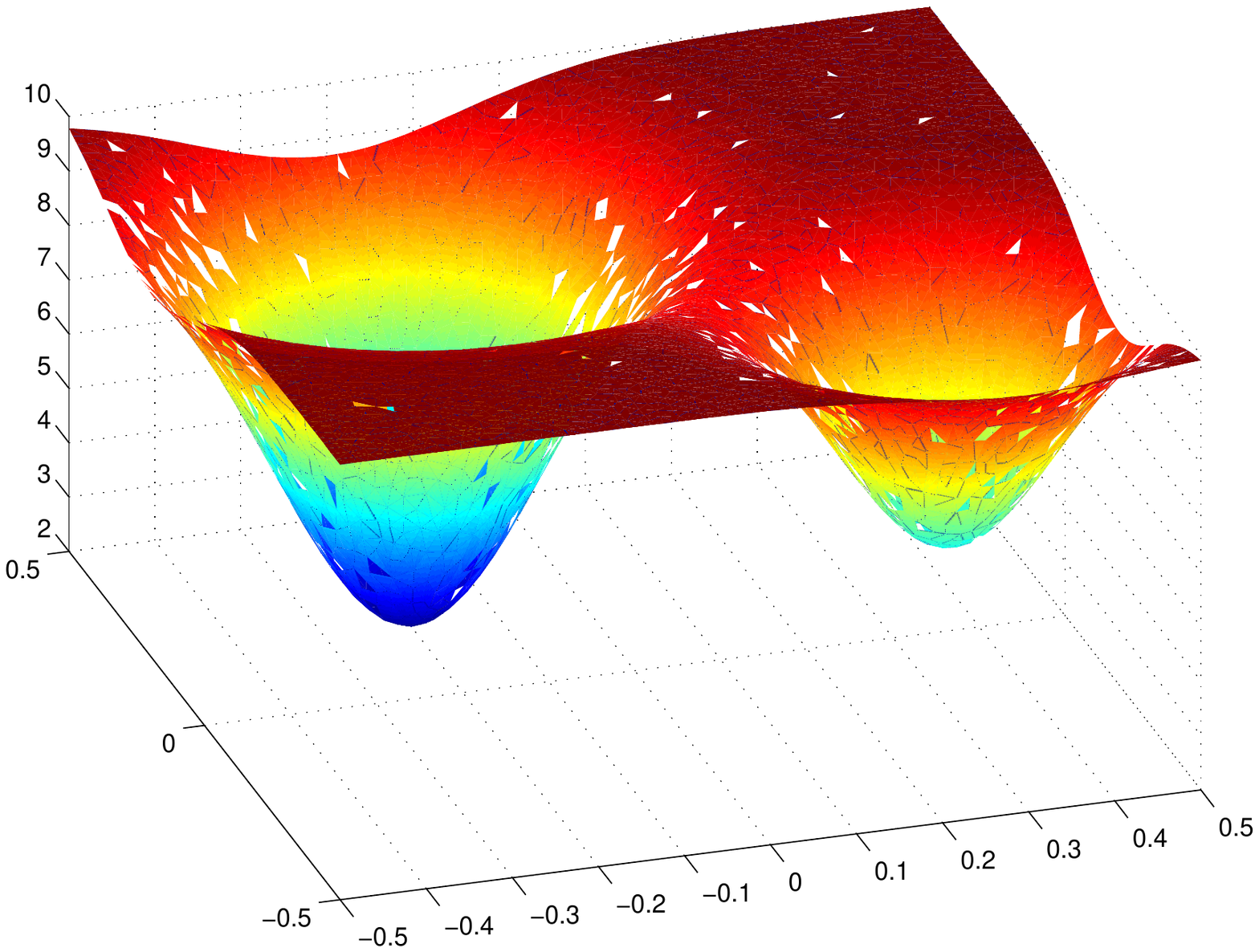}};
    \node at (0,0.5) {Potential $V$};

    \node at (1,0) {%
      \includegraphics[width=0.25\linewidth]{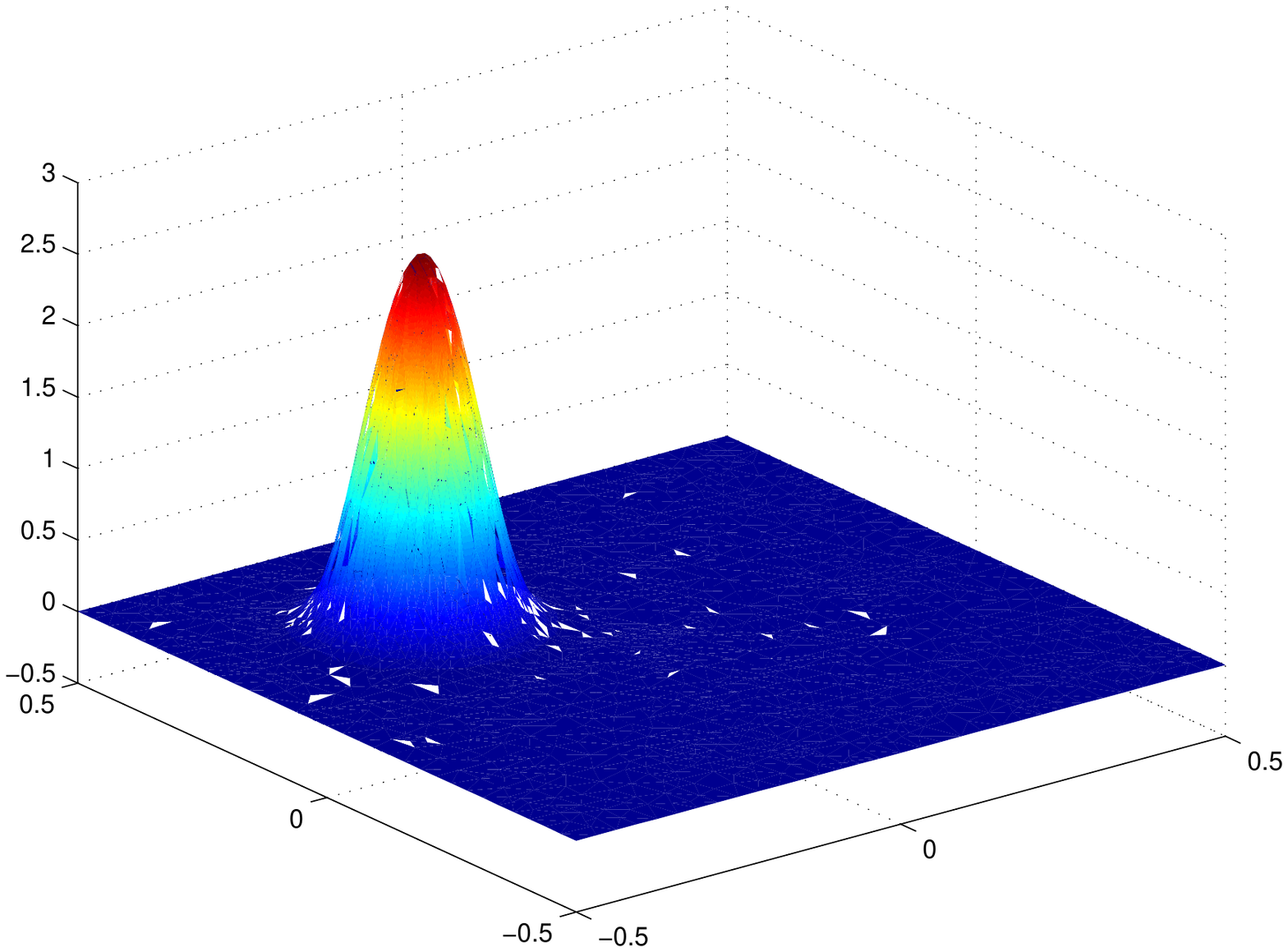}};
    \node at (1,0.5) {$\eps = 0.05$};
    \node at (1, 0.1) {$u$};

    \node at (2,0) {%
      \includegraphics[width=0.25\linewidth]{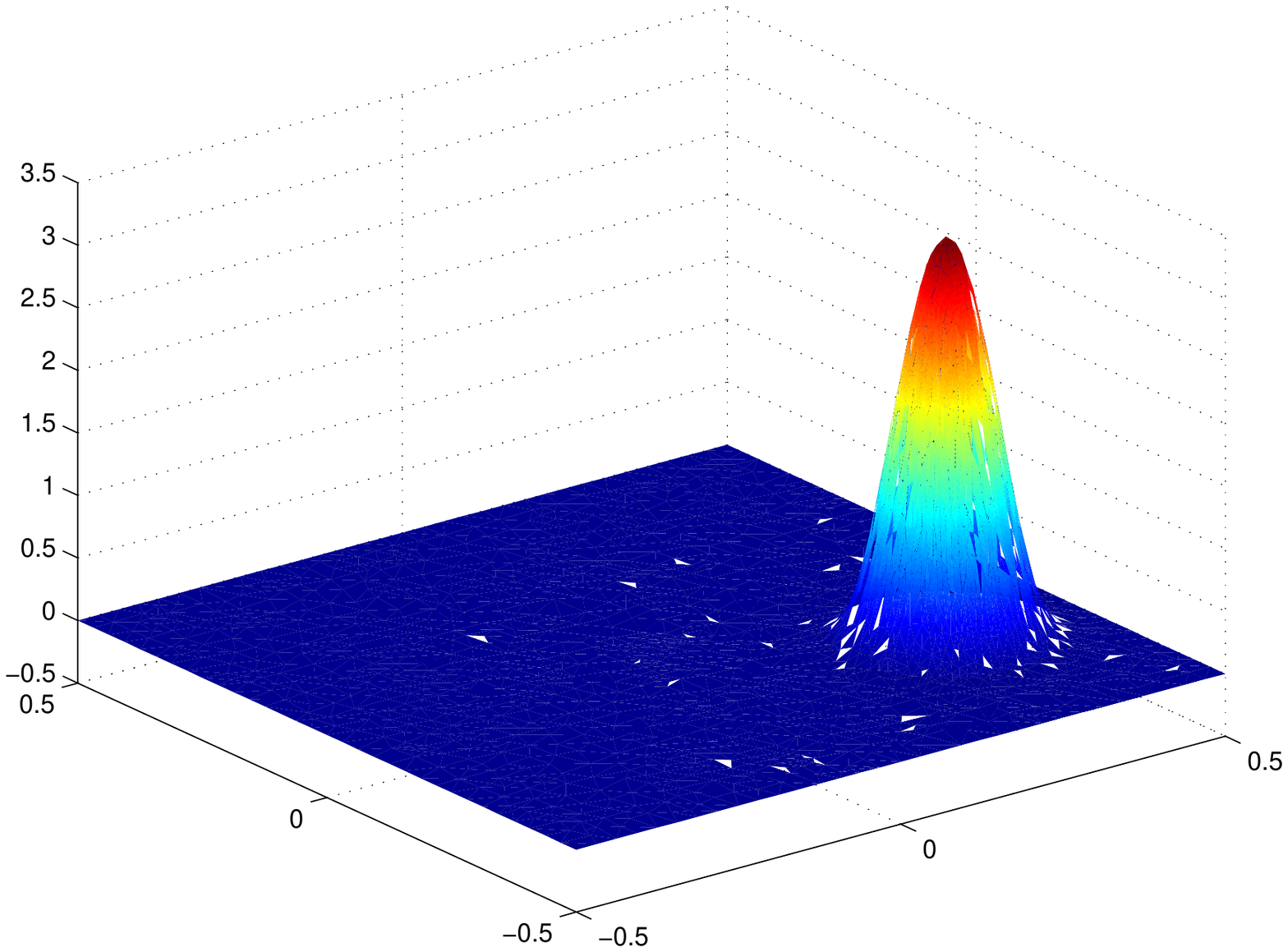}};
    \node at (2.1, 0.1) {$u$};
    \node at (2,0.5) {$\eps = 0.05$};

    \node at (3,0) {%
      \includegraphics[width=0.25\linewidth]{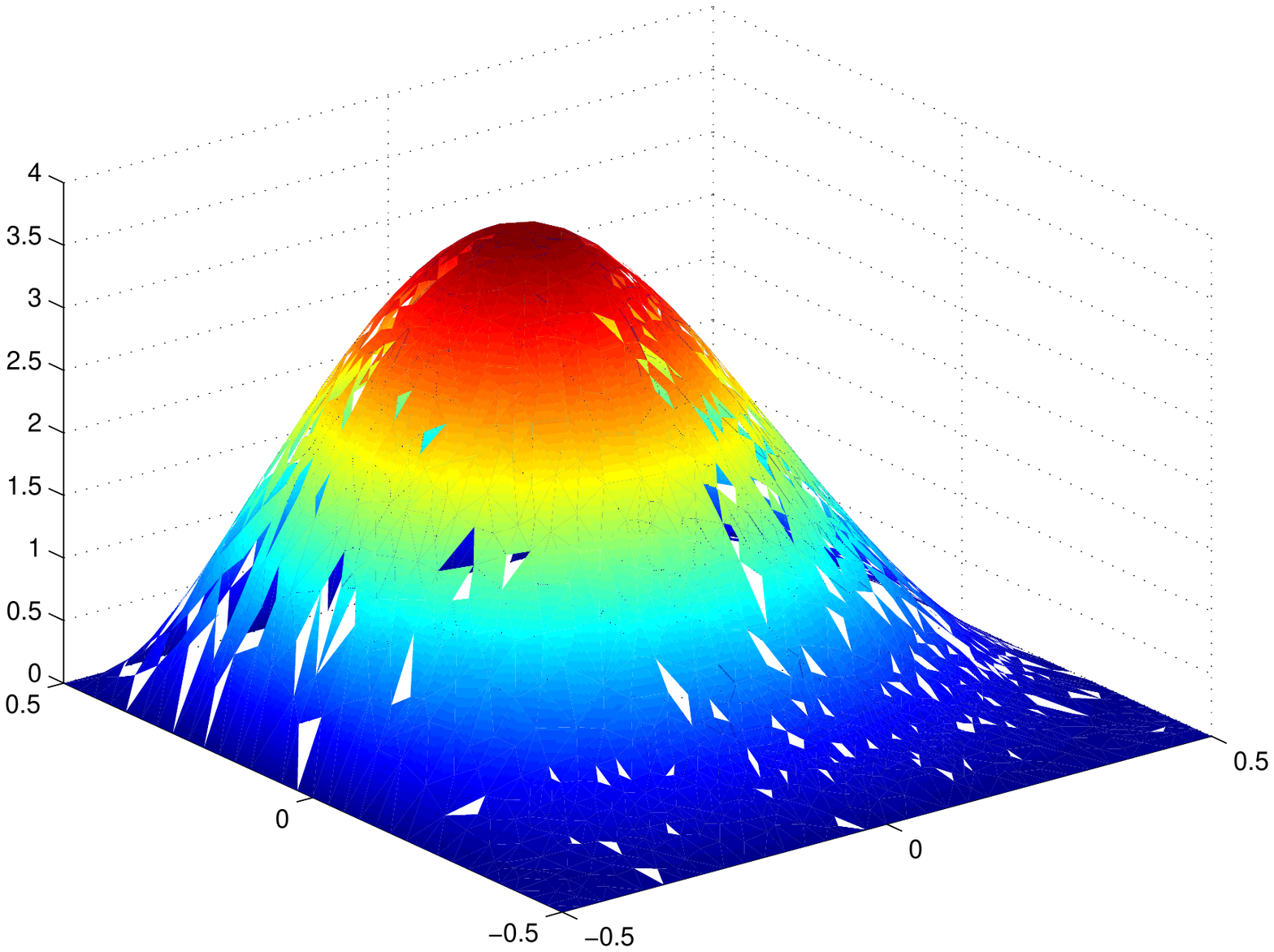}};
    \node[right] at (3.1, 0.1) {$u$};
    \node at (3,0.5) {$\eps = 0.25$};
  \end{tikzpicture}

  \vspace*{-10ex}%
  \caption{Potential $V$ and the corresponding sol.\ $u$
    to~\eqref{eq:concentration} on $(-0.5,0.5)^2$
    for $p=4$.}
  \label{fig:square,Vx,p=4}
\end{figure}

\section{Conclusion}

In this paper, we tackled the computation of ground state solutions
for a class of quasi-linear Schr\"o\-din\-ger equations which are
naturally associated with a non-smooth functional.  A change of
variable was used to overcome the lack of regularity and a mountain
pass algorithm was applied to the resulting functional to compute
saddle point solutions.

In the autonomous case, we outlined arguments and saw on
the above numerical computations that
the numerical solution $u_R$ on $\Omega_R$ converges to a radially symmetric
solution $u^*$ on the whole space as $R \to \infty$.
The existence of $u^*$ was proved by Colin and Jeanjean~\cite{cojean}.
The fact that the same numerical solution is found with many different
initial guesses suggests that the ground state solution on~$\R^N$ is
unique, a fact that was proved~\cite{gladiali} for $V$ large.
The numerical computations also suggest that the set of solutions
$u^*$ bifurcate from $0$ as $V \to 0$.

For the case of variable potential, the numerical computations
exhibited the existence of several solutions of mountain pass type
which are local minima of the functional $\mathcal{T}$ on the peak
selection set~$\Ran P$.  The asymptotic profile of these solutions
seem to be radial as it is expected from the theoretical results on
$\R^N$ for $\varepsilon \to 0$~\cite{Gloss2010,O-Severo2010}.
Interestingly, the multiplicity of positive solutions does not seem to
persist when $\varepsilon$ grows larger.  Whether the multiple
solutions come through a bifurcation from the branch of ground state
solutions as $\varepsilon$ goes to $0$ may be the subject of future
investigations.

\end{document}